\documentclass[english]{article}
\usepackage{graphicx}
\usepackage[english]{babel}
\usepackage[utf8]{inputenc}     
\usepackage{amssymb}
\usepackage{bbm}
\usepackage{amsmath}
\usepackage{amsthm}
\usepackage{cases}
\usepackage{xcolor}
\usepackage{tikz-cd}
\usepackage[colorlinks=true,linkcolor=blue,citecolor=blue,urlcolor=blue]{hyperref}
\usepackage{url}
\usepackage{accents}
\usepackage{biblatex} 
\newcommand{\la}{\mathcal {L}_{\in}} 
\newcommand{\R}{\mathbb {R}}

\newcommand{\F}{\mathcal{F}}

\newcommand{\lx}{L_\alpha [x]}
\newcommand{\lxg}{L_\alpha [x,G]}
\newcommand{\lr}{L_\alpha (\R^+)}

\newcommand{\wfc}{\text{wfc}}

\newcommand{\V}{\mathcal {V}}

\newcommand{\M}{\mathcal {M}}
\newcommand{\N}{\mathcal {N}}
\newcommand{\Q}{\mathcal {Q}}

\newcommand{\D}{\mathcal {D}}
\newcommand{\finite}{\text{fin}}

\newcommand{\mad}{\M^{\ad}}

\newcommand{\E}{\mathbb{E}} 							
  	
\newcommand{\fin}[1]{[#1]^{<\omega}}   				
\newcommand{\Ult}[2]{\mathrm{Ult}(#1,#2)}
\newcommand{\fu}[3]{\mathrm{Ult}_{#3}(#1,#2)}
\newcommand{\Hull}[2]{\mathrm{Hull}_{#2}^{#1}}
\newcommand{\cHull}[2]{\mathrm{cHull}_{#2}^{#1}}
\newcommand{\T}{\mathcal{T}}
\newcommand{\U}{\mathcal{U}}

\DeclareMathOperator{\crt}{crit}
\DeclareMathOperator{\card}{Card}
\DeclareMathOperator{\lh}{lh}
\DeclareMathOperator{\wfp}{wfp}
\DeclareMathOperator{\ran}{ran}

\DeclareMathOperator{\lgcd}{lgcd}

\DeclareMathOperator{\OR}{OR}
\DeclareMathOperator{\dom}{dom}

\DeclareMathOperator{\dir}{dir}
\DeclareMathOperator{\rk}{rk}
\DeclareMathOperator{\col}{Col}
\DeclareMathOperator{\ad}{ad}
\DeclareMathOperator{\kp}{KP}
\DeclareMathOperator{\hod}{HOD}
\DeclareMathOperator{\tho}{Th}
\DeclareMathOperator{\id}{id}

\DeclareMathOperator{\zf}{ZF}
\DeclareMathOperator{\zfc}{ZFC}
\DeclareMathOperator{\fml}{Fml}
\DeclareMathOperator{\tc}{tc}
\DeclareMathOperator{\HC}{HC}
\DeclareMathOperator{\od}{OD}
\DeclareMathOperator{\lv}{lv}

\newcommand{\prf}{\textsc{Proof.} }            

\newcommand{\eprf}{\hfill \qed \vspace*{5mm}}   

\theoremstyle{plain}
\newcounter{claimcounter}
\newcounter{subclaimcounter}
\newtheorem{theorem}{Theorem}[section]
\newtheorem{claim}[claimcounter]{Claim}
\newtheorem{subclaim}[subclaimcounter]{Subclaim}
\newtheorem{cor}[theorem]{Corollary}
\newtheorem{lem}[theorem]{Lemma}

\newcounter{casecounter}
\theoremstyle{remark}
\newtheorem{case}[casecounter]{Case}
\newcounter{casecounter a}
\newtheorem{case a}[casecounter a]{Case}
\newcounter{claimcounter a}
\newtheorem{claim a}[claimcounter a]{Claim}
\newtheorem*{claim*}{\protect\claimname}

\newtheorem*{notation*}{Notation}
\newtheorem*{theorem*}{Theorem}
\theoremstyle{remark}
\newtheorem{rem}[theorem]{Remark}

\newtheorem{dfn}[theorem]{Definition}

\usepackage{authblk}
\title{Analysis of HOD for Admissible Structures}
\author[1]{Jan Kruschewski}
\author[1]{Farmer Schlutzenberg}
\affil[1]{Institut für Diskrete Mathematik und Geometrie, TU Wien}
\date{}                     

\begin{document}
    
\maketitle

\begin{abstract}
    Let $n \geq 1$ and assume that there is a Woodin cardinal. For $x \in \R$ let $\alpha_x$ be the least $\beta$ such that 
    \[
        L_\beta [x] \models \Sigma_n \text{-} \kp + \exists \kappa (``\kappa \text{ is inaccessible and }\kappa^+ \text{ exists}").
    \]
    We adapt the analysis of $\text{HOD}^{L[x,G]}$ as a strategy mouse to $L_{\alpha_x}[x,G]$ for a cone of reals $x$. That is, we identify a mouse $\M^{\text{n-ad}}$ and define a class $H \subseteq L_{\alpha_x}[x,G]$ as a natural analogue of $\text{HOD}^{L[x,G]} \subseteq L[x,G]$, and  show that $H = M_\infty[\Sigma_0]$, where $M_\infty$ is an iterate of $\M^{\text{n-ad}}$ and $\Sigma_0$ a fragment of its iteration strategy. 
\end{abstract}

\tableofcontents

\section{Introduction}

In models of $\zfc$, the class HOD of hereditarily ordinal definable sets is an inner model of ZFC. However, unlike $L$, it is not absolutely definable, so that for example $\hod$ as computed in $\hod$ need not be equal to $\hod$. Moreover, unlike $L$, there is no general fine structural theory for $\hod$ which makes it difficult to show that combinatorial properties like $\Diamond$ or $\square$ hold in $\hod$.

One of the current key applications of inner model theory is to understand the $\hod$ of models of the Axiom of Determinacy. A general strategy often referred to as $\hod$-analysis has been developed to show that the $\hod$ of inner models of determinacy is a mouse together with fragments of its own iteration strategy, a structure often referred to as a $\hod$ mouse, and therefore a fine-structural model of $\zfc$.

The most basic example of such an analysis is the one of $\hod$ of $L[x,G]$, where $M_1^\sharp \leq_T x$ and $G$ is $(L[x],\col(\omega, {<\kappa_x}))$-generic, where $\kappa_x$ is the least inaccessible cardinal of $L[x]$, under the assumption of ${\undertilde{\Delta}^1_2}$-determinacy. In this context the determinacy hypothesis ensures that for every real $x$, $M_1^\sharp (x)$ exists and is $(\omega,\omega_1,\omega_1)$-iterable.

The goal of this paper is to adapt this technique to a context where the determinacy model, which is usually a model of ZF, is replaced by an admissible structure. In order to make this precise let us introduce the following definitions.
\begin{dfn}
    Let $\mathcal{L}_{\dot{\in}} = \{ \dot{=},\dot{\in} \}$ be the $\textit{language of set theory}$, short LST. Let $\mathcal{L}_{\dot{\in},\dot{E}} = \mathcal{L}_{\dot{\in}} \cup \{ \dot{E} \}$, where $\dot{E}$ is a predicate symbol, $\mathcal{L}_{\dot{\in},\dot{\R}} = \mathcal{L}_{\dot{\in}} \cup \{ \dot{\R} \}$, where $\dot{\R}$ is a constant symbol, and let $\mathcal{L}_{\mathrm{pm}}$ be the language of premice defined as in Definition 2.10 of \cite{Steel2010}.
\end{dfn}

\begin{dfn} \label{dfn: kp}
Let $\mathcal{L} \supseteq \mathcal{L}_{\dot{\in}}$ be an extension of $\mathcal{L}_{\dot{\in}}$ and let $k \geq 1$ be a natural number.
$\Sigma_k \textit{-Kripke-Platek set theory in the language } \mathcal{L}$, short $\Sigma_k \text{-}\kp_{\mathcal{L}}$, is the theory in the language $\mathcal{L}$ which consists of the following axioms of Extensionality, Pairing, Union, Infinity, and the following:
\begin{itemize}
    \item Foundation\footnote{Note that our notion of foundation deviates from the one in \cite{Barwise_2017} for the case $n=1$.}, i.e.~$\forall x (x \neq \emptyset \rightarrow \exists y(y \in x \land x \cap y = \emptyset))$,
    \item $\Delta_k$-Aussonderung, i.e.~letting for $\Sigma_k$ formulas $\varphi, \psi$ of the language $\mathcal{L}$, $\Phi_{\varphi,\psi}(\vec{x}) \equiv \forall z (\varphi(z,\vec{x}) \leftrightarrow \neg \psi(z,\vec{x}))$, we have for each pair of $\Sigma_k$ formulas $\varphi,\psi$ the axiom
    \[
        \forall v_1...\forall v_n [\Phi_{\varphi,\psi}(v_1,\dots,v_n) \rightarrow \forall a \exists b \forall x (x \in b \leftrightarrow x \in a \land \varphi(x,v_1,...,v_n))],
    \]

    \item $\Sigma_k$-Collection, i.e.~for all $\Sigma_k$ formulas $\varphi$ in the language $\mathcal{L}$,
    \begin{align*}
        \begin{split}
            \forall a \forall & v_1...\forall v_n [(\forall x  \in a \exists y \varphi(x,y,v_1,...,v_n)) \rightarrow  \\
            & (\exists b \forall x \in a \exists y \in b \varphi(x,y,v_1,...,v_n))].
        \end{split}
    \end{align*}
\end{itemize}
\end{dfn}

If $\mathcal{L}$ is clear from the context, we write $\Sigma_k \text{-}\kp$ instead of $\Sigma_k \text{-}\kp_{\mathcal{L}}$.

\begin{lem} \label{lem: axiomatizability}
    $\Sigma_k \text{-}\kp_{\mathcal{L}}$ is $\Pi_{k+2}$-axiomatizable, for $\mathcal{L} \supseteq \mathcal{L}_{\dot{\in}}$.
\end{lem}

\prf
    We argue by induction. In the case that $k=1$ this is clear. So suppose that $k > 1$ and that $\Sigma_{k-1}$-KP is $\Pi_{k+1}$-axiomatizable. It is easy to see that the scheme of $\Delta_k$-Aussonderung is $\Pi_{k+2}$-expressible. Thus, it suffices to see that we can express the $\Sigma_k$-Collection scheme in a $\Pi_{k+2}$-way over the theory $\Sigma_{k}$-KP. Let $\varphi \equiv \exists x_1 \forall x_2 \psi$ be $\Sigma_k$ in the language $\mathcal{L}$, where $\psi$ is $\Sigma_{k-2}$. Note that $\exists y \in b \exists x_1 \forall x_2 \psi$ is equivalent to $\exists x_1  \exists y \in b \forall x_2 \psi$. By $\Sigma_{k-1}$-KP, the formula $\exists y \in b \forall x_2 \psi$ is equivalent to a $\Sigma_{k-1}$ formula, so that $\exists y \in b \exists x_1 \forall x_2 \psi$ is equivalent to a $\Sigma_{k-1}$ formula over $\Sigma_{k-1}$-KP. But then it follows that every instance of the $\Sigma_k$-Collection scheme is $\Pi_{k+2}$ over $\Sigma_{k-1}$-KP.
\eprf

\begin{dfn} \label{def: base theory}
    For $n \geq 1$ let $\tho^\prime_n$ be the $\mathcal{L}_{\dot{\in},\dot{E}}$-theory consisting of the following statements:
\begin{itemize}
    \item $\Sigma_n \text{-}\kp_{{\mathcal{L}_{\dot{\in},\dot{E}}}}$,
    \item $V \text{=} L[\dot{E}]$\footnote{See Definition 18 of \cite{KRUSCHEWSKI_SCHLUTZENBERG_2025} for a definition.}, and
    \item $\exists \kappa \exists \delta (``\delta \text{ is Woodin}" \land ``\kappa \text{ is inaccessible}" \land~\kappa > \delta~\land ``\kappa^+ \text{ exists}")$,
\end{itemize}

and let $\tho_n$ be the $\mathcal{L}_{\dot{\in},\dot{E}}$-theory which consists of $\tho_n^\prime$ and the statement
\[
    \forall \alpha(L_\alpha[\dot{E}] \not \models \tho^\prime_n).
\]
\end{dfn}

\begin{dfn}
    For $n \geq 1$ let $\M^{n\text{-ad}}$ be the minimal $(n+1)$-sound premouse which models $\tho_n$ and is $(n, \omega_1, \omega_1+1)^*$-iterable.\footnote{See the paragraph before Corollary 1.10 in \cite{Steel_2002} for the definition of $(n, \omega_1, \omega_1+1)^*$-iterability.}

    Let $\delta^{\M^{n\text{-ad}}}$ be the unique Woodin cardinal of $\M^{n\text{-ad}}$ and $\kappa^{\M^{n\text{-ad}}}$ be the unique inaccessible cardinal of $\M^{n\text{-ad}}$ which is greater than $\delta^{\M^{n\text{-ad}}}$.

    Let $\Sigma^{\M^{n\text{-ad}}}$ be an $(n, \omega_1, \omega_1+1)^*$-iteration strategy for $\M^{n\text{-ad}}$.
\end{dfn}

At the end of Section \ref{higher admissible premice} we will show that $\M^{\text{n-ad}}$ exists assuming that there is a Woodin cardinal. Moreover, we will show that $\rho_{n+1}^{\M^{n\text{-ad}}} = \omega$.

We will now fix $n \geq 1$ until the end of the paper and refer to $\tho_n$, $\tho^\prime_n$, $\Sigma^{\M^{n\text{-ad}}}$, and $\M^{n\text{-ad}}$ simply as $\tho$, $\tho^\prime$, $\Sigma^{\mad}$, and $\mad$.

\begin{dfn} \label{after fixing n}
    For $x \in \R$ let $\alpha_x$ be the least $\beta$ such that 
    \[
        L_\beta [x] \models \Sigma_n \text{-} \kp~+~\exists \kappa (``\kappa \text{ is inaccessible and }\kappa^+ \text{ exists}").
    \]
\end{dfn}

Fix $x \in \R$ such that $\mad \leq_T x$ and let $\alpha = \alpha_x$. We will denote by $\kappa$ the unique inaccessible cardinal of $\lx$. Since $\M^{\ad}$ is recursive in $x$, $\mad \in \lx$ and $\OR^{\mad} < \omega_1^{\lx} < \kappa$.

The idea is now to replace $L[x,G]$ from the classical $\hod$ analysis with $\lxg$, where $G$ is  $(\lxg,\col(\omega, {<\kappa}))$-generic, and find an appropriate version of ``the $\hod$ of $\lxg$" which takes the role of $\hod$ in the classical analysis. The appropriate version of ``the $\hod$ of $\lxg$" is given by the following definition.

\begin{dfn}
    For $G$ which is $(\lxg,\col(\omega, {<\kappa}))$-generic and $X \in \lxg$ let  $\Sigma_n \text{-}\od_{\{X\}}^{\lxg}$ be the class of all $y \in \lxg$ which are ordinal definable over $\lxg$ from the parameter $\{X\}$ via a $\Sigma_n$ formula in the language $\mathcal{L}_{\dot{\in}}$, i.e.~there is a $\Sigma_n$ formula $\varphi$ in the language $\mathcal{L}_{\dot{\in}}$ and ordinals $\alpha_1,...,\alpha_m < \alpha$ such that for all $z \in \lxg$,
    \[
        z \in y \iff \lxg \models \varphi (z,\alpha_1,...,\alpha_m,X).
    \]
    Let
    \[
        \Sigma_n \text{-} \hod_{\{X\}}^{\lxg} = \{ y \colon \tc (\{ y \}) \subset \Sigma_n \text{-OD}_{\{X\}}^{\lxg}    \}.
    \]
    We write $\Sigma_n \text{-} \hod^{\lxg}$ for $\Sigma_n \text{-} \hod_{\{\emptyset\}}^{\lxg}$.
\end{dfn}

\begin{rem}
    Note that in the definition of $\Sigma_n \text{-}\od_{\{X\}}^{\lxg}$, we could also equivalently require that $y \in \Sigma_n \text{-}\od_{\{X\}^{\lxg}}$ iff $\{ y \}$ is ordinal definable over $\lxg$ from the parameter $\{ X\}$ for $y \in \lxg$, i.e.~there is a $\Sigma_n$ formula $\varphi$ in the language $\mathcal{L}_{\dot{\in}}$ and ordinals $\alpha_1,...,\alpha_m < \alpha$ such that for all $z \in \lxg$,
    \[
        z = y \iff \lxg \models \varphi (z,\alpha_1,...,\alpha_m,X).
    \]
    This is equivalent to our definition, since $\lxg$ is a model of $\Sigma_n$-Collection.
\end{rem}

We will then be able to show the following theorem.

\begin{theorem}
    Let $G$ be $(\lxg,\col(\omega, {<\kappa}))$-generic. There is a countable $\Sigma^{\mad}$-iterate $M_\infty$ of $\mad$ such that there is a fragment $\Sigma_0$ of the tail strategy of $\Sigma^{\mad}$ given by $M_\infty$ such that:
    \begin{enumerate}
        \item $\Sigma_n \text{-} \hod_{\{ \R^{\lxg} \}}^{\lxg} = M_\infty [\Sigma_0]$,
        \item $\Sigma_n \text{-} \hod_{\{ \R^{\lxg} \}}^{\lxg} \models \Sigma_n \text{-KP} \land \exists \delta (``\delta \text{ is Woodin}")$, and
        \item $\Sigma_n \text{-} \hod_{\{ \R^{\lxg} \}}^{\lxg}$ is a forcing ground of $\lxg$.
    \end{enumerate}
\end{theorem}

\begin{rem}
    In the case that $n \geq 2$, we have that $\{\R^{\lxg}\}$ is $\Sigma_n$-definable, so that $\Sigma_n \text{-} \hod_{\{ \R^{\lxg} \}}^{\lxg} = \Sigma_n \text{-} \hod^{\lxg}$.
\end{rem}

In the analysis of $\hod^{L[x,G]}$ it is convenient to have a proper class of fixed points, and in that context these can be taken to be the class of Silver indiscernibles (at least, in the argument from the assumption that $\M_1^{\sharp}$ exists).

In this paper, we will also isolate a class $S_\infty$ of fixed points, rather analogous to the Silver indiscernibles for $\M_1^{\sharp}$ (though they will not be model-theoretic indiscernibles). They will be similarly convenient for the analysis.

There might be ways to avoid the use of $S_\infty$ for the analysis of $\Sigma_n \text{-} \hod$, and hence the work needed to establish its existence. However, the role of the fixed points figures more prominently in the analysis of Varsovian models \cite{sargsyan2021varsovianmodelsii}, \cite{schlutzenberg2022varsovianmodelsomega}, so apart from the convenience of having $S_\infty$ at our disposal, it might also help towards generalizing the analysis carried out here to other admissible contexts. Apart from this, the construction of $S_\infty$ may be of independent interest and have other applications. It involves an analysis of a tree $T$ searching for an illfounded structure whose wellfounded part is $L_{\alpha}(\R^{\lxg})$.
See Section \ref{section: finding the sequence} for the construction of the fixed points and Section \ref{section: generating fixed points} for more details on their use in the analysis.

\section{$\Sigma_k$-admissible premice} \label{higher admissible premice}

In this section, we will investigate the fine structure of passive premice, which model the theory $\Sigma_k$-KP for some $k \geq 1$.

\begin{dfn}
    Let $\M$ be a passive premouse, and $k \geq 1$. We say that $\M$ is $\Sigma_k$\textit{-admissible} if $\M \models \Sigma_k \text{-}\kp_{\mathcal{L}_{\text{pm}}}$.

    We say that $\M$ is an admissible premouse if $\M \models \Sigma_k\text{-}\kp$ for some $k \geq 1$.
\end{dfn}

\begin{rem}
    We restrict our attention to passive premice, since an active premouse cannot model $\Sigma_k$-KP with the active extender as a predicate and without the active extender it is a model of $\zf^-$, so trivially of $\Sigma_k$-KP.

    Note that since we are dealing with passive premice we do not need to consider the complexities of $r\Sigma_m$ formulas of active premice which arise in \cite{Mitchell_Steel_2017}. Moreover, since an admissible premouse is passive, the notions of a premouse and its $\Sigma_0$ code coincide. We will use these facts throughout the paper without further mention.
\end{rem}

The following definition is a special case of Definition 5.2 in \cite{Schlutzenberg2019TheDO}. It will allow us to define an ordering on the $r\Sigma_k$ theory of an admissible premice which will be useful in the fine structural computations and is important for the pruning process described in Lemma \ref{lem: T is illfounded} in Section \ref{section: finding the sequence}.

\begin{dfn} \label{def: minimal skolem terms}
    Let $\varphi$ be a $r\Sigma_{k+1}$ formula of $l+1$ many free variables. The \textit{minimal Skolem term associated with} $\varphi$ is denoted $m\tau_\varphi$ and has $l$ free variables. 

    Let $R$ be a passive $k$-sound premouse with $\rho^R_k > \omega$. We define the partial functions
    \[
        m\tau^R_{\varphi} \colon R^l \to R,
    \]
    and 
    \[
        \lv^R_\varphi \colon R^l \to \OR^R.
    \]
    If $k = 0$ then $m\tau^R_{\varphi}$ is just the usual Skolem function associated with $\varphi$ such that the graph of $m\tau^R_{\varphi}$ is uniformly $r\Sigma^\M_1$, and let $\lv_\varphi^R(\vec{x})$ be the least $\beta$ such that $R \vert \alpha \models \exists y \varphi(y,\vec{x})$, if it exists. Otherwise, let $\lv_\varphi^R(\vec{x})$ be undefined.

    Suppose $k > 0$. Let $\vec{x} \in R^l$. If $R \models \neg \exists y \varphi(\vec{x},y)$, then $m\tau^R_{\varphi}(\vec{x})$ and $\lv_\varphi^R(\vec{x})$ are undefined. Suppose that $R \models \exists y \varphi(\vec{x},y)$. Let $\tau$ be the basic Skolem term associated with $\varphi$ (see \cite[2.3.3]{Mitchell_Steel_2017}). For $\beta < \rho^\M_k$, let $(\tau_\varphi)^\beta$ be defined as in the proof of \cite{Mitchell_Steel_2017}[2.10], with $q = \vec{p}_k^R$. Let $\beta_0$ be the least $\beta$ such that $(\tau_\varphi)^\beta(\vec{x})$ is defined and set
    \[
        m\tau^R_{\varphi}(\vec{x}) = (\tau_\varphi)^\beta(\vec{x}),
    \]
    and
    \[
        \lv^R_\varphi(\vec{x}) = \beta_0.
    \]
\end{dfn}

\begin{lem}
    For $R$ as in Definition \ref{def: minimal skolem terms} the graph of $m\tau_\varphi^R$ is $r\Sigma_{k+1}^R(\{\vec{p}^R_k\})$, recursively uniformly in $R,\varphi, \vec{p}_k^R$ for $R$.
\end{lem}

The following is a special case of Lemma 5.4 in \cite{Schlutzenberg2019TheDO}.

\begin{lem}
    Let $R$ be as in Definition \ref{def: minimal skolem terms}, and $X \subset R$. Then $\Hull{R}{k+1}(X \cup \{ \vec{p}_k^R \}) = \{ m\tau_\varphi^R (\vec{x}) : \varphi \text{ is } r\Sigma_{k+1} \land \vec{x} \in \fin{X} \}$.
\end{lem}

\begin{dfn} \label{dfn: ordering of rsigma_n}
Let $\varphi$ and $\psi$ be $r\Sigma_{k+1}$ formulas of $l+1 < \omega$ many free variables. Let $R$ be a passive $k$-sound premouse such that $\rho^R_k > \omega$. Let $\vec{x},\vec{y} \in  R^l$. Let $R \models \varphi(\vec{x}) \leq^* \psi(\vec{y})$ if and only if $R \models \exists z \varphi(\vec{x},y)$ and, if $R \models \exists z \psi(\vec{y},z)$, then $\lv^{R}_\varphi (\vec{x}) \leq \lv^{R}_\psi (\vec{y})$.
\end{dfn}

\begin{lem}
    Let $k < \omega$ and $R$ be a passive $k$-sound premouse. Then the relation $(\leq_{k+1}^*)^{R}$ is $r\Sigma_{k+1}^R (\{\vec{p}_k^R \})$ uniformly in $R$.
\end{lem}

Since the sort of collection which holds in admissible premice by definition is expressed in terms of the standard $\Sigma_k$-hierarchy of formulas, and for the fine structural computations we use the $r\Sigma_k$-hierarchy, we are now interested in the relationship between these two hierarchies.
The following lemma explains the relationship between the two hierarchies in general.

\begin{lem} \label{translatibility}
    Let $k \geq 1$ and let $\M$ be a passive $(k-1)$-sound premouse. Then there are $p \in \M$ and recursive functions $f_1$ and $f_2$ such that for every $r\Sigma_k$ formula $\varphi$ and $x \in \M$,
    \[
        \M \models \varphi(x) \iff \M \models f_1(\varphi)(x,p),
    \]
    and $f_1(\varphi)$ is $\Sigma_k$, and for every $\Sigma_k$ formula $\varphi$ and $x \in \M$,
    \[
        \M \models \varphi(x) \iff \M \models f_2(\varphi)(x),
    \]
    and $f_2(\varphi)$ is $r\Sigma_k$.
\end{lem}

In a certain context, which we will later work in, we can improve this by eliminating the parameter $p$, as the following lemma shows.

\begin{lem} \label{better translatibility}
    Let $\M$ be a passive premouse with a largest cardinal $\delta$ and $k \geq 1$. Suppose that $\rho_{k-1}^\M = \OR^\M$. Then there are recursive functions $f_1^k$ and $f_2^k$ such that for any $\Sigma_k$ formula $\varphi$, $f_1^k (\varphi)$ is an $r\Sigma_k$ formula such that for all $x \in \M$,
    \[
        \M \models \varphi(x) \iff \M \models f_1^k(\varphi)(x),
    \] 
    and for any $r\Sigma_k$ formula $\psi$, $f_2^k (\psi)$ is a $\Sigma_k$ formula such that for all $x \in \M$,
    \[
        \M \models \psi(x) \iff \M \models f_2^k(\psi)(x).
    \] 
\end{lem}

\prf
We argue by induction on $k$. In the case $k=1$, there is nothing to show. So suppose $k >1$ and there are $f_1^{k-1}$ and $f_2^{k-1}$ as in the lemma. The existence of the function $f_1^{k}$ is well known.
Before we describe the function $f_2^{k}$, note that for $\alpha \in (\delta, \OR^\M)$ and $q \in \M$, $H(\alpha,q) := \Hull{\M}{k-1}(\alpha \cup \{q\})$ is transitive and bounded in $\M$. The transitivity easily follows from the fact that $\delta$ is the largest cardinal of $\M$. Suppose for the sake of contradiction that $H(\alpha,q)$ is unbounded in $\M$ for some $\alpha \in (\delta, \OR^\M)$ and $q \in \M$. Then, by transitivity, $\M = H(\alpha,p)$. However, this means that $\rho_{k-1}^\M < \OR^\M$, a contradiction!

Let us now describe the function $f_2^{k}$. Let
\[
    S := \{ \gamma < \OR^\M : \M \vert \gamma \prec_{\Sigma_{k-1}} \M \}
\]
and note that $S$ is cofinal in $\OR^\M$ and $\Pi_{k-1}$-definable over $\M$.
For an $r\Sigma_k$ formula $\psi(x) \equiv \exists \alpha \exists q \exists t (t = \tho_{k-1} (\alpha \cup \{ q \}) \land \varphi(\alpha,q,t,x))$, where $\varphi$ is $\Sigma_1$, let $f_2^{k} (\psi(x)) = \exists t \exists \alpha \exists \beta  \exists q (\alpha < \beta \land q \in \M \vert \beta \land \beta \in S \land t = \tho_{k-1}^{\M \vert \beta} (\alpha \cup q) \land \varphi(\alpha,q,t,x))$. Note that for $\beta \in S$ by the induction hypothesis, $\M \vert \beta \prec_{r\Sigma_{k-1}} \M$. Then, using the fact that $S$ is cofinal in $\OR^\M$, it is easy to see that for all $x \in \M$, $\M \models \psi(x) \iff \M \models f_2^{k}(\psi)(x)$.
\eprf

\begin{lem} \label{lem: n-1 projectum}
    Let $k \geq 1$ and $\M$ be a passive premouse that models $\Sigma_k$-Collection and has a largest cardinal. Then $\rho_{k-1}^\M = \OR^{\M}$.
\end{lem}

\prf
    Let us argue by induction on $k \geq 1$. In the case $k=1$, there is nothing to show. So suppose that $k > 1$ and $\M$ models $\Sigma_{k}$-Collection. By the induction hypothesis, we may assume that $\rho^\M_{k-2} = \OR^\M$ so that $\vec{p}_{k-2}^\M = \emptyset$ and for all $\alpha < \OR^M$, $\tho_{k-2}^\M (\alpha \cup \{ \vec{p}_{k-2}^\M \}) \in \M$. Since $\rho_{k-2}^\M = \OR^\M$ and $\M$ has a largest cardinal, it follows that 
    \[
        S := \{ \gamma < \OR^\M : \M \vert \gamma \prec_{\Sigma_{k-2}} \M \}
    \]
is cofinal in $\OR^\M$. Moreover, $S$ is $\Pi_{k-2}$-definable over $\M$.  We aim to see that $\rho^\M_{k-1} = \OR^\M$. Suppose for the sake of contradiction that $ \kappa := \rho^\M_{k-1} < \OR^\M$.

    \begin{case} $\kappa = \omega$. For $m< \omega$, let $F_m$ be the set of $r\Sigma_{k-1}$ formulas with parameter $p_{k-1}^\M$ of length less or equal than $m$. Note that 
    \[
        \M \models \forall m < \omega \exists \gamma (\gamma \in S \land \forall \varphi \in F_m ( \varphi \rightarrow \M \vert \gamma \models \varphi)).
    \]
    Moreover, since $\forall \varphi \in F_m ( \varphi \rightarrow \M \vert \gamma \models \varphi)$ is an instance of the $\Sigma_{k}$-Collection scheme, it follows by another application of $\Sigma_{k}$-Collection that there is some $\gamma < \OR^\M$ such that $\M \vert \gamma \models \tho^\M_{k-1}(\omega \cup \{ p_{k-1}^\M \})$, a contradiction!
    \end{case}

    \begin{case}
        $\kappa > \omega$.
            Let $f \colon \kappa \to \M$ be such that $f(\alpha) = \tho_{k-1}^\M (\alpha \cup \{ p_{k-1}^\M \})$. Since $\tho_{k-1}^\M (\alpha \cup  \{ p_{k-1}^\M  \}) \in \M$, it follows from $\Sigma_k$-Collection that there is some $\beta \in S$ such that $\tho_{k-1}^\M (\alpha \cup \{ p_{k-1}^\M\}) = \tho_{k-1}^{\M \vert \beta} (\alpha \cup \{ p_{k-1}^\M \} )$. Let $F_\alpha$ be the set of $r\Sigma_k$ formulas with parameters in $\alpha \cup \{ p_{k-1}^\M \}$ and let $h$ be as in Lemma \ref{translatibility}. We have that $x = \tho_{k-1}^\M (\alpha \cup \{ p_{k-1}^\M \})$ if and only if
        \begin{equation} \label{complexity of f}
            \exists \beta (\beta \in S \land x = \tho_{k-1}^{\M\vert \beta} (\alpha \cup \{ p_{k-1}^\M\})) \land \forall \varphi \in F_\alpha (\varphi \rightarrow \M \vert \beta \models h(\varphi)).
        \end{equation}
        Line (\ref{complexity of f}) is $\Sigma_k$. Thus, $f$ is $\Sigma_k$-definable over $\M$ from parameters. By $\Sigma_{k}$-Collection, there is $\xi<\alpha$, such that for all $\alpha < \kappa$, $\tho_{k-1}^\M (\alpha \cup \{ \vec{p}_{k-1}^\M \}) \in \M \vert \xi$. But this means that $\tho_{k}^\M (\kappa \cup \{ \vec{p}_{k-1}^\M \}) \in \M$, a contradiction!
    \end{case}
    This finishes the proof.
\eprf

\begin{lem} \label{lem: largest cardinal}
    Let $k \geq 1$ and let $\M$ be a passive premouse that models $\Sigma_k$-Collection and has a largest cardinal. Suppose that $\rho_k^\M < \OR^\M$. Then $\rho_k^\M$ is the largest cardinal of $\M$.
\end{lem}

\prf
    By Lemma \ref{lem: n-1 projectum}, $\rho_{k-1}^\M = \OR^\M$. In particular, $\vec{p}_{k-1}^\M = \emptyset$. Let $\rho := \rho_k^\M$ and $p = p_k^\M$. Suppose for the sake of contradiction that $\rho^+$ exists in $\M$. Let
    \[
        H := \Hull{\M}{k} (\rho \cup \{ p, \rho \})
    \]
    and set $\xi := \sup(H \cap \rho^{+\M}) = H \cap \rho^{+\M} < \OR^\M$. 
    Note that
    \[
        \M \models \forall \gamma < \xi \exists \vec{x} \in \fin{\rho} \exists \varphi (\gamma = m\tau_{\varphi}^\M (\vec{x},p)).
    \]
    However, this means that
    \begin{equation} \label{another bounding statement}
        \M \models \forall \gamma < \xi \exists \beta \exists \vec{x} \in \fin{\rho} \exists \varphi (\lv_{m\tau_{\varphi}^\M (\vec{x},p)) = \gamma}(\vec{x},p,\gamma)=\beta.
    \end{equation}
    But the part in parentheses of line (\ref{another bounding statement}) is $r\Sigma_k$, so that by Lemma \ref{translatibility} and Lemma \ref{lem: n-1 projectum} we may assume that it is $\Sigma_k$ with parameters from $\M$. But then by $\Sigma_k$-Collection, there is a uniform $\beta$. Thus, $\M$ can compute from $T_\beta := \tho_{k-1}^M (\beta) \in \M$ a surjection $f \colon \fin{\rho} \to \xi$ such that $f \in \M$. In the case that $\xi = \rho^{+\M}$ this gives immediately a contradiction. So, suppose that $\xi < \rho^{+\M}$. Note that for $\gamma < \xi$, $\vec{x} \in \fin{\rho}$, and an $r\Sigma_k$ formula $\varphi$, there is a subtheory of $T_\beta$, which witnesses the statement $(m\tau_{\varphi}^\M (\vec{x},p)) = \gamma$. This subtheory is recursively definable from the parameters $\gamma, \vec{x}$, and $\rho$. Since $\Hull{\M}{k} (\rho \cup \{ p, \rho \}) \subset H$, $H$ is unbounded in $\M$. Thus, we may assume that the bound $\beta$ is in $H$. However, since $\rho \in H$, $f \in H$ and therefore $\xi \in H$, a contradiction!
\eprf

What we have shown so far gives the following criteria for when a passive premouse is $\Sigma_k$-admissible.

\begin{lem} \label{lem:KP in Premice}
    Let $\M$ be a passive premouse that has a largest cardinal, and $k \geq 1$. Then the following statements are equivalent:
    \begin{enumerate}
        \item \label{kp} $\M \models \Sigma_k \text{-}\kp$,
        \item $\M \models \Sigma_k\text{-Collection}$, and
        \item \label{no unbounded} there is no total unbounded function $f \colon \alpha \rightarrow \lfloor \M \rfloor$ such that $\alpha \in \OR^{\M}$ and $f \in \Sigma_k^\M (\M)$. 
    \end{enumerate}
\end{lem}

\prf
    It suffices to see that \ref{no unbounded} implies \ref{kp}. The only axiom of $\Sigma_k$-KP that is not clear is $\Delta_k$-Aussonderung. By Lemma \ref{lem: n-1 projectum}, $\rho_{k-1}^\M = \OR^\M$. Let
    \[
        S:= \{ \gamma < \alpha : \M \vert \gamma \prec_{\Sigma_{k-1}} \M \}
    \]
    and note that $S$ is cofinal in $\OR^\M$. Let $x \in \M$ and suppose that $\varphi$ and $\psi$ are $\Sigma_{k}$ formulas such that $\M \models \forall z (\varphi(z) \leftrightarrow \neg \psi(z))$.
    We have $\M \models \forall y \in x \exists \gamma (\M \vert \gamma \models \varphi(y) \lor \psi(\gamma))$. By $\Sigma_k$-Collection, there is $\gamma < \OR^\M$ which works uniformly. Since $S$ is cofinal in $\OR^\M$, we may assume that $\gamma \in S$. However, then $\{ y \in x : \M \models \varphi(y) \}$ is definable over $\M \vert \gamma$.
\eprf

\begin{dfn}
    Let $\M$ be a passive premouse, and $k \leq \omega$. We let
    \[
        S^\M_k := \{ \alpha < \OR^\M : \M \vert \alpha \prec_{\Sigma_k} \M \}.
    \]
\end{dfn}

We have used the following observation before, but let us record it as a lemma.

\begin{lem}
    Let $k \geq 1$ and let $\M$ be a passive premouse that models $\Sigma_k$-Collection and has a largest cardinal. Then $S_{k-1}^\M$ is cofinal in $\OR^\M$. Moreover, $S_{k-1}^\M$ is $\Pi_{k-1}$-definable over $\M$.
\end{lem}

The proof of the following lemma is a straightforward adaption of the proof of Lemma 15 in \cite{KRUSCHEWSKI_SCHLUTZENBERG_2025}, which we leave as an exercise to the reader.

\begin{lem} \label{lem:General Preservation of KP}
    Let $m \leq \omega$ and let $\N$ be a $m$-sound premouse. Let $\mathcal{T}$ be a $m$-maximal iteration tree on $\N$ such that $\lh(\mathcal{T})=\theta+1$. Let $b:=[0,\theta]^{\T}$ be the main branch of $\T$ and $\alpha$ be least such that $\alpha + 1 \in b$ and $(\alpha+1,\theta]^{\T}$ does not drop in model. Let $\eta = \alpha+1$. Then $\M_{\eta}^{*\T}$ is a $\Sigma_k$-admissible passive premouse with a largest, regular, and uncountable cardinal if and only if $\M_{\theta}^{\T}$ is a $\Sigma_k$-admissible passive premouse with a largest, regular, and uncountable cardinal.
\end{lem}

The following lemma collects some fine structural consequences of the theory $\tho_k$.

\begin{lem} \label{lem: standard parameter of admissible mice}
    Let $k \geq 1$ and $\M$ be a passive premouse that models $\tho_k$.
    For $i < k$, $\rho_i^\M = \OR^\M$ and $p_i^\M = \emptyset$. Moreover, $\rho_k^\M = \theta < \OR^\M$, where $\theta$ is the largest cardinal of $\M$. 
    
    If $\M$ is $k$-sound and $(k,\omega_1,\omega_1 +1)^*$-iterable, then, if $k=1$, $p_k^\M = \{ \theta \}$, and if $k \geq 2$, then $p_k^\M = \emptyset$. Moreover, $\rho_{k+1}^\M = \omega$ and $p_{k+1}^{\M} = \emptyset$. 
    
    If $\M = \M^{k\text{-ad}}$, then $\Hull{\M}{k+1}(\emptyset) = \M$.
\end{lem}

\prf
    Let $\theta$ be the largest cardinal of $\M$. By Lemma \ref{lem: n-1 projectum}, for $i < k$ $\rho_k^\M = \OR^\M$ and $p_i^{\M} = \emptyset$.

    Let us show that $\rho_k^{\M} < \OR^{\M}$. We will show that $H:= \Hull{\M}{k}(\theta + 1) = \M$. Suppose not, then $H \unlhd \M$, since $H$ is transitive and $\theta$ is the largest cardinal of $\M$. We want to show that $H \models \Sigma_k\text{-Collection}$, so that $H \models \Sigma_k\text{-KP}$, which would be contradicting the fact that $\M \models \tho_k$. Suppose for the sake of contradiction that $\Sigma_k$-Collection fails in $H$. Let $\varphi$ be a $\Pi_{k-1}$ formula and $p \in H$ be such that
    \[
        H \models \forall \alpha < \theta \exists y \varphi(\alpha,y,p),
    \]
    but there is no $z \in H$ such that 
    \[
        H \models \forall \alpha < \theta \exists y \in z \varphi(\alpha,y,p).
    \]
    Note that we may assume that the failure is of this form, since $\theta$ is the largest cardinal of $H$ and $H$ is a premouse. Note that
    \[
        \M \models \forall \alpha < \theta \exists y \varphi(\alpha,y,p),
    \]
    and so
    \[
        \M \models \exists z \forall \alpha < \theta \exists y \in z \varphi(\alpha,y,p),
    \]
    as witnessed by $H$. However, since $\rho_{k-1}^\M = \OR^\M$, this means that
    \[
        \psi := \exists \beta (\beta \in S_{k-1} \land (\forall \alpha < \theta \exists y \in \M \vert \beta (\M \vert \beta \models \varphi(\alpha,y,p))))
    \]
    holds in $\M$. Note that $\psi$ is $\Sigma_k$. Thus, $H \models \psi$. However, this means that $H \models \forall \alpha < \theta \exists y \in (H \vert \beta) \varphi(\alpha,y,p)$, which is a contradiction! Thus, $H \models \Sigma_k\text{-Collection}$, and therefore $H = \M$ and $\rho_k^{\M} < \OR^\M$.

    Let us suppose for the rest of the proof that $\M$ is $k$-sound and $(k,\omega_1,\omega_1+1)^*$-iterable.
    Next, we aim to determine value of $p_k^\M$. Let us first consider the case $k=1$. 
    By Lemma \ref{lem: largest cardinal}, $\theta = \rho_1^{\M}$. 
    By condensation, $\M \vert \theta \prec_{\Sigma_1} \M$ and thus $\Hull{\M}{1}(\theta) = \M \vert \theta$. But then, as in the first part of the proof, $H := \Hull{\M}{1}(\theta + 1) = \M$, so that $p_1^{\M} = \{ \theta \}$.
    Let us now consider the case of $k \geq 2$. Again, by Lemma, \ref{lem: largest cardinal}, $\rho_{k}^\M = \theta$. Note that $\{ \theta \}$ is $r\Sigma_2$ definable over $\M$ as the unique cardinal that has a Woodin and an inaccessible cardinal below and is the successor of the inaccessible cardinal. Thus, $H := \Hull{\M}{k}(\theta) = \Hull{\M}{k}(\theta + 1)$. As before, $H = \M$, so that $p_k^\M = \{ \emptyset \}$.

    We now show $\rho_{k+1}^\M = \omega$. We claim $\cHull{\M}{k+1}(\emptyset) = \M^{k\text{-ad}}$. From this it follows that $\rho_{k+1}^\M = \omega$, as otherwise by condensation $\M^{k\text{-ad}} \lhd \M$ which would be a contradiction to the fact that $\M \models \tho_k$. Recall that by Lemma \ref{lem: axiomatizability}, $\tho_k$ has an $r\Pi_{k+2}$ axiomatization. However, this means that $H:= \Hull{\M}{k+1}(\emptyset) \models \tho_k$.
    Note that $\vec{p}_k^\M \in H$, since $\{\theta\}$ is $r\Sigma_2$-definable by the above argument. Thus, the transitive collapse of $H$ is sound, which means that $\cHull{\M}{k+1}(\emptyset) = \M^{n\text{-ad}}$. 
    
    The claim that $\Hull{\M}{k+1}(\emptyset) = \M$ now follows by the same argument.
\eprf

By Lemma \ref{lem: standard parameter of admissible mice}, we immediately have the following.

\begin{cor} \label{cor about mad}
    Suppose that there is an $n$-sound premouse $\M$ that models $\tho$ and is $(n, \omega_1, \omega_1+1)^*$-iterable. Then, $\mad$ exists and $\mad = \mathfrak{C}_{n+1}(\M)$. Therefore, if there is a Woodin cardinal, $\mad$ exists and $\mad = \mathfrak{C}_{n+1}(\M)$.
\end{cor}

\section{Forcing over $\Sigma_k$-admissible premice}

We assume that the reader is familiar with the level-by-level correspondence in terms of the forcing theorem and fine structure between a premouse and its forcing extension, as described in Lemma 3.6 and Theorem 3.9 of \cite{Steel_2008}. From the methods employed there and in Lemma 3.20 of \cite{schlutzenberg2023mouse} we have the following.

\begin{lem} \label{lem: level by level forcing}
    Let $\M$ be a $k$-sound premouse such that $\M = L_\alpha (\M \vert \kappa)$ for some $\kappa < \OR^\M$ and $\alpha \geq 1$. Let $\mathbb{P} \in \M \cap \mathcal{P}(\kappa)$ be a forcing poset that is definable from parameters over $\M \vert \kappa$. Let $g$ be $(\M,\mathbb{P})$-generic.

    Then for all $(\xi,m) \leq (\alpha,k)$ such that $\xi \geq \kappa$
    \begin{enumerate}
        \item $\max \{ \rho_m^{\M \vert \xi}, \kappa \} = \max \{ \rho_m^{\M \vert \xi [g]}, \kappa \}$,
        \item if $\kappa \notin p_m^{\M \vert \xi}$, then $p_m^{\M \vert \xi} \cup \kappa = p_m^{\M \vert \xi [g]} \cup \kappa$, and if $\kappa \in p_m^{\M \vert \xi}$, then $p_m^{\M \vert \xi} \setminus \{ \kappa \} \cup \kappa = p_m^{\M \vert \xi [g]} \cup \kappa$, and
        \item there is an $r\Sigma_{m+1}^{\M \vert \xi} (\{ \kappa \})$-relation $\Vdash_{m+1}^{\text{strong}}$ \textit{ (the strong $r\Sigma_{m+1}$ forcing relation)} such that for all $r\Sigma_{m+1}$ formulas $\varphi (v_0,...,v_{l-1})$ and all $\mathbb{P}$-names $\tau_0,...,\tau_{l-1} \in \M \vert \xi$,
        \[
            \M \vert \xi [g] \models \varphi(\tau_0^g,...,\tau_{l-1}^g) \iff \exists p \in g (p \Vdash_{m+1}^{\text{strong}} \varphi(\tau_0,...,\tau_{l-1})).
        \]
    \end{enumerate}
\end{lem}

We then have the following forcing theorem for $\Sigma_n$-admissible premice.

\begin{lem} \label{lem: forcing theorem}
    Let $\M$ be a $\Sigma_k$-admissible premouse. Let $\mathbb{P} \in \M \cap \mathcal{P}(\kappa)$ be such that $\kappa^{+\M}$ exists. Let $\varphi (v)$ be an $r\Sigma_k$ formula or an $r\Pi_k$ formula, and $\sigma_0,\dots,\sigma_m \in \M^{\mathbb{P}}$. Suppose that $g$ is $(\M,\mathbb{P})$-generic. Then the following are equivalent:
    \begin{itemize}
        \item $\M[g] \models \varphi(\sigma_0^g,\dots,\sigma_m^g)$, and
        \item there is $p \in g$ such that $p \Vdash_{k}^{\text{strong}} \varphi(\sigma_0,\dots,\sigma_m)$.\footnote{If $\varphi$ is $r\Pi_k$, we define $ p \Vdash_{k}^{\text{strong}} \varphi(\sigma_0,\dots,\sigma_m)$ to mean that there is no condition $q \leq p$ such that $q \Vdash_{k}^{\text{strong}} \neg \varphi(\sigma_0,\dots,\sigma_m)$.}
    \end{itemize}
\end{lem}

\prf
    By Lemma \ref{lem: n-1 projectum}, $\M$ is $(k-1)$-sound. Thus, in the case that $\varphi$ is an $r\Sigma_k$ formula, the claim follows from Lemma \ref{lem: level by level forcing}.

    In the case that $\varphi$ is an $r\Pi_k$ formula, suppose that $\M[g] \models \varphi(\sigma_0^g,\dots,\sigma_m^g)$, i.e.~$\M[g] \not \models \neg \varphi(\sigma_0^g,\dots,\sigma_m^g)$. Note that $\neg \varphi(\sigma_0^g,\dots,\sigma_m^g)$ is an $r\Sigma_k$ formula; therefore, we have, by Lemma Lemma \ref{lem: level by level forcing}, that for all $p \in g$, $p \not \Vdash_k^{\text{strong}} \varphi (\sigma_0,\dots,\sigma_m)$. Let 
    \[
        D := \{ p \in \mathbb{P} : p \Vdash_{k}^{\text{strong}} \varphi(\sigma_0,\dots,\sigma_m) \lor p \Vdash_{k}^{\text{strong}} \neg \varphi(\sigma_0,\dots,\sigma_m) \}.
    \]
    Note that since $\rho_k^\M \geq \kappa^{+\M}$, $D \in \M$. Clearly, $D$ is dense in $\mathbb{P}$ so that $g \cap D \neq \emptyset$. Let $p \in g \cap D$. Note that $p \Vdash_k^{\text{strong}} \varphi(\sigma_0,\dots,\sigma_m)$, since otherwise by Lemma \ref{lem: level by level forcing}, $\M[g] \models \neg \varphi(\sigma_0,\dots,\sigma_m)$.

    Now suppose that there is some $p \in g$ such that $p \Vdash_k^{\text{strong}} \varphi(\sigma_0,\dots,\sigma_m)$. Suppose for the sake of contradiction, $\M[g] \not \models \varphi(\sigma_0^g,\dots,\sigma_m^g)$, i.e.~$\M[g] \models \neg \varphi(\sigma_0^g,\dots,\sigma_m^g)$. By Lemma \ref{lem: level by level forcing}, there is $q \in g$ such that $ q \Vdash_k^{\text{strong}} \neg \varphi(\sigma_0,\dots,\sigma_m)$. However, since $g$ is a filter, there is some $r \leq p,q$, a contradiction!
\eprf

\begin{cor} \label{cor: forcing theorem for conjunctions}
    Let $\M$ be a $\Sigma_k$-admissible premouse. Let $\mathbb{P} \in \M \cap \mathcal{P}(\kappa)$ be such that $\kappa^{+\M}$ exists. Let $\varphi (v)$ be an $r\Sigma_k \land r\Pi_k$ formula and let $\sigma_0,\dots,\sigma_m \in \M^{\mathbb{P}}$. Suppose that $g$ is $(\M,\mathbb{P})$-generic. Then the following are equivalent:
    \begin{itemize}
        \item $\M[g] \models \varphi(\sigma^g_0,\dots,\sigma^g_m)$, and
        \item there is $p \in g$ such that $p \Vdash^{\text{strong}}_{k} \varphi(\sigma_0,\dots,\sigma_m)$.\footnote{If $\varphi \equiv \psi_1 \land \psi_2$, where $\psi_1$ is $r\Sigma_k$ and $\psi_2$ is $r\Pi_k$, $p \Vdash^{\text{strong}}_{n} \varphi$ means that $p \Vdash^{\text{strong}}_{k} \psi_1$ and $p \Vdash^{\text{strong}}_{k} \psi_2$.}
    \end{itemize}
\end{cor}

The following lemma says that forcing over $\Sigma_k$-admissible premice with posets of size less than the largest cardinal of the premouse, preserves admissibility. Its proof is straightforward, so we omit it.

\begin{lem} \label{lem: preservation of kp}
    Let $\M$ be a $\Sigma_k$-admissible premouse and $\kappa$ an $\M$-cardinal such that $\kappa^{+\M}$ exists. Let $\mathbb{P} \in \M \cap \mathcal{P}(\kappa)$ be a forcing poset. If $g$ is  $(\M,\mathbb{P})$-generic, then $\M[g]$ is $\Sigma_k$-admissible.
\end{lem}

\section{A variant of the truncation lemma}

Recall the following coarse definition. If $M$ is a possibly ill-founded structure in some signature $\mathcal{L}$ that extends $\mathcal{L}_{\dot{\in}}$, we call
\[
    \text{wfp}(M) := \{ x \in \lfloor M \rfloor \mid \in^M \restriction (\text{trc}_{\in^M} (\{x\}))^2 \text{ is wellfounded} \}
\]
the wellfounded part of $M$. By \cite{Barwise_2017} and Problem 5.27 of \cite{schindler2014set}, if $M \models \text{KP}$, then $\text{wfp}(M) \models \text{KP}$. This is also sometimes referred to as the Truncation Lemma. We aim to show something similar in the case that $M$ is an illfounded structure which is a model of $V = L[E]$. Let us define the wellfounded cut as in Definition 19 of \cite{KRUSCHEWSKI_SCHLUTZENBERG_2025}. The next lemma is a variant of the Truncation Lemma which we will often refer to as Ville's Lemma or a higher version of Ville's Lemma.

\begin{lem} \label{lem:higher ville}
    Let $k \geq 1$ and $M = (\lfloor M \rfloor, \in^M, \E^M)$ be an illfounded $\mathcal{L}_{\in, \dot{E}}$-structure such that $M \models ``V = L[E]"$, $\text{wfp} (M)$ is transitive, $M$ is $\omega$-wellfounded, and if $k \geq 2$, then $S^M_{k-2} \cap (\OR^M \setminus \text{wfo}(M)) \neq \emptyset$. Then, if $\text{wfc}(M) \prec_{\Sigma_{k-1}} M$, then $\text{wfc}(M) \models \Sigma_k \text{-}\kp$. 
\end{lem}

\begin{proof}
    Note that $\omega \in \text{wfc}(M)$, as $M$ is $\omega$-wellfounded. Suppose $\text{wfc}(M) \prec_{\Sigma_{k-1}} M$. It suffices to see that $\text{wfc}(M) \models \Sigma_k \text{-Collection}$. By induction on $\alpha$, it easily follows that for $\alpha < \text{wfo}(M)$, $J_\alpha^{\E^M} = (J_\alpha)^M \in \text{wfp}(M)$, so that $\text{wfc}(M) = J_{\text{wfo}(M)}^{\E^M} \subseteq \text{wfp}(M)$.
    Let $\varphi$ be a $\Pi_{k-1}$ formula and $a,p \in \text{wfc}(M)$ such that
    \[
        \text{wfc}(M) \models \forall x \in a \exists y \varphi(x,y,p).
    \]
    Since $\text{wfc}(M) \prec_{\Sigma_{k-1}} M$,
    \[
        M \models \forall x \in a \exists y \varphi(x,y,p).
    \]
    Let $\gamma \in S^M_{k-2} \cap (\OR^M \setminus \text{wfo}(M))$.
    Note that because $\text{wfc}(M) \prec_{\Sigma_{k-1}} M$ and $M\vert \gamma \prec_{\Sigma_{k-2}} M$, clearly $\text{wfc}(M) \prec_{\Sigma_{k-2}} M\vert \gamma$, and, in fact, $\text{wfc}(M) \prec_{\Sigma_{k-1}} M\vert \gamma$.

    In $M \vert \gamma$ we may define a function $F$ with $\dom(F) = a$ such that for $x \in a$,
    \[
        F(x) = \eta \iff M \vert \gamma \models x \in a \land \exists y \in M \vert (\eta +1) \varphi(x,y,p) \land \forall y \in (M \vert \eta) \neg \varphi(x,y,p).
    \]
    Since $\text{wfc}(M) \subseteq M \vert \gamma$, it follows that $F(x) < \text{wfo}(M)$ for all $x \in a$. However, this means that $\eta := \bigcup_{x \in a} F(x) \subset \text{wfo}(M)$. Since $F$ is definable over $M \vert \gamma$, we must have that $\eta < \text{wfo}(M)$. This means that
    \[
        \text{wfc}(M) \models \forall x \in a \exists y \in M \vert (\eta + 1) \varphi(x,y,p).
    \]
    Thus, $\text{wfc}(M) \models \Sigma_k \text{-} \kp$.
\end{proof}

\section{A generating class of fixed points} \label{section: finding the sequence}

Later on in the analysis, we will make use of a sequence $S_\infty$ of ordinals cofinal in $\alpha$ that is fixed pointwise by iteration maps between many premice in the direct limit systems to be considered. Moreover, this sequence will be sufficiently generating for those premice, as described in Lemma \ref{S is cofinal}.

The sequence $S_\infty$ will be of the form $\langle \alpha_k \mid k < \omega \rangle {^\frown} \langle \gamma_k \mid k < \omega \rangle$, where $\langle \alpha_k \mid k < \omega \rangle \subset \kappa^{+\lx}$ is cofinal in $\kappa^{+\lx}$ and $\langle \gamma_k \mid k < \omega \rangle$ is cofinal in $\alpha$ and defined from $\langle \alpha_k \mid k < \omega \rangle$. We define $\langle \alpha_k \mid k < \omega \rangle$ via the leftmost branch of a tree $T$ that essentially searches for an illfounded model whose wellfounded cut is $L_\alpha(\R^+)$, where $\R^+$, introduced in Definition \ref{def: R plus} below, is the set of reals of a symmetric extension of $\lx$.

We will define this tree $T$ and show the necessary facts first for the case $n = 1$ to illustrate some of the basic ideas. We will then define $T$ for a general $n \geq 1$ and prove the necessary facts about it. These proofs are similar to the ones in the case that $n=1$, but involve more fine structure.

For an arbitrary tree $T$, i.e.~$T$ is a set of finite sequences closed under initial segments, and $s \in T$, define $T_s = \{ t \in T : t \subseteq s \lor s \subseteq t \}$.

\begin{dfn} \label{def: R plus}
    Let $G$ be $(\lx,\col(\omega, {< \kappa}))$-generic. Let $\HC^+ = \HC^{\lxg}$ and let 
    \[
        \R^+ = \HC^+ \cap \R = L_{\alpha}(\HC^+) \cap \R = \lr \cap \R.
    \] 
    We write $\theta = \kappa^{+\lr}$.
\end{dfn}

For the fine structural theory of the model $\lr$ we refer the reader to Chapter 1 of \cite{Steel_2008_lr}. This means that in particular, when working with $\lr$ we always consider it in the language $\mathcal{L}_{\dot{\in},\dot{\R}}$ with $\dot{\R}$ interpreted as $\R^+$, and when taking fine structural hulls we always include all reals in $\R^+$.

    \begin{lem} \label{basic facts about L(R)}
        $\lr \models \Sigma_n \text{-}\kp_{\mathcal{L}_{\dot{\in},\dot{\R}}} \land  ~\omega_1 = \kappa \land ``\omega_2 = \kappa^+ \text{ is the largest }\aleph"$. Also,
        \begin{itemize}
            \item $\rho_{n-1}^{\lr} = \alpha$, and
            \item $\rho_{n}^{\lr} = \kappa^{+\lx} =\theta $.
            
        \end{itemize}
        Moreover, $\alpha$ is minimal such that $\lr \models \Sigma_n \text{-}\kp_{\mathcal{L}_{\dot{\in},\dot{\R}}} \land ``\kappa^+ \text{ exists}"$.
    \end{lem}
    
    \prf
        By Lemma \ref{lem: preservation of kp}, $\lxg$ is $\Sigma_n$-admissible, and so it easily follows that $\lr \subset \lxg$ is $\Sigma_n$-admissible. The rest follows from Lemma \ref{lem: level by level forcing}.
    \eprf

Let us define $T$ in the case that $n = 1$, i.e.~until further notice we will assume that $n = 1$.

\begin{dfn} \label{dfn: tree in the case n=1}
Let $T_1$ be the tree of attempts to construct a sequence $\langle \alpha_k,\beta_k \rangle_{k < \omega}$ such that the following hold:
    \begin{enumerate}
        \item $\kappa < \beta_k < \alpha_k < \theta$,
        \item $L_{\alpha_k}(\R^+)\models ``\kappa^+\text{ exists}"$, and
        \item \label{1 embedding} there is a $\Sigma_1$-elementary embedding $\pi \colon L_{\alpha_k}(\R^+)\to L_{\alpha_{k+1}}(\R^+)$ such that $\pi \restriction \kappa^{+L_{\alpha_k} (\R^+)} = \id$ and $\pi(\beta_k) > \beta_{k+1}$.
    \end{enumerate}

\end{dfn}

For a node $s \in T_1 \setminus \{ \emptyset \}$, let $(\alpha_s,\beta_s) = s(\lh(s)-1)$ and let $\theta_s = \kappa^{+L_{\alpha_s}(\R^+)}$. We will later prove a more general version of the following lemma (see Lemma \ref{lem: projectum of L(R)}).

\begin{lem} 
    Let $s \in T_1\setminus \{ \emptyset \}$. Then 
    \[
        L_{\alpha_s}(\R^+) = \Hull{L_{\alpha_s}(\R^+)}{1} (\R^+\cup \{\R^+\} \cup (\theta+1)).
    \]
\end{lem}

It follows that the embedding $\pi$ as given in clause \ref{1 embedding} of Definition \ref{dfn: tree in the case n=1} is uniquely determined by $\alpha_k$ and $\alpha_{k+1}$.
Moreover, $T_1$ is definable over $L_{\alpha}(\R^+) \vert \theta$, so that $T_1 \in L_{\alpha}(\R^+)$.

\begin{lem} \label{lem: illfoundedness in n=1 case}
    $T_1$ is illfounded. 
\end{lem}

\prf
    Let $h \colon \omega \to \theta$ be sufficiently $(\lr, \col(\omega,\theta))$-generic such that $\lr[h] \models \Sigma_1$-KP\footnote{See Theorem 10.17 of \cite{Mathias2015ProvidentSA} for an example of such a generic.}. Let $T^\prime$ be the tree that is defined as $T_1$ in Definition \ref{dfn: tree in the case n=1} with the exception that we do not require $\pi(\beta_k) > \beta_{k+1}$ in clause \ref{1 embedding} to hold, and the additional requirement that $\alpha_k > h(k)$. Note that $T^\prime \subseteq {^{<\omega}\theta}$. For a node $s \in T^\prime \setminus \{ \emptyset \}$ we write $\alpha_s = s(\lh(s)-1)$ and $\theta_s = \kappa^{+L_{\alpha_s}(\R^+)}$.
    \begin{claim} \label{n equals one case illfoundedness}
        $T^\prime$ is illfounded.
    \end{claim}
    \prf 
        Using that $\rho_1^{\lr} = \theta$, it is easy to see that there is a sequence $\langle \delta_k \mid k < \omega \rangle$ such that for all $k < \omega$, $h(k) < \delta_k < \delta_{k+1} < \theta$ and 
        \[
            \delta_k = \theta \cap \Hull{\lr}{1}(\delta_k \cup \{ \theta \}).
        \]
        Letting $\alpha_k$ be the ordinal height of the transitive collapse of $\Hull{\lr}{1}(\delta_k \cup \{ \theta \})$ it is easy to see that $\langle \alpha_k \mid k < \omega \rangle$ is a branch through $T^\prime$.
    \eprf

    We can associate to a cofinal branch $b = \langle \alpha_k \mid k < \omega \rangle$ through $T^\prime$ a branch model $M_b$ which is the direct limit of the models $L_{\alpha_k}(\R^+)$ and the maps $\pi_{mk} \colon L_{\alpha_m}(\R^+) \to L_{\alpha_k}(\R^+)$, since for $m < k < l < \omega$,
    \[
        \pi_{ml} = \pi_{kl} \circ \pi_{mk}.
    \] 
    To complete the proof it suffices to see that there is a cofinal branch $b$ of $T^\prime$ such that $M_b$ is illfounded. Suppose for the sake of contradiction not, i.e.~for every cofinal branch $b$ of $T^\prime$, $M_b$ is wellfounded.

    Let $b = \langle \alpha_k \mid k < \omega \rangle$ be a cofinal branch through $T^\prime$. By assumption $M_b$ is wellfounded. Note that $M_b = L_\beta(\R^+)$ for some $\beta$. 
    We claim $\beta \in (\theta, \alpha]$. Suppose not. Then there is $k < \omega$ such that $\alpha \in \ran (\pi_{k\infty})$, where $\pi_{k\infty} \colon L_{\alpha_k} (\R^+) \to L_\beta (\R^+)$ is the direct limit map. Let $\bar{\alpha} \in L_{\alpha_k} (\R^+)$ be such that $\pi_{k\infty} (\bar{\alpha}) = \alpha$.
    Note that $\bar{\alpha} > \theta_{(\alpha_k)}$.
    But then
    \[
    L_{\bar{\alpha}} (\R^+) \models \Sigma_1\text{-}\kp \land ``\text{the cardinal successor of }\kappa \text{ exists}",
    \] 
    contradicts the minimality of $\alpha$.

    For a node $s \in T^\prime$, let $\tho_s := \{ \varphi(\vec{x},\theta,\R^+) \colon \varphi \text{ is a }\Sigma_1 \text{ formula}, \vec{x}\in \R^+ \cup \theta_s, \text{ and }L_{\alpha_s} (\R^+) \models \varphi (\vec{x},\theta_s, \R^+) \}$. For $b$, a cofinal branch of $T^\prime$, let $\beta_b \in (\theta, \alpha]$ be such that $M_b = L_{\beta_b} (\R^+)$. 
    Note that for $s \in b$, we have $\pi_{s\infty}(\theta_s) = \theta$, $\pi_{s\infty} \restriction \theta_s = \id$, and $\pi_{s\infty} \restriction \R^+ = \id$. Thus, for all $s \in b$, there is some $\gamma < \alpha$ such that $L_\gamma (\R^+) \models \tho_s$.

    We now want to prune $T^\prime$ inside $\lr [h]$.
    Note that for every node $s \in T^\prime$ at least one of the following holds true inside $\lr[h]$:
    \begin{itemize}
        \item there is a ranking function for $T^\prime_s$, or
        \item there is some $\gamma < \alpha$ such that $L_\gamma (\R^+) \models \tho_s$.
    \end{itemize}
    Since both of these are $\Sigma_1$ statements it follows by $\Sigma_1$-Collection in $\lr[h]$ that there is some $\gamma < \alpha$ such that for all $s \in T^\prime$ there is a ranking function for $T^\prime_s$ in $L_\gamma (\R^+) [h]$ or $L_\gamma (\R^+) \models \tho_s$. Let $T^{\prime \prime}$ be the result of pruning the tree $T^\prime$ over $L_\gamma (\R^+)$, i.e.~removing the nodes $s$ of $T^\prime$ for which there is a ranking function for $(T^\prime)_s$ in $L_\gamma(\R^+)$, and let
    \[
        \tho := \bigcup_{s \in T^{\prime \prime}} \tho_s.
    \]
    Note that $\tho$ is definable over $L_\gamma (\R^+)[h]$ and therefore $\tho \in \lr[h]$. By the way we picked $h$, it follows that $\tho = \tho^{L_\alpha (\R^+)}_{1} (\R^+ \cup \{ \R^+ \} \cup \theta + 1)$. But then it easily follows by the $\Sigma_1$-admissibility of $\lr[h]$ that $\alpha \in \alpha$, a contradiction!
\eprf

\setcounter{claimcounter}{0}

Note that, by the lemma, there exists a branch through $T_1$, and hence the left-most branch of $T_1$ exists.

\begin{theorem}  \label{thm: branch in the n=1 case}
    Let $b$ be the left-most branch of $T_1$. Then:
    \begin{itemize}
        \item if $M_b$ is the direct limit given by $b$, then $\text{wfc}(M_b) = L_\alpha(\R^+)$, and
        \item for every $s \in b$, $\{s\}$ is $\Sigma_1 \land \Pi_1$ definable over $\lr$ from the parameters $\{ \kappa, T_1 \}$.
    \end{itemize}
\end{theorem}

\prf
    Let us suppose that $T_1$ is pruned, i.e.~$T_1$ does not have end nodes, and let $b$ be the left-most branch of $T_1$, i.e.~for every node $s \in b$, if $t \in T_1$ is such that $\lh (s) = \lh (t)$ and $t <_{\text{lex}} s$, then $T_t$ is wellfounded, which by $\Sigma_1$-KP is equivalent to the existence of a ranking function for $T_t$ in $\lr$. Note that this gives a $\Sigma_1 \land \Pi_1$ definition for $\{ s \}$ over $L_\alpha(\R^+)$, since by $\Sigma_1$-Collection, the statement that for all $t$ as above there is a ranking function is $\Sigma_1$, and the statement that for $(T_1)_s$ there is no ranking function is $\Pi_1$. It remains to see that if $M_b$ is the direct limit given by $b$, then $\text{wfc}(M_b) = L_\alpha(\R^+)$. This will essentially follow from the following claim.
    \begin{claim}
        $\kappa^{+M_b} = \kappa^{+\lx} = \theta$.
    \end{claim}
    \prf
        Suppose for the sake of contradiction that $(\kappa^+)^{M_b} < \theta$. Let $\beta = \text{wfo}(M_b) = \text{wfc}(M_b) \cap \OR$. Note that by the definition of $T_1$, we have $(\kappa^+)^{M_b} < \beta$. Moreover, $\beta < \theta$, for if $\beta \geq \theta$, then in fact $\kappa^{+L_\beta (\R^+)} = \theta$, a contradiction! By Ville's Lemma, $L_\beta(\R^+)$ is $\Sigma_1$-admissible and so $L_\beta (\R^+) \models \tho_1$\footnote{See Definition \ref{def: base theory}}, contradicting the minimality of $\alpha$!
    \eprf

    Let $\beta$ still be $\text{wfo}(M_b)$. Using the same argument as in the claim, we must have $\beta \geq \alpha$. However, $\beta > \alpha$ cannot be true either, as the argument from the proof of Lemma \ref{lem: illfoundedness in n=1 case} shows.
\eprf

In the case where $n=1$ the branch of $T_1$ identified in the last theorem is $\langle \alpha_k \mid k < \omega \rangle$ of the sequence $S_\infty$. We will now consider the general case, i.e.~$n$ is the natural number we fixed before Definition \ref{after fixing n}.

\begin{dfn} \label{definition of T}
Let $T$ be the tree of attempts to construct a sequence $\langle \alpha_k,\beta_k \rangle_{k < \omega}$ such that the following hold:
\begin{enumerate}
    \item $\kappa < \beta_k < \alpha_k < \theta$,
    \item $L_{\alpha_k}(\R^+) \models ``\kappa^+ \text{ exists}"$,
    \item \label{T: projectum}$\rho^{L_{\alpha_k}(\R^+)}_{n-1} = \alpha_k$,
    \item \label{T: embedding} there is an $r\Sigma_n$-elementary embedding $\pi \colon L_{\alpha_k} (\R^+) \to  L_{\alpha_{k+1}} (\R^+)$ such that $\pi \restriction \kappa^{+ L_{\alpha_k} (\R^+)} = \id$,
    \item \label{T: illfounded} $\pi(\beta_k) > \beta_{k+1}$.
\end{enumerate}
\end{dfn}

Note that from condition \ref{T: projectum} it follows that $S^{L_{\alpha_k}(\R^+)}_{n-1}$ is club in $\alpha_k$ and that $L_{\alpha_k} (\R^+)$ is $\Sigma_{n-1}$-admissible.
For a node $s \in T \setminus \{ \emptyset \}$ we let $(\alpha_s, \beta_s) = s(\lh(s)-1)$.

\setcounter{claimcounter}{0}

\begin{lem} \label{lem: projectum of L(R)}
    Let $s \in T\setminus \{ \emptyset \}$. Then $L_{\alpha_s}(\R^+) = H$,
    where $H := \Hull{L_{\alpha_s}(\R^+)}{n} (\R^+ \cup \{\R^+\} \cup (\theta_s+1))$.
\end{lem}

\prf
Note that $H$ is transitive, since $\theta_s$ is the largest cardinal of $L_{\alpha_s}(\R^+)$. Suppose for the sake of contradiction that $H \subsetneq L_{\alpha_s}(\R^+)$, so that $H$ is bounded in $L_{\alpha_s}(\R^+)$. Let $\beta := \OR \cap H$ so that $H = L_\beta (\R^+)$ and $\beta < \alpha$.
\begin{claim}
    $L_\beta (\R^+)$ is $\Sigma_n$-admissible.
\end{claim}
\prf
    We proceed by induction. In the case where $n=1$, it is easy to see that $L_\beta(\R^+)$ is $\Sigma_1$-admissible, since every instance of $\Sigma_1$-Collection of $L_\beta(\R^+)$ is bounded by $\beta$. But then in $L_{\alpha_s}(\R^+)$ this is a $\Sigma_1$-statement, so there is a bound in $L_{\beta}(\R^+) \prec_{\Sigma_1} L_{\alpha_s}(\R^+)$. 
    
    So suppose that $n \geq 2$ and for the sake of contradiction that $L_\beta(\R^+)$ is not $\Sigma_n$-admissible. Let $\varphi$ be a $\Pi_{n-1}$ formula and $a,p \in L_\beta (\R^+)$ be a witness to this, that is
    \[
        L_\beta (\R^+) \models \forall x \in a \exists y \varphi(x,y,p),
    \]
    but there is no bound in $L_\beta (\R^+)$.
    Note that 
    \[
        L_{\alpha_s} (\R^+) \models \exists \beta^\prime \forall x \in a  \exists y \in L_{\beta^\prime} (\R^+) \varphi(x,y,p),
    \]
    as witnessed by $\beta$. But then because $\rho_{n-1}^{L_{\alpha_s}(\R)} = \alpha_s$,
    \[
         \psi := \exists \beta^\prime (\beta^\prime \in S_{n-1} \land (\forall x \in a  \exists y \in L_{\beta^\prime} (\R^+) (L_{\beta^\prime}(\R^+) \models \varphi(x,y,p)))),
    \]
    holds in $L_{\alpha_s}(\R^+)$, which is $\Sigma_n$. Since $L_\beta (\R^+) \prec_{\Sigma_n} L_{\alpha_s} (\R^+)$, $L_\beta (\R^+) \models \psi$. However, this means that $L_\beta (\R^+) \models \exists \beta^\prime \forall x \in a  \exists y \in L_{\beta^\prime} (\R^+) \varphi(x,y,p)$, so that there is a bound in $L_\beta (\R^+)$, a contradiction!
\eprf

This is a contradiction, since $L_\beta (\R^+)$ cannot be $\Sigma_n$-admissible by Lemma \ref{basic facts about L(R)}!
\eprf

\setcounter{claimcounter}{0}

Note that this lemma shows that for nodes $s,t \in T$ such that $s <_{T} t$ the embedding $\pi_{s,t} \colon L_{\alpha_s} (\R^+) \to L_{\alpha_t} (\R^+)$ given by condition \ref{T: embedding} is uniquely determined by $\alpha_s$ and $\alpha_t$.
Moreover, it is easy to see that $T$ is definable over $L_{\alpha}(\R^+) \vert \theta$ and thus, in particular $T \in L_{\alpha}(\R^+)$.

\begin{dfn} \label{def: signature of the structures}
    For $\beta > \kappa$ such that $L_\beta (\R^+) \models ``\kappa^+ \text{ exists}"$, let $N_\beta$ be the structure $(\lfloor L_\beta(\R^+)\rfloor, \in, \R^+, (x)_{x \in \R^+}, \kappa^{+\lx})$ in the signature $\{\dot{\in},\dot{\R},(\dot{x})_{x\in \R^+},\dot{\kappa^+}\}$.
\end{dfn}

\begin{lem} \label{lem: T is illfounded} 
    $T$ is illfounded.
\end{lem}

\prf
    Let $h$ be sufficiently $(\lr,\col(\omega,\theta))$-generic so $\lr[h] \models \Sigma_n$-KP.\footnote{The existence of such generics follows from straightforward adaption of the proof of Theorem 10.17 in \cite{Mathias2015ProvidentSA}.} Let $T_{h}$ be the tree defined as $T$, dropping $\beta_k$ and clause \ref{T: illfounded} from the definition of $T$, and with the additional requirement that $\alpha_k > h(k)$. Thus, $T_{h} \subseteq {^{<\omega}\theta}$. For a node $s \in T_{h} \setminus \{\emptyset\}$, we write $\alpha_s = s(\lh(s)-1)$, and let 
    \begin{align*}
        \tho_s := \{ \varphi(\vec{x},\theta,\R^+) : \varphi \text{ is } r\Sigma_n \land \vec{x} \in (\fin{\theta \cup \R^+}) 
        \land  L_{\alpha_s}(\R^+) \models \varphi(\vec{x}, \theta,\R^+) \}.
    \end{align*}

    Let us define a sequence of trees $\langle T_\gamma \mid \gamma < \alpha \rangle$, which we will call  the $(n+1)$\textit{-pruning process of} $T_{h}$.

    Set $T_0 := T_{h}$. 
    Suppose that $T_\gamma$ is defined, where $\gamma < \alpha$. Let 
    \[
        T_{\gamma+1} := \{ s \in T_\gamma : \forall \xi < \theta \exists t \in (T_\gamma \setminus \{\emptyset\}) ( s \subseteq t \land \xi < \theta_t) \}.
    \]
    Let $\lambda < \alpha$ be a limit ordinal and suppose that $T_\gamma$ is defined for all $\gamma < \lambda$. In the case where $\lambda \notin S^{N_\alpha}_{n-1}$, let $T_\lambda := \bigcap_{\gamma < \lambda} T_\gamma$. In the case that $\lambda \in S^{N_\alpha}_{n-1}$ we let
    \[
        T_\lambda := \{ s \in \bigcap_{\gamma < \lambda} T_\gamma \mid \neg (\exists \varphi \exists \psi (\varphi \leq^*_s \psi \land N_\alpha \vert \lambda \models \psi < \varphi)) \},
    \]
    where $\varphi \leq^*_s \psi$ means that $\varphi \leq^* \psi \in \tho_s$.

    Finally, set $T^\prime_{h} = \bigcap_{\gamma < \alpha} T_\gamma$. Note that the tree $T^\prime_h$ does not have end nodes. Moreover, $S_{n-1}^{N_\alpha}$ is $r\Sigma_n^{N_\alpha}$. Thus, the sequence $\langle T_\gamma \mid \gamma < \alpha \rangle$ is definable by a $\Sigma_n$ recursion over $N_\alpha$.

    \begin{claim}
        $T^\prime_{h}$ is illfounded.
    \end{claim}
    \prf
        The proof is analogous to the proof of Claim \ref{n equals one case illfoundedness} in the proof of Lemma \ref{lem: illfoundedness in n=1 case}. However, this time we have to verify that the branch produced is in fact in $T^\prime_{h}$.

        Since $\rho_n^{L_{\alpha}(\R^+)} = \rho_n^{N_\alpha}= \theta$, there exists a sequence $\langle \delta_k \mid k < \omega \rangle$ such that for all $k < \omega$,
        \begin{itemize}
            \item $h (k) < \delta_k < \delta_{k+1}< \kappa^{+\lx}$, and
            \item $\delta_k = \kappa^{+\lx} \cap \Hull{N_\alpha}{n} (\delta_k \cup \{ \kappa^{+\lx} \}) \in \OR$.
        \end{itemize}
        Let $cH_k := c\Hull{N_\alpha}{n} (\delta_k \cup \{ \kappa^{+\lx} \})$ and set $\alpha_k = \OR \cap cH_k$. Note that for $\lambda \in S^{N_\alpha}_{n-1}$, it cannot be the case that there are $r\Sigma_n$ formulas $\varphi$ and $\psi$ such that $N_{\alpha_k} \models \varphi \leq^* \psi$ and $N_\alpha \vert \lambda \models \psi < \varphi$. It is then easy to see that $\langle \alpha_k \mid k < \omega \rangle$ is a branch through $T^\prime_{h}$.
    \eprf

    For a branch $b$ of $T^\prime_h$, let $M_b$ be the direct limit given by the branch. Note that we consider $M_b$ in the signature $\{\dot{\in},\dot{\R},(\dot{x})_{x\in \R^+},\dot{\kappa^+}\}$.

    \begin{claim} \label{claim: agreement of theories}
        If $M_b$ is wellfounded, then $\tho_{r\Sigma_n}^{M_b}(\theta) \subseteq \tho_{r\Sigma_n}^{N_\alpha}(\theta)$.\footnote{Recall that by Definition \ref{def: signature of the structures} the structure $N_\alpha$ has the signature $\{\dot{\in},\dot{\R},(\dot{x})_{x\in \R^+},\dot{\kappa^+}\}$.}
    \end{claim}
    \prf
        Suppose not. Let $\psi$ be an $r\Sigma_n$ formula and $\vec{x} \in \fin{\theta}$ be such that $\psi(\vec{x}) \in \tho_{r\Sigma_n}^{M_b}(\theta)$ but $\psi(\vec{x}) \notin \tho_{r\Sigma_n}^{N_\alpha}(\theta)$. We claim that there is an $r\Sigma_n$ formula $\varphi$ and $\vec{y} \in \fin{\theta}$ such that $\varphi(\vec{y}) \in \tho_{r\Sigma_n}^{N_\alpha}(\theta)$ but $\varphi(\vec{y}) \notin \tho_{r\Sigma_n}^{M_b}(\theta)$. Suppose not, i.e.~$\tho_{r\Sigma_n}^{M_b}(\theta) \supsetneq \tho_{r\Sigma_n}^{N_\alpha}(\theta)$. Then, we must have that $\tho_{r\Sigma_n}^{N_\alpha}(\theta)$ is a $\leq^*$-initial segment of $\tho_{r\Sigma_n}^{M_b}(\theta)$. But since $\tho_{r\Sigma_n}^{M_b}(\theta) \in N_\alpha$ and so all its initial segments are elements in $N_\alpha$, this means that $\tho_{r\Sigma_n}^{M_b}(\theta) \in N_\alpha$, a contradiction! Now, $N_\alpha \models \varphi(\vec{y}) <^* \psi(\vec{x})$, but for some $s \in b$, $N_{\alpha_s} \models \psi(\vec{x}) \leq^* \varphi(\vec{y})$. However, this means that $s$ must have gotten pruned during the $(n+1)$-pruning process, a contradiction!
    \eprf

    To finish the proof, it suffices to see that there is a branch $b$ through $T^\prime_{h}$ such that $M_b$ is illfounded. Suppose for the sake of contradiction that for every branch $b$ of $T^\prime_{h}$, $M_b$ is wellfounded.
    Note then that for any branch $b$ of $T^\prime_{h}$, $M_b = L_{\gamma_b} (\R^+)$ for some $\gamma_b$. By the construction of the tree $T_{h}$, $\kappa^{+L_{\gamma_b}} = \sup\{\theta_s \mid s \in b\} = \theta$ so that $\gamma_b > \theta$. Note that we consider $M_b$ as a structure in the signature $\{\dot{\in},\dot{\R},(\dot{x})_{x\in \R^+},\dot{\kappa^+}\}$ so that $M_b = N_{\gamma_b}$ for some $\gamma_b$. As in the proof of Lemma \ref{lem: illfoundedness in n=1 case}, one verifies that $\gamma_b \in (\theta, \alpha]$.
    
    \begin{claim}
        If $b$ is a branch of $T^\prime_{h}$, then $N_{\gamma_b} \prec_{\Sigma_{n-1}} N_\alpha$.
    \end{claim}

    \prf
        Let $b \in [T^\prime_{h}]$ and suppose for the sake of contradiction that $N_{\gamma_b} \not \prec_{\Sigma_{n-1}} N_\alpha$. This means that $\gamma_b < \alpha$ and $\delta := \sup(S^{N_\alpha}_{n-1} \cap \gamma_b) < \gamma_b$. Note that $N_{\gamma_b} = \Hull{N_{\gamma_b}}{n} (\theta)$. Let $\pi \colon N_{\gamma_b} \rightarrow N_\alpha$ be such that if $x = m\tau_\varphi^{N_{\gamma_b}} (\vec{x})  \in N_{\gamma_b}$, for an $r\Sigma_n$ formula $\varphi$ and $\vec{x} \in \fin{\theta}$, then $\pi(x) = m\tau_\varphi^{N_\alpha} (\vec{x})$. Note that $\pi$ is well-defined by Claim \ref{claim: agreement of theories}.
        Moreover, $r\Sigma_n$ statements are upwards preserved by $\pi$, i.e.~if $\varphi$ is $r\Sigma_n$, $p \in N_{\gamma_b}$, and $N_{\gamma_b} \models \varphi(p)$, then $N_\alpha \models \varphi(\pi(p))$.

        Note that $\pi \restriction (\theta +1) = \id$. Furthermore, since $\sup(S^{N_{\gamma_b}}_{n-1} \cap \gamma_b) = \gamma_b$ and $\pi [S^{N_{\gamma_b}}_{n-1}] \subset S^{N_\alpha}_{n-1}$ by the $\Sigma_{n-1}$-elementarity of $\pi$ and the fact that $S^{N_{\gamma_b}}_{n-1}$ is $\Pi_{n-1}$ definable over $N_\alpha$, it follows that $\pi \neq \id$ and so there is $\crt(\pi) > \theta$. This is a contradiction, since $\theta = \lgcd(N_{\gamma_b})$, but $\crt(\pi)$ is a cardinal of $N_{\gamma_b}$!
    \eprf

    It follows from the claim that for every node $s \in T_{h}$ there is $\gamma < \alpha$ such that $s \notin T_\gamma$, or there is some $\gamma \in S^{N_\alpha}_{n-1}$ such that $N_{\gamma} \models \tho_s$.
    Note that this is a disjunction of two $r\Sigma_n$ formulas. Thus, this is an instance of the $\Sigma_n$-Collection scheme, so that there is some $\gamma$, which works uniformly for all nodes $s \in T_{h}$. Recall that we are assuming that all branches through $T^{\prime}_{h}$ give wellfounded models. Thus, in particular,
    \[
        \tho_{r\Sigma_n}^{N_\alpha} (\theta) = \bigcup_{s \in T^\prime_{h}} \tho_s.
    \]
     But this means that over $N_\gamma$ the theory $\tho_{r\Sigma_n}^{N_\alpha} (\kappa^{+\lx})$ can be computed. Since $\rho_n^{N_\alpha} = \kappa^{+\lx}$, this is a contradiction!
     Thus, there must be branches of $T^\prime_{h}$ whose direct limit $M_b$ is illfounded. This shows that the tree $T$ is illfounded.
    \eprf

\setcounter{claimcounter}{0}
\setcounter{subclaimcounter}{0}
\setcounter{casecounter}{0}

\begin{theorem} \label{thm: branch theorem}
    There is a branch $b$ through $T$ such that
    \begin{itemize}
        \item if $M_{b}$ is the direct limit given by $b$, then $\text{wfc}(M_{b}) = L_{\alpha}(\R^+)$, and
        \item for every $s \in b$, $\{ s \}$ is $\Sigma_n \land \Pi_n$-definable over $L_{\alpha}(\R^+)$ from the parameter $\kappa^{+\lx}$.
    \end{itemize}
\end{theorem}

\begin{rem}
    In contrast to Theorem \ref{thm: branch in the n=1 case} the branch $b$ of Theorem \ref{thm: branch theorem} cannot be the left-most branch of $T$ in the case where $n \geq 2$, as otherwise it would be in $\lr$ which is impossible. Similarly, it cannot be the left-most branch of any tree in $\lr$. But we will prune $T$ in a certain way, producing a subtree $T^\prime$ that is definable over $\lr$ (but not an element of it), and we can take $b$ as the left-most branch of $T^\prime$.
\end{rem}

\prf
    Define the $(n+1)$-pruning $\langle T_\gamma \mid \gamma < \alpha \rangle$ of $T$ from $T$ just as the $(n+1)$-pruning of $T_h$ was defined from $T_h$ in the proof of Lemma \ref{lem: T is illfounded}. Let $T^\prime$ be the result; that is, $T^\prime$ is the last tree produced by the process.

    We claim $T^\prime$ is illfounded. We showed that $T^\prime_h$ has a cofinal branch $c = \langle \alpha_k \mid k < \omega \rangle$ such that $M_c$ is illfounded and $\sup \{ \alpha_k \mid k < \omega \} = \theta$. Let $c^\prime = \langle (\alpha_k, \beta_k) \mid k < \omega \rangle$, where the $\beta_k$'s witness the illfoundedness of $M_c$. Then $c^\prime$ is a branch of $T^\prime$. For in the successor steps of the $(n+1)$-pruning process, for every node $s \in c^\prime$ and for every $\xi < \theta$, there is some $t \in c^\prime$ such that $s <_T t$ and $\xi < \theta_t$, since $\sup\{\alpha_k \mid k < \omega \} = \theta$. And in the limit steps $\lambda \in S_{n-1}^{N_\alpha}$, there is no disagreement on the ordering $\leq^*$ between $N_\alpha$ and $N_{\alpha_k}$, since $\alpha_k \in T^\prime_h$.
    
    Let $b = \langle (\alpha_k, \beta_k) \rangle_{k<\omega}$ be the left-most branch of $T^\prime$ in the lexicographical ordering. Let $M_b$ denote the direct limit given by $b$, which we consider in the signature $\{\dot{\in},\dot{\R},(\dot{x})_{x\in \R^+},\dot{\kappa^+}\}$, and let $\beta = \text{wfo} (M_b)$. Note that $\kappa^{+M_b} < \beta < \OR^{M_b}$, i.e.~$M_b$ is illfounded and its $\kappa^+$ is in the wellfounded cut. We consider $M_b$ as a structure in the signature $\{\dot{\in},\dot{\R},(\dot{x})_{x\in \R^+},\dot{\kappa^+}\}$.

\begin{claim}
    $\kappa^{+M_b} = \theta$.
\end{claim}

\prf
    By construction, we have $\kappa^{+M_b} \leq \theta$. Let us suppose, for the sake of contradiction, that $\kappa^{+M_b} < \theta$.

\begin{subclaim}
    $\tho_{r\Sigma_n}^{N_\alpha}(\kappa^{+M_b}) \subsetneq  \tho_{r\Sigma_n}^{M_b}(\kappa^{+M_b})$.
\end{subclaim}
\prf
    Since $M_b$ is illfounded, there is an $r\Sigma_n$ formula $\varphi$ and $\vec{y} \in \fin{\kappa^{+M_b}}$ such that $\varphi(\vec{y}) \in \tho_{r\Sigma_n}^{M_b}(\kappa^{+M_b})$, but $\varphi(\vec{y}) \notin \tho_{r\Sigma_n}^{N_\alpha}(\kappa^{+M_b})$. So, $\tho_{r\Sigma_n}^{N_\alpha}(\kappa^{+M_b}) \neq  \tho_{r\Sigma_n}^{M_b}(\kappa^{+M_b})$.

    Suppose for the sake of contradiction that $\tho_{r\Sigma_n}^{N_\alpha}(\kappa^{+M_b}) \not \subseteq  \tho_{r\Sigma_n}^{M_b}(\kappa^{+M_b})$, i.e.~there is an $r\Sigma_n$ formula $\psi$ and $\vec{x} \in \fin{\kappa^{+M_b}}$ such that $\psi(\vec{x}) \in \tho_{r\Sigma_n}^{N_\alpha}(\kappa^{+M_b})$, but $\psi(\vec{x}) \notin \tho_{r\Sigma_n}^{M_b}(\kappa^{+M_b})$.
    Note that $N_\alpha \models \psi(\vec{x}) < \varphi(\vec{y})$, since $\varphi(\vec{y})$ does not hold in $N_\alpha$. However, $\varphi(\vec{y}) \leq_s^* \psi(\vec{x})$ for some $s \in b$. This is a contradiction, since $s$ must have been removed during the $(n+1)$-pruning of $T$!
\eprf

    Note that not only 
    \begin{equation} \label{agreement}
        \tho_{r\Sigma_n}^{N_\alpha}(\kappa^{+M_b}) \subsetneq  \tho_{r\Sigma_n}^{M_b}(\kappa^{+M_b}),
    \end{equation}
    but that $\tho_{r\Sigma_n}^{N_\alpha}(\kappa^{+M_b})$ is a $\leq^*$-initial segment of $\tho_{r\Sigma_n}^{M_b}(\kappa^{+M_b})$.\footnote{Note that since $M_b$ might be illfounded it could be that this not literally true, since $\tho_{r\Sigma_n}^{N_\alpha}(\kappa^{+M_b})$ might not be an element of $M_b$. In this case, we mean that $\tho_{r\Sigma_n}^{N_\alpha}(\kappa^{+M_b})$ is a cut of $\tho_{r\Sigma_n}^{M_b}(\kappa^{+M_b})$.}
    Let $\beta^\prime := \sup \{ \lv_\varphi^{M_b} (\vec{x}) \mid \varphi(\vec{x}) \in \tho_{r\Sigma_n}^{N_\alpha}(\kappa^{+M_b}) \}$.
    \begin{subclaim}
        $\beta^\prime \leq \beta$. 
    \end{subclaim} 
    \prf
        Let $\varphi(\vec{x}) \in \tho_{r\Sigma_n}^{N_\alpha}(\kappa^{+M_b})$. Let $\gamma = \lv_\varphi^{M_b} (\vec{x})$, so $\tho_{r\Sigma_{n-1}}^{M_b} (\gamma)$ witnesses $\varphi(\vec{x} )$. We may assume that $\gamma \geq \kappa^{+M_b}$. Since $\gamma = \lv_\varphi^{M_b} (\vec{x})$, there is $f \in r\Sigma_n^{M_b} (\{\vec{x}\})$ such that $f \colon \kappa^{+M_b} \to \gamma$ is surjective. There are $r\Sigma_n$ formulas $\psi_1$ and $\psi_2$ such that for $\eta, \zeta < \kappa^{+M_b}$,
        \[
            f(\eta) < f(\zeta) \iff M_b \models \psi_1(\eta,\zeta,\vec{x}),
        \]
        and
        \[
            f(\eta) \geq f(\zeta) \iff M_b \models \psi_2(\eta,\zeta,\vec{x}).
        \]
        Moreover, for all $\eta, \zeta < \kappa^{+M_b}$, either $\psi_1(\eta,\zeta,\vec{x}) \in \tho_{r\Sigma_n}^{M_b}(\kappa^{+M_b})$ or $\psi_2(\eta,\zeta,\vec{x}) \in \tho_{r\Sigma_n}^{M_b}(\kappa^{+M_b})$.
        Let $f^\prime \colon \kappa^{+\lx} \to \gamma^\prime$ be the function given by the evaluation of the defining formula of $f$ with parameters $\vec{x}$ in $N_\alpha$. Then for all $\eta, \zeta < \kappa^{+M_b}$ $ f^\prime(\eta) < f^\prime(\zeta) \iff N_\alpha \models \psi_1(\eta,\zeta,\vec{x})$ and $f^\prime(\eta) \geq f^\prime(\zeta) \iff N_\alpha \models \psi_2(\eta,\zeta,\vec{x})$, and either $\psi_1(\eta,\zeta,\vec{x}) \in \tho_{r\Sigma_n}^{N_\alpha}(\kappa^{+M_b})$ or $\psi_2(\eta,\zeta,\vec{x}) \in \tho_{r\Sigma_n}^{N_\alpha}(\kappa^{+M_b})$. But then by (\ref{agreement}), the theories must agree on these statements, so that we have an order-preserving embedding from $\gamma$ into $\gamma^\prime$, so $\gamma$ is wellfounded. This means that $\gamma < \beta$.
    \eprf

    \begin{case} $\beta = \beta^\prime$.
        Similarly as in the previous paragraph we might associate with every $\gamma < \beta^\prime$ some $\gamma^\prime$. Let $\bar{\beta}$ be the supremum of the $\gamma^\prime$'s for $\gamma < \beta^\prime$.
        Note that $\text{wfc} (M_b) \prec_{\Sigma_{n-1}} M_b$, since $\beta^\prime$ is a limit of elements of $S^{M_b}_{n-1}$ and $\bar{\beta}$ is a limit of elements in $S^{N_\alpha}_{n-1}$. However, it easily follows from \ref{T: projectum} of Definition \ref{definition of T}, that $S^{M_b}_{n-1} \cap (\OR^{M_b} \setminus \beta) \neq \emptyset$. Thus, by Lemma \ref{lem:higher ville}, $N_\beta$ is $\Sigma_n$-admissible, which contradicts the minimality of $\alpha$!
    \end{case}
    \begin{case}
    $\beta^\prime < \beta$.
    For a node $s \in T^\prime$ and $\gamma \in (\theta_s + 1, \alpha_s]$, we let $\tho_{s} (\gamma) := \tho_{r\Sigma_n}^{N_\gamma}(\theta_s )$. We also set $\tho_{s} (\alpha) := \tho_{r\Sigma_n}^{N_\alpha}(\theta_s)$.

    \begin{subclaim}
    There is $t \in T^\prime$ and an $r\Sigma_n$ formula $\varphi$ and $\vec{x} \in \fin{\theta_t}$ such that
    \begin{enumerate}
        \item there is $\gamma_t < \alpha_t$ such that $\tho_t (\gamma_t) = \tho_t (\alpha)$,
        \item $\varphi(\vec{x}) \in \tho_t (\alpha_t) \setminus \tho_{t} (\alpha)$ and $\lv_\varphi^{N_{\alpha_t}} (\vec{x}) = \gamma_t$, and
        \item \label{subclaim: 3}for all $t^\prime \geq_{T^\prime} t$, $\varphi(\vec{x}) \in \tho_{t^\prime}(\alpha_{t^\prime}) \setminus \tho_{t^\prime} (\alpha)$ and $\lv_\varphi^{N_{\alpha_{t^\prime}}} (\vec{x})= \pi_{t,t^\prime}(\gamma_t)$ and for all $r\Sigma_n$ formulas $\psi$ and $\vec{y} \in \fin{\theta_{t^\prime}}$ such that $\psi(\vec{y}) <^*_{t^\prime} \varphi(\vec{x})$, $\psi(\vec{y}) \in \tho_{t^\prime} (\alpha)$.
    \end{enumerate}
\end{subclaim}
\prf
    Let $s \in b$ be such that $\beta^\prime \in \ran(\pi_{s,b})$ and let $\bar{\beta}^\prime \in L_{\alpha_s}(\R^+)$ be such that $\pi_{s,b} (\bar{\beta}^\prime) = \beta^\prime$. Note that if we set $\gamma_s = \bar{\beta}^\prime$, then $\tho_s (\gamma_s) = \tho_s (\alpha)$ by the $r\Sigma_n$-elementarity of $\pi_{s,b}$. 

    We claim that there is an $r\Sigma_n$ formula $\varphi$ and $\vec{x} \in \fin{\theta_s}$ such that $\varphi(\vec{x}) \in \tho_s (\alpha_s) \setminus \tho_{s} (\alpha)$ and $\lv_\varphi^{N_{\alpha_s}} (\vec{x}) = \gamma_s$. Note that $L_{\beta^\prime} (\R^+)$ is not $\Sigma_n$-admissible, since $\kappa^{+M_b} < \theta$. But then also $L_{\gamma_s} (\R^+)$ is not $\Sigma_n$-admissible and thus there is an $r\Sigma_n$ formula $\psi$, $\delta \leq \kappa^{+L_{\gamma_s}(\R^+)}$, and $\vec{y} \in \fin{\kappa^{+L_{\gamma_s}(\R^+)}}$ such that
    \[
        L_{\gamma_s}(\R^+) \models \forall \alpha < \delta \exists z \psi(z,\alpha,\vec{y}),
    \]
    but there is no bound for this in $L_{\gamma_s}(\R^+)$. We may assume without loss of generality that $\delta = \kappa^{+L_{\gamma_s}(\R^+)}$. Now note that the statement $\varphi(\vec{y},\dot{\kappa^+})$ which says that there is some $\gamma$ such that for all $\alpha < \dot{\kappa^+}$ there is some subtheory $z$ of $\tho_{r\Sigma_{n-1}}(\gamma)$ which witnesses that $\psi(z,\alpha,\vec{y})$, is an $r\Sigma_n$-fact of $\leq^*$-rank $\gamma_s$ in $L_{\alpha_s}(\R^+)$. Moreover, $\varphi(\vec{y},\dot{\kappa^+})$ cannot be in $\tho_s (\alpha)$ by the definition of $\beta^\prime$.

    Set $s_0 := s$. If for all extensions $s^\prime$ of $s_0$ in $T^\prime$, \ref{subclaim: 3} holds we are done, so suppose that there is some $s_1 \geq_{T^\prime} s_0$ such that \ref{subclaim: 3} fails, i.e.~there is some $\gamma_{s_1} < \pi_{s_0,s_1} (\gamma_s)$ such that for some $r\Sigma_n$ formula $\varphi_1$ and $\vec{x}_1 \in \fin{\kappa^{+L_{\alpha_{s_1}}(\R^+)}}$ such that  $\varphi_1 (\vec{x}_1) \in \tho_{s_1} (\gamma_{s_1}) \setminus \tho_{s_1} (\alpha)$, $\lv^{N_{\alpha_{s_1}}}_{\varphi_1}(\vec{x}_1)= \gamma_{s_1}$. Let $s_1$ be the lexicographical least such node and let $\gamma_{s_1}$ be the least failure at $s_1$. If \ref{subclaim: 3}  holds of $s_1$ we are done, otherwise we let $s_2$ be the least node witnessing the contrary and let $\gamma_{s_2} < \pi_{s_1,s_2} (\gamma_{s_1})$ be the least ordinal witnessing the failure of \ref{subclaim: 3} at $s_1$. If \ref{subclaim: 3} holds at $s_2$ we are done, otherwise we continue as before. 

    Suppose for the sake of contradiction that this continues infinitely so that $\langle (s_k,\gamma_{s_k}) \mid k < \omega  \rangle$ is defined. Let $c := (s \restriction (\lh(s) - 1))^{\frown}  \langle (\alpha_{s_k},\gamma_{s_k}) \mid k < \omega \rangle$ and note that by construction $c$ is a branch through $T^\prime$. However, since $\gamma_s < \beta_s$, $c$ is left of $b$, a contradiction!
    \end{case}
    This finishes the proof of the subclaim.
\eprf

Let $t \in T^\prime$ be as in the subclaim. For $s \in (T^\prime)_t$, let $\gamma_s := \pi_{t,s} (\gamma_t)$. Note that for every $s \in (T^\prime)_t$, $\tho_s (\gamma_s) = \tho_s (\alpha)$, as otherwise there is a disagreement about the ordering $\leq^*$ between $N_\alpha$ and $\tho_t$.

By $\Sigma_n$-Collection, there is some $\xi_s < \alpha$ such that $\xi_s \in S^{N_\alpha}_{n-1}$ and $N_{\xi_s} \models \tho_s (\gamma_s)$. But this means that for every node $s \in (T)_t$, it is pruned at some stage $\xi < \alpha$ during the $(n+1)$-pruning process, or there is some $\xi < \alpha$ such that $\xi \in S^{N_\alpha}_{n-1}$ and $N_\xi \models \tho_s (\gamma_s)$. Since this is a disjunction of two $r\Sigma_n$-statements, it follows that once again by $\Sigma_n$-Collection, there is a uniform such $\xi$. This means that we can uniformly compute $\tho_{r\Sigma_n}^{N_\alpha} (\gamma_s)$ for $s \in (T_{\xi})_t$ over $N_\xi$. Note that we cannot compute $T^\prime_t$ over $N_\xi$ as there might be nodes $s \in (T_{\xi})_t$ that get pruned after stage $\xi$ in the $(n+1)$-pruning process. However, for such $s$ still $N_\xi \models \tho_s (\gamma_s)$.
However, in the successor step of the $(n+1)$-pruning process we assured that for every $\zeta < \theta$, there is an extension $s$ of $t$ in $T^\prime$ such that $\zeta < \theta_s$, so that $\sup \{ \gamma_s \mid s \in (T^\prime)_t \} = \theta$. But this means that $\tho_{r\Sigma_n}^{N_\alpha}(\theta)$ can be computed over $N_\xi$, a contradiction!
\eprf

As before, we have
\[
    \tho_{r\Sigma_n}^{N_\alpha}(\theta) \subsetneq  \tho_{r\Sigma_n}^{M_b}(\theta).
\]

But then, since $N_\alpha = \Hull{N_\alpha}{n} (\theta)$, it follows that there is a $\Sigma_{n-1}$-elementary embedding
\[
    \pi \colon N_\alpha \to M_b,
\]
that preserves $r\Sigma_n$ statements upwards. Then, since all proper initial segments of $N_\alpha$ are $\Sigma_n$-definable from parameters below $\theta$ and $\pi \restriction \theta+1 = \id$, we have that $\pi$ is the inclusion map. This means that $\text{wfc}(M_b) \supseteq L_{\alpha} (\R^+)$. From the minimality of $\alpha$ it then follows that $\text{wfc}(M_b) = L_{\alpha} (\R^+)$.

Regarding the definability of $\{ s \}$ from the parameter $\theta$ for $s \in b$, note that we may define $\{ s \}$ as follows: Note that $T$ is $\Sigma_1$-definable from the parameter $\theta$. Given $s^\prime \in T$, we have $s^\prime = s$ if and only if for all $t \in T$ such that $\lh (t) = \lh (s) = m$ and $t <_{\text{lex}} s$, there exists $\gamma$ such that $t \notin T_\gamma$, where $T_\gamma$ is a tree in the $(n+1)$-pruning process of $T$ and, for all $\gamma < \alpha$, $s \in T_\gamma$, i.e.~the node $s$ does not get removed during the $(n+1)$-pruning process.
\eprf

\setcounter{claimcounter}{0}
\setcounter{subclaimcounter}{0}

\begin{dfn} \label{dfn of branch}
    Let $\langle (\alpha_k, \beta_k) \rangle_{k < \omega}$ be left-most branch of the tree $T^\prime$ as in the proof of Theorem \ref{thm: branch theorem}. Let $\vec{p} := \langle \alpha_k \mid k < \omega \rangle$.
\end{dfn}

\section{The direct limit systems} \label{direct limit systems}

In this section, we will define a direct limit system $\F$ of iterates of $\mad$. We will also refer to this system as the \textit{external system}. We will then define over $\lx$ in a $\Sigma_1$-fashion a direct limit system $\tilde{\D}$, which we will refer to as the \textit{internal covering system}. The point is that $\lx$ can approximate $\F$ with $\tilde{\D}$, since by what we will show, $\F$ is in a certain sense dense in $\tilde{\D}$. It will follow that the direct limits derived from these systems agree.

In Section \ref{M infinitys version} and Section \ref{hod}, we will need a relativization of the definitions and results of this section to the context of $\lxg$. We will leave the straightforward adaption of the systems and the involved definitions to the context of $\lxg$ as an exercise for the reader.

Let us begin by introducing the relevant iterability notions.

\subsection{The relevant iterability notions} \label{section: relevant premice}

For a passive premouse $\M$ which models $\tho$ and an $n$-maximal iteration tree $\T$ on $\M$ of limit length, we define the structure $\Q(\T)$ as in Definition 2.4 of \cite{Steel_Woodin_2016}. Since we are working in a $1$-small context, this simply means that $\Q(\T) = L_\gamma(\M(\T))$ for some $\gamma < \OR$ such that there is $m < \omega$ such that $L_\gamma(\M(\T))$ is $m$-sound and $\rho_{m+1}^{L_\gamma (\M(\T))} < \delta(\T)$ or there is $A \in \Sigma_\omega^{L_\gamma (\M(\T))}(L_\gamma (\M(\T)))$ which witnesses a failure of $\delta(\T)$ being Woodin with respect to $\E^{\M(\T)}$.

\begin{dfn}
    Let $N$ be a passive premouse that models $\tho$. Let $\delta^N$ and $\mu^N$ be the unique Woodin cardinal, respectively, inaccessible greater than $\delta^N$ of $N$ and let $\theta^N = (\mu^N)^{+N}$. We set $N^- = N \vert (\delta^N)^{+N}$ and call $N^-$ the \textit{suitable part} of $N$.

    We write $\mathbb{B}^N$ for the $\omega$-generator extender algebra of $N$ at $\delta^N$, constructed using extenders $E \in \mathbb{E}^N$ such that $\nu_E$ is an $N$-cardinal.
    
    If $\mu^N = \kappa$, $\theta^N = \kappa^{+\lx}$ and $\OR^N = \alpha$, we say that $N$ is a \textit{pre-$\mad$-like $x$-weasel}. Moreover, if, moreover, $x$ is $(N,\mathbb{B}^N)$-generic, we say that $N$ is a \textit{good pre-$\mad$-like $x$-weasel}.
\end{dfn}

For a passive premouse $N$, that models $\tho$, by Lemma \ref{lem: n-1 projectum}, $\rho^N_{n-1} = \OR^N$, and by Lemma \ref{lem: largest cardinal}, $\rho^N_{n} = \theta^N$.

\begin{rem}
If $N$ and $P$ are good pre-$\mad$-like $x$-weasels such that $N^-, P^- \in \lx$, then, since the extender algebra may be absorbed into $\col(\omega, {<\kappa})$ and the definition of $\lr$ is homogeneous, there is $g^\prime$ which is $(N,\col(\omega,{<\kappa}))$-generic and $ g^{\prime \prime}$ which is $(P,\col(\omega,{<\kappa}))$-generic such that we have $L(\R^{N[g^\prime]}) = L(\R^{P[g^{\prime \prime}]})= \lr$.
\end{rem}

\begin{dfn}
    Let $N$ be a premouse that models $\tho$ and $\beta \geq \omega$. We say that $\T$ is a \textit{$n$-maximal $\beta$-wellfounded iteration tree on $N$} if $\T$ is defined as a $n$-maximal iteration tree with the exception that for $\alpha < \lh(\T)$, if $[0,\alpha]^\T \cap \D^\T = \emptyset$ and $\OR^{\M_\alpha^\T} \not \leq \beta$, then we only require that instead of full wellfoundedness $\beta \subseteq \wfc(\M_\alpha^\T)$.
\end{dfn}

\begin{dfn}
    Let $N$ be a passive premouse that models $\tho$, $\beta,\eta \geq \omega$, and $\T$ an $n$-maximal $\beta$-wellfounded iteration tree on $N$ of limit length less than $\eta$. Then $\T$ is \textit{$(\eta,\beta)$-short} if $\Q(\T)$ exists and $\Q(\T)\not \models \tho$; otherwise we say that $\T$ is \textit{$(\eta,\beta)$-maximal}.
    If $M = \M_\alpha^\T$ for some $\alpha < \lh(\T)$, we say that $M$ is a \textit{$\beta$-wellfounded $n$-maximal iterate of $N$}.
\end{dfn}

\begin{rem} \label{rem: short trees}
    Note that if $\Q(\T) \not \models \tho$, then clearly, for no $N \unlhd \Q(\T)$ such that $\M(\T) \unlhd N$, $N \models \tho$. By condensation, it follows that for all $N$ such that $N \unlhd \Q(\T)$, $N$ is not a model of $\tho$.

    Let $N$ be a pre-$\mad$-like $x$-weasel such that $N^- \in \lx \vert \kappa$ and suppose that $N$ is $(n,\omega_1,\omega_1 +1)^*$-iterable. Let $\T \in \lx$ be a $n$-normal iteration tree on $N$ of limit length less than $\kappa$. Then $\T$ is $(\kappa,\beta)$-short for all $\beta$ iff $\Q(\T) = J_\gamma (\M(\T))$ for some $\gamma < \kappa$.
\end{rem}

\begin{dfn}
    Let $N$ be a $(n-1)$-sound premouse, $\beta \geq \omega$, and $\T$ a $n$-maximal $\beta$-wellfounded iteration tree on $N$. Let $\beta < \OR$ and let $b$ be a cofinal non-dropping branch through $\T$. We say that $b$ is \textit{$\beta$-wellfounded} if $\Q(\T) = \Q(b,\T)$ and if $\OR^{\M_b^\T} \leq \beta$, then $\M_b^\T$ is wellfounded, and else, $\beta \subseteq \wfc(\M_b^\T)$. 
    
    If $b$ is a $\beta$-wellfounded cofinal branch, we say that $\M_b^\T$ is a \textit{$\beta$-wellfounded $k$-maximal iterate of $N$ (via the tree $\T^{\frown}b$)}.
\end{dfn}

\begin{dfn}
    Let $N$ be a passive premouse that models $\tho$ and $\eta, \beta \geq \omega$. We say that $N$ is \textit{$(\eta,\beta)$-normally-short-tree-iterable} if there is a function $f$ whose domain includes all $(\eta,\beta)$-short $n$-maximal $\beta$-wellfounded trees $\T$ on $N$, and for an $(\eta,\beta)$-short $n$-maximal $\beta$-wellfounded tree $\T$ on $N$, $f(\T) = b$, where $b$ is a non-dropping cofinal branch through $\T$ that is $\beta$-wellfounded.
\end{dfn}

\begin{dfn}
    Let $N$ be a passive premouse that models $\tho$ and $\eta, \beta \geq \omega$. Then a premouse $P$ is a \textit{$(\eta,\beta)$-pseudo-normal iterate of $N$} if $P \models \tho$, and there is a $n$-maximal $\beta$-wellfounded tree $\T$ on $N$ of length less than $\eta$ such that either $P$ is a model in $\T$, or there is a $\beta$-wellfounded cofinal branch $b$ of $\T$ such that $P = \wfc(\M^\T_{b})$, or $\T$ is $(\eta,\beta)$-maximal and $P = L_\gamma (\M(\T))$, where $\gamma < \OR$ is such that $L_\gamma (\M(\T)) \models \tho$.
\end{dfn}


\begin{rem}
    Let $N$ be a pre-$\mad$-like $x$-weasel such that $N^- \in \lx \vert \kappa$ and suppose that $N$ is $(n,\omega_1,\omega_1 +1)^*$-iterable. Let $\T \in \lx$ be a $(\kappa,\kappa)$-maximal tree on $N$. Let $P$ be the $(\kappa,\kappa)$-pseudo-normal iterate of $N$ given by $\T$. Then $P = L_\alpha (\M(\T))$ and $P^- \in \lx \vert \kappa$.
\end{rem}


\begin{dfn}
    Let $N$ be a passive premouse that models $\tho$ and $\eta, \beta \geq \omega$. Then an \textit{$(\eta,\beta)$-relevant finite pseudo-stack} on $N$ is a sequence $\langle \T_j \rangle_{j < k}$ for some $k < \omega$ such that there is a sequence $\langle N_j \rangle_{j < k}$ where $N_0 = N$, and for $j < k$, $N_j$ is a passive premouse which models $\tho$ and $\T_j$ is an $n$-maximal $\beta$-wellfounded iteration tree on $N_\beta$ of length less than $\eta$, and if $j + 1 < k$, then either $\T_j$ has successor length and is terminally non-dropping, i.e.~there is no drop in model on its main branch, and $N_{j + 1} = \M_\infty^{\T_j}$ or $N_{j+1} = \wfc(\M^{\T_j}_\infty)$, where $\M^{\T_j}_\infty$ is a $\beta$-wellfounded $n$-maximal iterate of $N_j$, or $\T_j$ is $(\eta,\beta)$-maximal and $N_{j+1} = L_\gamma (\M(\T_j))$ for some $\gamma < \OR$.
\end{dfn}

\begin{dfn}
    Let $N$ be a passive premouse that models $\tho$ and $\eta, \beta \geq \omega$. $P$ is a \textit{non-dropping $(\eta,\beta)$-pseudo-iterate} of $N$ if there is a $(\eta,\beta)$-relevant finite pseudo-stack $\langle \T_j \rangle_{j \leq k +1}$ on $N$ such that $\T_{k+1}$ has successor length and $b^{\T_{k+1}}$ does not drop in model or degree and $P = \M_\infty^{\T_{k+1}}$ or $P = \wfc(\M_\infty^{\T_{k+1}})$ and $\M_\infty^{\T_{k+1}}$ is $\beta$-wellfounded, or $\T_{k+1}$ is $(\eta,\beta)$-maximal and $P = L_\gamma (\M(\T_{k+1}))$, where $\gamma$ is such that $L_\gamma (\M(\T_{k+1})) \models \tho$.
\end{dfn}

\begin{dfn}
    Let $N$ be a passive premouse that models $\tho$ and $\eta,\beta \geq \omega$. Then $N$ is \textit{$(\eta,\beta)$-short-tree-iterable} if there is a function $f$ whose domain includes all sequences $\vec{\T} = \langle \T_\beta \rangle_{\beta \leq k+1}$ such that
    \begin{enumerate}
        \item $\langle \T_\beta \rangle_{\beta \leq k}$ is an $(\eta,\beta)$-relevant finite pseudo-stack that gives rise to a non-dropping $(\eta,\beta)$-pseudo iterate $P$, and
        \item $\T_{k+1}$ is an $(\eta,\beta)$-short $n$-maximal $\beta$-wellfounded iteration tree on $P$,
    \end{enumerate}
    and for such $\vec{\T}$, $f(\vec{\T}) = b$, where $b$ is a non-dropping $\beta$-wellfounded cofinal branch through $\T_{k+1}$.
\end{dfn}

\begin{dfn} \label{def: m-like}
    Let $N$ be a passive premouse that models $\tho$. Then $N$ is $\mad$\textit{-like} if
    \begin{enumerate}
        \item $N$ is $(\kappa,\kappa)$-short-tree-iterable, and
        \item \label{def: m-like two} every non-dropping $(\kappa,\kappa)$-pseudo-iterate of $N$ is in fact a non-dropping $(\kappa,\alpha)$-pseudo-normal iterate of $N$.
    \end{enumerate}
\end{dfn}

\begin{rem}
    We will use the notion of $\mad$-like in $\lx$ and related models so that $\kappa$ is a fixed parameter and therefore does not appear in the terminology. 
    Note that by \cite{schlutzenberg2024normalizationtransfinitestacks} every non-dropping iterate of $\mad$ is given by a tree of length less than $\kappa$ is $\mad$-like. Moreover, since $\mad$ is $(n, \omega_1, \omega_1+1)^*$-iterable, it is also $(\omega_1,\gamma)$-short-tree-iterable for all $\gamma < \OR$.

\end{rem}

\begin{dfn}
    If $N$ is a pre-$\mad$-like $x$-weasel that is $\mad$-like, we say that $N$ is a \textit{$\mad$-like $x$-weasel}. Moreover, if $x$ is $(N,\mathbb{B}^N)$-generic, then we say that $N$ is a \textit{good $\mad$-like $x$-weasel}.
\end{dfn}

\begin{rem}
    Let $N$ be a $\mad$-like $x$-weasel such that $N^- \in \lxg$ and let $\T$ be a $(\kappa,
    \kappa)$-short tree on $N$. Then there is a unique $\kappa$-wellfounded cofinal branch $b$ through $\T$.
\end{rem}

\begin{dfn} \label{dfn: tree leading to P}
    Let $N$ and $P$ be $\mad$-like. We write $N \dashrightarrow P$ if $P$ is a non-dropping $(\kappa,\kappa)$-pseudo-normal iterate of $N$ and denote the $n$-maximal $\kappa$-wellfounded tree leading from $N$ to $P$ by $\T_{NP}$. In the case where $\T_{NP}$ is $(\kappa,\kappa)$-maximal and there is a cofinal $\kappa$-wellfounded branch $b$ through $\T_{NP}$, we let $\T_{NP}$ include this branch.
\end{dfn}

It is easy to see that $\T_{NP}$ is unique, so this is well-defined.

\begin{lem}
    $\dashrightarrow$ is a partial order on the set of $\mad$-like premice.
\end{lem}

\prf
    Reflexivity and anti-symmetry are clear. Transitivity follows from \ref{def: m-like two}.~of Definition \ref{def: m-like}.
\eprf

\setcounter{claimcounter}{0}

\subsection{The external direct limit system}

\begin{dfn}
    Let $\F$ be the set of all iterates $N$ of $\mad$ via $\Sigma^{\mad}$ such that $N$ is a good pre-$\mad$-like $x$-weasel and $N^- \in \lx \vert \kappa$.
\end{dfn}

\begin{lem} \label{comparison in F}
    Let $N,P \in \F$. Then there is $Q \in \F$ such that $N \dashrightarrow Q$ and $P \dashrightarrow Q$.
\end{lem}

\prf
    Let us define $\T$ on $N^-$ and $\U$ on $P^-$ recursively as follows: At successor steps we follow the standard process of iterating away the least disagreement. If the there is no more disagreement at a successor step, we stop the process.
    If we reach a limit stage less than $\kappa$, we distinguish two cases. The first case is that $\Q(\T),\Q(\U) \in \lx$. Then, since $\lx \vert \kappa$ is a ZFC-model and so $\Sigma_1^1$-absolute, we can run the standard argument\footnote{See for example the proof of Lemma 3.10 in \cite{Steel_Woodin_2016}} to see that the unique cofinal wellfounded branches of $\T$ and $\U$ are in $\lx \vert \kappa$.
    The other case is that either $\Q(\T) \notin \lx$ or $\Q(\U) \notin \lx$. Let us assume without loss of generality that $\Q(\T) \notin \lx$. Note that this means that $\M(\T) = \M(\U) \in \lx \vert \kappa$. However, it is then easy to see that if $c$ is the cofinal wellfounded branch through $\T$ according to $\Sigma^{\mad}$, then $\M_c^\T = L_\alpha(\M^\T)$, and $L_\alpha(\M^\T)$ is a pre-$\mad$-like $x$-weasel. 
    Note that by the standard argument the process cannot last longer than $\eta+1$-many steps, where $\eta = \max\{\delta^N,\delta^P\}$.

    Let $R := L_\alpha(\M^\T)$. Working in $\lx \vert \kappa$, let $\T$ on $R^-$ be the $x$-genericity iteration at $\delta^R$ of $R$, i.e.~the iteration tree constructed in the proof of Theorem 7.14 of \cite{Steel2010}. Quite similar arguments as before show that if $Q$ is the iterate given by $\T$, $Q \in \F$.
\eprf

\begin{cor}
    $(\F,\dashrightarrow)$ is a directed partial order.
\end{cor}

\begin{lem}
    $\F \neq \emptyset$.
\end{lem}

\prf
    Let $\T$ be the $x$-genericity iteration at $\delta^{\mad}$ of $\mad$. A similar argument as in the proof of Lemma \ref{comparison in F} shows that $\M_\infty^\T \vert (\delta^{\M_\infty^\T})^{+\M_\infty^\T} \in \lx \vert \kappa$. Thus, $\kappa$ remains inaccessible in $\M_\infty^\T [x]$. Moreover, by Lemma \ref{lem:General Preservation of KP}, $\M_\infty^\T \models \tho$. By Lemma \ref{lem: preservation of kp}, it follows that $\M_\infty^\T [x] \models \Sigma_n\text{-}\kp$.
    We claim $\alpha^\prime := \OR^{\M_\infty^\T} = \alpha$. Let us first assume that $\alpha^\prime < \alpha$. Then $\M_\infty^\T [x] \models \Sigma_n\text{-}\kp \land ``\kappa \text{ is inaccessible}" \land ``\kappa^+ \text{exists}"$, and $\alpha^\prime < \alpha$. This contradicts the minimality of $\alpha$. Now, let us suppose that $\alpha^\prime > \alpha$. Note that $\M_\infty^\T [x] \vert \alpha \models \Sigma_n\text{-}\kp \land ``\kappa \text{ is inaccessible}" \land ``\kappa^+ \text{exists}"$, since $\M_\infty^\T \vert (\delta^{\M_\infty^\T})^{+\M_\infty^\T} \in \lx \vert \kappa$. However, this means that $\M_\infty^\T  \not \models \tho$, a contradiction.
\eprf

Note that for $N,P,Q \in \F$ such that $N \dashrightarrow P \dashrightarrow Q$, we have
\[
    i^{\F}_{N,Q} = i^{\F}_{P,Q} \circ i^{\F}_{N,P},
\]
where $i^{\F}_{N,P} \colon N \to P$ is the embedding given by the iteration tree $\T_{NP}$ and $\Sigma^{\mad}$, and likewise $i^{\F}_{P,Q}$ and $i^{\F}_{N,P}$. We may define
\[
    M_\infty^{\F} = \dir \lim \langle P,Q;i^{\F}_{P,Q} \mid P,Q \in \F \text{ with } P \dashrightarrow Q \rangle
\]
and let $i^{\F}_{P\infty}$ be the direct limit map. By \cite{schlutzenberg2024normalizationtransfinitestacks}, $M_\infty^{\F}$ is in fact a normal non-dropping countable iterate of $\mad$ and thus is a model of the theory $\tho$ and is $(\omega_1,\omega_1)$-short-tree-iterable (in $V$). Moreover, for each $P \in \F$, $i^{\F}_{P\infty}$ is given by the iteration map according to $\Sigma^{\mad}$.

\begin{dfn}
    Let $N \in \F$ and $s \in (\fin\alpha \setminus \{ \emptyset \})$. Then $N$ is $s$\textit{-stable} if for all $P \in \F$ such that $N \dashrightarrow P$  we have $i^{\F}_{N,Q} (s) = s$.
    \end{dfn}
    
    The proof of the following lemma is as in the $L[x,G]$-case as presented in \cite{Steel_Woodin_2016}.
    
    \begin{lem} \label{existence of s-stable}
        Let $N \in \F$. Then for all $s \in (\fin\alpha \setminus \{ \emptyset \})$ there is $P \in \F$ such that $N \dashrightarrow P$ and $P$ is $s$-stable.
    \end{lem}

\section{The internal covering system}

We are now going to define the internal covering system $\tilde{\D}$. We will first introduce the notion of $s$-iterability in order to state the definition of $\tilde{\D}$. We then show that $\tilde{\D}$ is $\Sigma_1$-definable over $\lx$ from the parameters $\theta$ and $\R^{\lx}$. Finally, we will show that $\F$ is in a certain sense dense in $\tilde{\D}$ and $\tilde{\D}$ is correct enough to approximate $M^{\F}_\infty$ correctly. The definitions and lemmas in this section are mostly adaptions of \cite{Steel_Woodin_2016}. However, in \cite{Steel_Woodin_2016}, the authors do not need to worry much about the complexity of the direct limit, so that, for example, the detailed analysis of Subsection \ref{definability of direct limit} is not necessary in the $L[x,G]$-case.

Let us begin with the definition of $s$-iterability.
\begin{dfn}
    For an ordinal $\beta$ let $\finite(\beta) = \fin{\beta} \setminus \{ \emptyset \}$ and for $s \in \finite(\beta)$ let $s^- = s \setminus \{ \max(s) \}$.
\end{dfn}

\begin{dfn}
    Let $N$ be a passive premouse that models $\tho$ and let $s \in \finite(\OR^N)$ be such that $\delta^N < \max(s)$. Set 
    \[
        \gamma^N_s := \sup (\delta^N \cap \Hull{N \vert \max (s)}{\omega}(
        \{ s^- \}))
    \]
and
    \[
        H^N_s = \Hull{N \vert \max(s)}{\omega}(\gamma^N_s \cup \{ s^- \}).
    \]

    Let $\mathcal{L}_s$ be the language of set theory together with the set of constant symbols $\{ \dot{\alpha} \}_{\alpha \in s^-}$ and let $\fml (\mathcal{L}_s)$ be the set of formulas in the language $\mathcal{L}_s$. For $\alpha \in s^-$, let $\dot{\alpha}^N = \alpha$ and set
    \[
        \tho^N_s = \{ \langle \varphi , t \rangle \colon \varphi \in \fml (\mathcal{L}_s), t \in \fin{\delta^N} \text{, and } N \vert \max(s) \models \varphi[t]\}.
    \]

\end{dfn}

Note that via coding $\tho^N_s \subseteq \delta^N$. A standard argument shows the following.

\begin{lem}
    Let $N$ be a passive premouse that models $\tho$ and $s \in \finite (\OR^N)$ such that $\delta^N < \max(s)$. Then 
    \[
        \gamma^N_s = \sup(H_s^N \cap \delta^N).
    \]

\end{lem}


\begin{dfn} \label{def: s-iterability}
    Let $N$ be a pre-$\mad$-like $x$-weasel such that $N^- \in \lx \vert \kappa$. Let $s \in \finite(\alpha)$ be such that $\delta^N < \max(s)$. Then $N$ is \textit{$s$-iterable} if for all pre-$\mad$-like $x$-weasel $P_1,P_2,P_3$ such that $P_1^-, P_2^-, P_3^- \in \lx \vert \kappa$ and $N \dashrightarrow P_1 \dashrightarrow P_2 \dashrightarrow P_3$, letting $\T_{ij} = \T_{P_i P_j}$, we have that $\col(\omega, {< \kappa})$ forces the following statements over $\lx$:
    \begin{enumerate}
        \item \label{s-iterability one} there is a $\T_{12}$-cofinal branch $b$ which respects $s$ in the sense that  $\delta^{P_2} \in \wfp (\M_b^{\T_{12}})$, $b$ does not drop, and $i^{\T_{12}}_b (\tho_s^{P_1})=\tho_s^{P_2}$, and
        \item \label{s-iterability two} whenever $b_{12}, b_{23}, b_{13}$ are $\T_{12},\T_{23},\T_{13}$-cofinal branches respectively which respect $s$, we have 
        \[
            i^{\T_{13}}_{b_{13}} \restriction \gamma_s^{P_1} = i^{\T_{23}}_{b_{23}} \circ  i^{\T_{12}}_{b_{12}} \restriction \gamma_s^{P_1}.
        \]
    \end{enumerate}
\end{dfn}

It is easy to see that the following holds.

\begin{lem} \label{stable implies iterability}
    Let $N \in \F$. If $N$ is $s$-stable, then $N$ is $s$-iterable.
\end{lem}

\begin{dfn}
    Let $\tilde{\D}$ be the set of all $(N,s)$ such that $N$ is a pre-$\mad$-like $x$-weasel such that $N^- \in \lx \vert \kappa$,  $s \in \finite(\alpha)$, and such that in $\lx$ the following holds:
    \begin{enumerate}
        \item $N \text{ is } \mad\text{-like}$,
        \item $\delta^{L_\alpha(N)} < \max(s)$, and
        \item $N \text{ is }s\text{-iterable}$.
    \end{enumerate}
    For $(N,s),(P,t) \in \tilde{\D}$, let 
    \[
        (N,s) \leq (P,t) \text{ iff } N \dashrightarrow P \text{ and } s \subset t.
    \]
\end{dfn}

\subsection{The definability of the internal covering system}

We now want to show that $\tilde{\D}$ is $\Sigma_1$-definable over $\lx$ from the parameters $\theta$ and $\R^{\lx}$.

\begin{lem}
    The set of all $N$ that are pre-$\mad$-like $x$-weasels such that $N^- \in \lx \vert \kappa$ is definable over $\lx \vert \kappa$.
\end{lem}

\prf Let $A \in \lx \vert \kappa$. Clearly, since $\kappa$ is a limit cardinal of $\lx$, it is definable over $\lx \vert \kappa$ that $A$ is a premouse with a Woodin cardinal $\delta^A$. Moreover, by condensation $(\delta^A)^{+L_\kappa(A)} = (\delta^A)^{+L_\alpha(A)}$, so that the condition that $\OR^A = (\delta^A)^{+L_\alpha(A)}$ is also definable over $\lx \vert \kappa$. Also, it is definable over $\lx \vert \kappa$ that for all $\gamma < \kappa$, $L_\gamma (A) \not \models \tho$. Note that since $A \in \lx$, we have that $N := L_\alpha(A)$ is $\Sigma_n$-admissible. We claim that for all $\beta < \alpha$, $N \vert \beta \not \models \tho$. This is already true for $\beta \leq \kappa$ by assumption. Let us suppose for the sake of contradiction that there is $\beta \in (\kappa, \alpha)$ such that $N \vert \beta \models \tho$. Then $N \models \exists \gamma (\gamma > \delta^N \land N \vert \gamma \models \tho)$, which is a $\Sigma_1$ statement with parameter $\delta^N$. Let $H := \Hull{N}{n} (\delta^N + 1)$. Let $\pi \colon \bar{N} \to H$ be the inverse of the transitive collapse map. Note that $\bar{N} \lhd N \vert \kappa$. However, then there is some $\gamma \in (\delta^N, \kappa)$ such that $N \vert \gamma \models \tho$, a contradiction! Thus, $L_\alpha (A) \models \tho$.
\eprf

Next, we show with Lemma \ref{definability of short tree} and Lemma \ref{lem: absoluteness between lx and v} that the notion of $(\kappa,\kappa)$-short-tree-iterability, which is part of the definition of $s$-iterability, is definable over $\lx \vert \kappa$.

\begin{lem} \label{definability of short tree}
    Let $N$ be a pre-$\mad$-like $x$-weasel such that $N^- \in \lx \vert \kappa$ and $\T \in \lx \vert \kappa$ be a $n$-maximal iteration tree on $N$ of limit length less than $\kappa$. Then for a non-dropping cofinal branch $b$ of $\T$ that is in $\lx$ the following are equivalent:
    \begin{enumerate}
        \item \label{lem: wellfoundedness one} $\alpha \subseteq \wfc(\M_b^\T)$,
        \item \label{lem: wellfoundedness two}  $i_b^\T (\kappa)  \subseteq \wfc(\M_b^\T)$, and
        \item \label{lem: wellfoundedness three}  $\kappa  \subseteq \wfc(\M_b^\T)$.
    \end{enumerate}
\end{lem}

\prf
    Clearly, \ref{lem: wellfoundedness one} implies \ref{lem: wellfoundedness two}, since $i_b^\T (\kappa) < \alpha$, and \ref{lem: wellfoundedness two} implies \ref{lem: wellfoundedness three}. To show that \ref{lem: wellfoundedness three} implies \ref{lem: wellfoundedness one}, let $b \in \lx \vert \kappa^{+\lx}$ be a non-dropping cofinal branch of $\T$ such that $\kappa \subset \wfc(\M^\T_b)$.
    Suppose for the sake of contradiction that there is $\beta < \alpha$ such that $\beta \notin \wfc(\M^\T_b)$ and suppose that $\beta$ is the least such. Note that since $\theta$ is regular in $\lx$, there is $N^\prime \lhd N$ such that if we consider $\T$ on $N^\prime$ as $\T^\prime$, $\M_b^{\T^\prime}$ is not wellfounded. Since $\M_b^{\T^\prime} \in \lx$, there is $\langle \beta_k \mid k < \omega \rangle \in \lx$ such that $i^\T_{i j} (\beta_i) < \beta_j$ for $i < j < \omega$.

    Let $\xi = i_b^{\T^\prime} ((\delta^N)^{+N})$. Note that $\lambda < \xi < \kappa$.
    Let $X \prec_{1000} \lx \vert \theta$ be such that $\card(X) < \kappa$ and $\{ N \vert \beta_0, \T^\prime, b \} \cup \{ \langle \beta_n \mid n < \omega \rangle \} \cup N \vert (\xi + \omega) \in X$. Let $\pi \colon M \to \lx \vert \theta$ be the inverse of the transitive collapse map of $X$, so that $\pi$ is $\Sigma_{1000}$-elementary. Let $\{ \bar{N} \vert \bar{\beta}_0,\bar{\T}, \bar{b} \} \cup \{ \bar{\beta}_n \mid n < \lambda \} \in M$ be such that $\pi((\bar{N} \vert \bar{\beta}_0,\bar{\T},\bar{b})) = (N \vert \beta_0, \T^\prime,b)$ and $\pi(\bar{\beta}_n) = \beta_n$ for all $n < \lambda$. Note that $\M_b^\T \vert i_{b}^\T (\OR^{\M_0^{\bar{\T}}}) = \M_b^{\bar{\T}}$. But then $\M_b^{\T}$ is illfounded below $\kappa$, a contradiction!
\eprf

A similar argument for $n$-maximal iteration trees of successor length gives:

\begin{lem} \label{lem: absoluteness between lx and v}
    Let $N$ be pre-$\mad$-like $x$-weasel such that $N^- \in \lx \vert \kappa$. Let $\T \in \lx $ be a putative $n$-maximal iteration tree on $N$ such that $\lh(\T) = \lambda + 1 < \kappa$. Let $b = [0,\lambda]^\T$. Then the following are equivalent:
    \begin{enumerate}
        \item there is a drop in model along $b$ and $\M_\lambda^\T$ is wellfounded and has height less than $\kappa$, or $b$ is non-dropping and $\alpha \subseteq \wfc(\M_\lambda^\T)$, and
        \item there is  a drop in model along $b$ and $\M_\lambda^\T$ is wellfounded and has height less than $\kappa$, or $b$ is non-dropping and $\kappa \subseteq \wfc(\M_\lambda^\T)$.
    \end{enumerate}
\end{lem}

Since \ref{def: m-like two}.~of Definition \ref{def: m-like} for non-dropping $(\kappa,\kappa)$-pseudo-iterates which are in $\lx \vert \kappa$ is easily seen to be definable over $\lx \vert \kappa$, we have the following.

\begin{cor}
    The set of all $N$ that are pre-$\mad$-like $x$-weasels such that $N^- \in \lx \vert \kappa$ and such that $\lx \models N \text{ is }\mad\text{-like}$ is definable over $\lx \vert \kappa$.
\end{cor}

Now it follows almost immediately that the notion of $s$-iterability is locally definable over $\lx$ for a fixed $s \in \finite(\alpha)$.

\begin{lem} \label{lem: definability of s-iterability}
    Let $N$ be a pre-$\mad$-like $x$-weasel such that $N^- \in \lx \vert \kappa$, and $s \in \finite(\alpha)$ such that $\theta < \max(s)$. Then the following are equivalent:
   \begin{enumerate}
       \item $\lx \models ``N^- \text{ is }s\text{-iterable}"$, and
       \item \label{point two}$\lx \vert (\max(s)+\omega) \models ``N^- \text{ is }s\text{-iterable}"$.
   \end{enumerate}

\end{lem}

This is straightforward since for any $H$ which is $(\lx,\col(\omega, {<\kappa}))$-generic, $\lx[H]$ and $L_{\max(s)+\omega}[x][H]$ have the same set of reals.

\begin{cor}
    For $N$ a pre-$\mad$-like $x$-weasel such that $N^- \in \lx \vert \kappa$, and $s \in \finite(\alpha)$ such that $\theta < \max(s)$ the statement $``N \text{ is }s\text{-iterable}"$ is $\Sigma_1$-definable over $\lx$ in parameters $\{N,s,\R^{\lx}, \theta \}$.
\end{cor}

Note that we need the parameter $\R^{\lx}$ in order to quantify in a bounded way over all possible branches that respect $s$.

\begin{lem} \label{definability of the covering system}
    $\tilde{\D}$ is $\Sigma_1$-definable over $\lx$ from the parameters $\theta$ and $\R^{\lx}$.
\end{lem}

\subsection{The relation between the internal and the external system}

\begin{lem} \label{lem: absoluteness of short tree iterability}
    Let $N \in \F$. Then $\lx \models ``N\text{ is } (\kappa,\kappa) \text{-short-tree-iterable}"$.
\end{lem}

\prf
    By Lemma \ref{lem: absoluteness between lx and v}, it suffices to see that for a $(\kappa,\kappa)$-short tree $\T \in \lx$ on $N$, there is a branch $b \in \lx$ such that $\kappa \subseteq \wfc(\M_b^\T)$. By Remark \ref{rem: short trees}, $\Q(\T) = J_\gamma (\M(\T))$ for some $\gamma < \kappa$, so that $\Q(\T) \in \lx \vert \kappa$.

    It is easy to see that $\T$ is $(\kappa,\kappa)$-short in $V$. Thus, there is a cofinal wellfounded branch $b \in V$ such that 
    \[
        \Q(b,\T) = \Q(\T).
    \]
    Let $h$ be $(\lx \vert \theta,\col(\omega,\kappa))$-generic. In $(\lx \vert \theta) [h]$, $\T, N \vert (\delta^{N+})^{+N}$, and $\Q(\T)$ are countable. Moreover, $(\lx \vert \theta) [h]$ is a $\zf^-$-model. Thus, by $\Sigma_1^1$-absoluteness there is a cofinal branch $c \in (\lx \vert \theta) [h]$ through $\T^\prime$ such that 
    \[
        \Q(c,\T) = \Q(\T).
    \]
    But this implies that $b =c$ and therefore, $b \in (\lx \vert \theta) [h]$. However, $h$ was arbitrary, and thus, by Solovay's Lemma, $b \in \lx \vert \theta$. So $b \in \lx$.
    \eprf

    \begin{cor}
        Let $s \in \finite(\alpha)$ and $N \in \F$ be $s$-stable. Then $(N,s) \in \tilde{\D}$.
    \end{cor}

\begin{lem} \label{lem: comparison lemma}
    Let $N$, $P$ and $s,t \in \finite(\alpha)$ be such that $(N,s),(P,t) \in \tilde{\D}$ and $\max\{s,t\} > \mu$, where $\mu$ is the cardinal successor of $\max \{\delta^N,\delta^P\}$ in $\lx$. Then there is $R$ such that $(R,s\cup t) \in \tilde{\D}$, and $(N,s) \leq (R,s\cup t)$ and $(P,t) \leq (R,s\cup t)$. Moreover, $\delta^R \leq \mu < \kappa$.
\end{lem}

\prf
Let $Q \in \F$ be $s \cup t$-stable and work in $\lx$. Let $\T$ on $N$, $\U$ on $P$, and $\V$ on $Q$ result from the standard process of iterating away the least disagreement at successor steps, and from choosing according to the $(\kappa,\kappa)$-short-tree-strategies for $N$,$P$, and $Q$ at limit steps. By the same argument as in the proof of \ref{comparison in F}, the process cannot last $\mu + 1$-many steps. Moreover, since $N$,$P$, and $Q$ are sufficiently iterable in $\lx$, the process terminates. 

Note that there are three ways in which the process can terminate. 
The first case is that the process stops at a limit stage, and both trees are $(\kappa,\kappa)$-short. In this case, we either have fully wellfounded cofinal branches through $\T$ and $\U$, and we may then argue as in the proof of Lemma \ref{comparison in F}, or there is a branch which is $\kappa$-wellfounded. Let us suppose without loss of generality that $\T$ does not have a fully wellfounded branch. In this case, by Lemma \ref{definability of short tree}, the $\kappa$-wellfounded branch $b$ of $\T$ is $\alpha$-wellfounded. However, then by the proof of Lemma \ref{lem:General Preservation of KP} and the fact that $L_\alpha(\M^\T_b \vert \kappa) \models \tho$, it follows that $R := \M^\T_b \vert \alpha = \wfc(\M^T_b)$, so that $R$ is a $(\kappa,\kappa)$-pseudo normal iterate of $N$,$P$ and $Q$. Then it is easy to see that $(R,s\cup t) \in \tilde{\D}$.

The second case is that there is no more disagreement at a successor step. In this case, the final iterate $R$ is $\kappa$-wellfounded. But then again $R$ is $\alpha$-wellfounded and then much as in the proof of Lemma \ref{comparison in F}, $R$ is a $\mad$-like $x$-weasel in $\lx$ and $\delta^R < \mu < \max(s\cup t)$ and $R$ is $s\cup t$-iterable.

The third case is that the process stops at a limit stage, and both trees are $(\kappa,\kappa)$-maximal. Let
\[
    R := L_{\alpha^\prime}(\M(\T)) = L_{\alpha^\prime}(\M(\U)),
\]  
where $\alpha^\prime$ is such that $L_{\alpha^\prime}(\M(\T)) \models \tho$. We aim to see that $\alpha^\prime = \alpha$. Let us first suppose, for the sake of contradiction, that $\alpha^\prime > \alpha$. Much as before, we have $L_\alpha(\M(\T)) \models \tho$, since $\M(\T) \in \lx \vert \kappa$. This contradicts the fact that $R \models \tho$. Now, let us suppose, for the sake of contradiction, that $\alpha^\prime < \alpha$. Note that since $Q$ is an iterate of $\mad$, there is by Lemma \ref{lem:General Preservation of KP}, a cofinal wellfounded branch $b$ leading from $Q$ to $R$ and an $n$-embedding $i^\V_b \colon Q \to R$ in $V$. However, we may then derive a contradiction in the same way as in the proof of Lemma \ref{comparison in F}.
\eprf

\begin{lem} \label{x genericity iteration}
    Let $N$ and $s \in \finite(\alpha)$ be such that $(N,s) \in \tilde{\D}$ and $\max\{s,t\} > \zeta$, where $\zeta$ is the cardinal successor of $\delta^N$ in $\lx$. Then there is $R$ such that $(R,s) \in \tilde{\D}$, and $(N,s) \leq (R,s)$ and $x$ is generic over $R$ for the extender algebra at $\delta^R$. Moreover, $\delta^R \leq \zeta < \kappa$.
\end{lem}

\prf
    We work in $\lx$. Let $\T$ be the $x$-genericity iteration at $\delta^N$ of $N$, i.e.~the iteration tree constructed in the proof of Theorem 7.14 of \cite{Steel2010}. We argue much as in the proof of Lemma \ref{lem: comparison lemma}. The only difference is in the third case of that proof, i.e.~the process is of limit length and reaches a $(\kappa,\kappa)$-maxmial tree. Let $R = L_{\alpha^\prime}(\M(\T))$, where $\alpha^\prime$ is such that $L_{\alpha^\prime}(\M(\T)) \models \tho$. By Lemma \ref{lem: preservation of kp}, $\Sigma_n$-KP is preserved by forcing with the extender algebra over $R$. Moreover, $\mu^R$ remains inaccessible in $R[x]$. Thus, since $R^- \in \lx$, $L_{\alpha^\prime} (R^-) [x]$ models
    \[
        \Sigma_n\text{-}\kp \land \exists \kappa (``\kappa \text{ is inaccessible and }\kappa^+ \text{ exists"}).
    \]
    This means that $\alpha \leq \alpha^\prime$ by the minimality of $\alpha$. Suppose for the sake of contradiction that $\alpha < \alpha^\prime$. Note that since $\delta^R < \kappa$, and $\kappa$ is inaccessible in $R[x] \vert \alpha = \lx$, $\kappa$ is inaccessible in $R \vert \alpha$ and $\kappa^{+\lx}$ is a cardinal in $R \vert \alpha$. However, this means that $R \models \tho$, a contradiction! Thus, $\alpha = \alpha^\prime$. But then it follows that $\kappa = \mu^R$ and $\kappa^{+\lx} = \theta^R$, so that $R$ is a good $\mad$-like $x$-weasel and $R^- \in \lx \vert \kappa$.
    \eprf

    \begin{cor}
        $(\tilde{\D},\leq)$ is a directed partial order.
    \end{cor}



\begin{cor} \label{lem: suitable iterates}
    Let $N$, $P$ and $s,t \in \finite(\alpha)$ be such that $(N,s),(P,t) \in \tilde{\D}$ and $\max\{s,t\} > \mu$, where $\mu$ is the cardinal successor of $\max \{\delta^N,\delta^P\}$ in $\lx$. Then there is $R$ such that $(R,s\cup t) \in \tilde{\D}$, and $(N,s) \leq (R,s\cup t)$ and $(P,t) \leq (R,s\cup t)$ and $x$ is generic over $R$ for the extender algebra at $\delta^R$. Moreover, $\delta^R \leq \mu < \kappa$.
\end{cor}

\prf
    We first compare as in Lemma \ref{lem: comparison lemma} to arrive at a common pseudo-normal iterate $R$ and then do a genericity iteration of $R$ as in Lemma \ref{x genericity iteration}.
\eprf

We now derive a direct limit from $\tilde{\D}$ as follows. 

\begin{dfn} \label{def: embeddings}
    Let $(N,s),(P,t) \in \tilde{\D}$ be such that $(N,s) \leq (P,t)$. We denote by 
    \[
        i^{\tilde{\D}}_{(N,s),(P,t)} \colon H^{N}_s \to H^{P}_t
    \]
    the embedding such that if $a \in H^{N}_s$ and $\varphi$ is a $r\Sigma_n$ formula and $\vec{y} \in \fin{\gamma_s^N}$ such that $\tau_\varphi^{N \vert \max(s)} (\vec{y},\{s^-\}) = a$, then $i^{\tilde{\D}}_{(N,s),(P,t)} (a) = \tau^{P \vert \max(t)}_\varphi (i_b^{\T_{NP}}(\vec{y}),\{s^-\})$, where $b$ is a $\T_{NP}$-cofinal branch which respects $s$ in a $\col(\omega,{<\kappa})$-extension of $\lx$. 
\end{dfn}

Using Condition \ref{s-iterability two} of Definition \ref{def: s-iterability}, it is straightforward to check that the map is well-defined and unique. Moreover, $i^{\tilde{\D}}_{(N,s),(P,t)}$ is a $\Sigma_0$-elementary embedding.

\begin{lem}
    Let $(N,s), (P,t), (R,u) \in \tilde{\D}$ be such that $(N,s)\leq (P,t)\leq (R,u)$. Then
    \[
        i^{\tilde{\D}}_{(N,s),(R,u)} = i^{\tilde{\D}}_{(P,t),(R,u)} \circ i^{\tilde{\D}}_{(N,s),(P,t)}.
    \]
\end{lem}

    \setcounter{claimcounter}{0}
    \setcounter{subclaimcounter}{0}

\begin{lem} \label{lem: stable is dense}
    The set $\{ (P,s) \in \tilde{\D} : P \in \F \land s \in \finite(\alpha) \land P\text{ is }s\text{-stable} \}$ is dense in $\tilde{\D}$.
\end{lem}

\prf
    Let $(P,s) \in \tilde{\D}$. By \ref{existence of s-stable}, there is $N \in \F$ such that $N$ is $s$-stable. By Corollary \ref{lem: suitable iterates}, there is a good $\mad$-like $x$-weasel $R$ such that $R^- \in \lx \vert \kappa$, $N \dashrightarrow R$, and $P \dashrightarrow R$. Since $N \dashrightarrow R$, $R$ is $s$-stable and thus $s$-iterable by Lemma \ref{lem: absoluteness of short tree iterability} and Lemma \ref{stable implies iterability}. Since $P \dashrightarrow R$, it follows that $(P,s) \leq (R,s)$. 
\eprf

We can establish a bit more similarity between $\F$ and the direct limit system derived from $\tilde{\D}$ as the following lemma shows that for $N,P \in \F$ which are $s$-stable, the map $i^{\tilde{\D}}_{(N,s),(P,t)}$ approximates the map $i^\F_{N,P}$.

\begin{lem} \label{lem: agreement of maps}
    Let $(N,s),(P,t) \in \tilde{\D}$ be such that $(N,s) \leq (P,t)$, $N,P \in \F$, and $N$ is $s$-stable. Then
    \[
        i^{\F}_{N,P} \restriction H^{N}_s = i^{\tilde{\D}}_{(N,s),(P,t)}.
    \] 
\end{lem}

\begin{dfn} \label{def: internal direct limit system}
    Let 
    \[
        M_\infty^{\tilde{\D}} = \dir \lim \langle H^{N}_s, H^{P}_t ; i^{\tilde{\D}}_{(N,s),(P,t)} : (N,s), (P,t) \in \tilde{\D} \text{ and } (N,s)\leq (P,t) \rangle
    \]
    and let $i^{\tilde{\D}}_{(N,s)\infty} \colon H^{N}_s \to M_\infty^{\tilde{\D}}$ be the ($\Sigma_0$-elementary) direct limit maps.
\end{dfn}

We now aim to establish $M^\F_\infty = M^{\tilde{\D}}_\infty$. Before we do this, we need to proof some properties about the sequence $\vec{p}$ which we constructed in Section \ref{section: finding the sequence}.

\subsection{The generating fixed points} \label{section: generating fixed points}

\begin{dfn}
    Let $\vec{\gamma} = \langle \gamma_k \mid k < \omega \rangle$ be such that 
    \[
        \gamma_k = \sup (\Hull{\lr}{n} (\alpha_k \cup \{ \theta \} \cup  \R^+ \cup \{ \R^+\}) \cap \OR),
    \]
    where $\vec{p} = \langle \alpha_k \rangle_{k<\omega}$ is as in Definition \ref{dfn of branch}. Let $S_\infty := \vec{p}^\frown \vec{\gamma}$.
\end{dfn}

Note that by $\Sigma_n$-Collection $\gamma_k < \alpha$ for all $k < \omega$. Moreover, $\sup(\vec{\gamma}) = \alpha$, as otherwise $\tho_{r\Sigma_n}^{\lr} (\theta +1 \cup  \R^+ \cup \{ \R^+\}) \in \lr$.
We also have $\gamma_k = \sup \{ \lv_{\varphi}^{\lr} (\vec{x}) : \varphi(\vec{x}) \in \tho_{r\Sigma_n}^{\lr} (\alpha_k \cup \{ \theta \} \cup \R^+ \cup \{ \R^+ \}) \}$ for $k <\omega$. In particular, $\gamma_k \in S^{\lr}_{n-1}$.

\begin{lem} \label{definability of LR}
    Let $N \in \F$. Let $k < \omega$ and $s = \vec{p} \restriction k$. Then $\{ s \}$ is $r\Sigma_n \land r\Pi_n$-definable from the parameter $\theta^N$ over $N$, uniformly in $N$.
\end{lem}

\prf
    Note that $\mathbb{B}^N \times \col^N(\omega, {< \kappa})$ is equivalent to $\mathbb{P} := \col^N (\omega, {< \kappa})$ and therefore there exists $G^\prime$ which is $(N,\col^N (\omega, {< \kappa}))$-generic and equivalent to $(x,G)$, i.e. $N[G^\prime] = N[x,G] = \lxg$.

    Let
    \[
        \dot{\R} = \{ (\dot{z},p) \mid p \Vdash_N^\mathbb{P} ``\dot{z} \text{ is a real}" \} \cap N \vert \kappa
    \]
    be the canonical name of $\R^+$. Since $\col(\omega,{<\kappa})$ is homogeneous, $\dot{\R}$ is homogeneous, and $\lr$ is a $\Sigma_1$-definable class of $\lxg$ from the parameter $\dot{\R}$, there is an $r\Sigma_n$ formula $\varphi_{\lr}$ such that for all $r\Sigma_n$ formulas $\varphi$ and all $\eta < \alpha$,
    \[
        \emptyset \Vdash_N^\mathbb{P} \varphi_{\lr}(\check{\eta},\check{\varphi},\dot{\R}) \iff \lr \models \varphi(\eta).
    \]
    By Theorem \ref{thm: branch theorem}, $\{ s \}$ is $r\Sigma_n \land r\Pi_n$-definable from the parameter $\theta^N = \kappa^{+\lx}$ over $\lr$. It then follows by Lemma \ref{cor: forcing theorem for conjunctions} that $\{ s \}$ is $r\Sigma_n \land r\Pi_n$-definable from the parameter $\theta^N$ over $N$. 
\eprf

\begin{lem} \label{fix points}
    Let $N,P \in \F$ be such that $N \dashrightarrow P$. Then for all $k < \omega$, $i_{N,P}^\F (\alpha_k) = \alpha_k$ and for all $k < \omega$, $i_{N,P}^\F (\gamma_k) = \gamma_k$.
\end{lem}

\prf
    For $\alpha_k$, $k < \omega$, the claim follows immediately from Lemma \ref{definability of LR}.
    Let $k < \omega$ and let $H^N := \Hull{N}{n} (\alpha_k \cup \{ \theta \})$, $H^P := \Hull{P}{n} (\alpha_k \cup \{ \theta \})$ and $H := \Hull{\lr}{n}(\alpha_k \cup \{\theta\} \cup \R^+ \cup \{\R^+\})$.
    
    We claim $H \cap \OR = H^N \cap \OR = H^P \cap \OR$. Note that since $N^- \in \lx \vert \kappa$, $N^-$ is coded by a real in $\R^+$. Then it is easy to see that $H^N \subseteq H$. On the other hand, letting $\mathbb{P} := \col^N (\omega, {< \kappa})$ and $G^\prime$ be $(N,\col^N(\omega,{< \kappa}))$-generic such that $N[G^\prime] = L_\alpha[x,G]$ there are $\mathbb{P}$-names for every real $x \in \R^+$ in $N \vert \kappa$ so that, similarly to the proof of Lemma \ref{definability of LR}, for every ordinal in $H$ there is a condition forcing its definition over $N$, so that $H \cap \OR \subseteq H^N \cap \OR$.

    In particular, $\gamma_k = \sup(H^N \cap \alpha) = \sup(H^P \cap \alpha)$. In order to finish the proof, it suffices to see that $i_{NP}(\sup(H^N \cap \alpha)) = \sup(H^P \cap \alpha)$. Note that
    \[
        N \models \forall \eta < \gamma_k \exists \varphi \in \fml \exists \vec{x} \in \fin{\alpha_k} (\lv_\varphi (\vec{x},\theta) > \eta),
    \]
    where $\fml$ is the set of $r\Sigma_n$ formulas. But then by $r\Sigma_{n+1}$-elementarity, this is preserved by $i_{N,P}^{\F}$, so 
    \[
        P \models \forall \eta < i^\F_{N,P}(\gamma_k) \exists \varphi \in \fml \exists \vec{x} \in \fin{\alpha_k} (\lv_\varphi (\vec{x},\theta) > \eta).
    \]
    However, since $\sup(H^P \cap \alpha) = \gamma_k$, this means that $i_{N,P}^\F(\gamma_k) = \gamma_k$.
\eprf

\begin{lem} \label{S is cofinal}
    Let $N$ be a good pre-$\mad$-like $x$-weasel such that $N^- \in \lx \vert \kappa$ and let $S \subseteq \alpha$ be such that $S$ is cofinal in $\theta$ and $\alpha$. Then $\sup(\delta^N \cap \Hull{N}{n}(S)) = \delta^N$.
\end{lem}

\begin{proof}
    Let $X:= \Hull{N}{n}(S)$ and suppose for the sake of contradiction that $\sup(X\cap\delta^N) = \gamma < \delta^N$. Let $Y := \Hull{N}{n}(\gamma \cup S)$. Since $S$ is cofinal in $\alpha$ and $\delta^N$ is regular in $N$, it follows that $\gamma = \sup(Y \cap \delta^N)$. Let $\pi \colon \bar{N} \to N$ be the inverse of the transitive collapse map of $Y$. Let $(\bar{\delta},\bar{\kappa}, \bar{\theta}) \in \bar{N}$ be such that $\pi ((\bar{\delta},\bar{\kappa}, \bar{\theta})) = (\delta^N,\kappa,\theta)$. Note that $\bar{\delta}$ is a Woodin cardinal in $\bar{N}$, $\bar{\kappa}$ is inaccessible in $\bar{N}$, and $\bar{\theta}$ is the cardinal successor of $\bar{\kappa}$ in $\bar{N}$. 
    As $\gamma = \bar{\delta}$ is the critical point of $\pi$, $\bar{N} \vert \vert \gamma = N \vert \vert \gamma$. Moreover, if $\bar{\alpha} := \OR^{\bar{N}}$, then $\bar{N} = L_{\bar{\alpha}} (\bar{N} \vert \gamma)$. Since $N$ does not have a proper initial segment that models $\tho$, there are no extenders of $\E^N$ indexed in the interval $(\gamma,\bar{\alpha}]$. Thus, $\bar{N} \lhd N$.
    
    We now aim to derive a contradiction by showing that $\bar{N}$ is $\Sigma_n$-admissible. Suppose for the sake of contradiction that $\bar{N}$ is not $\Sigma_n$-admissible, i.e.~$\Sigma_n$-Collection fails. Let $\lambda$ be the least failure of $\Sigma_n$-Collection, i.e.~there is a function $f \in \Sigma_n^{\bar{N}} (\bar{N})$ such that $\dom(f) = \lambda$ where $\lambda \leq \bar{\theta}$ but $f \notin \bar{N}$, and $\lambda$ is the least such. Let $\varphi_f$ be a $\Sigma_n$ formula and $p \in \bar{N}$ which define $f$.

    Let us first consider the case that $\lambda < \bar{\theta}$. Let $\tilde{f}$ be the partial function with $\dom(\tilde{f}) = \sup(\pi [\lambda])$ and for $x \in \dom(\tilde{f})$, $\tilde{f}(x)$ is the unique $y \in N$ such that $N \models \varphi_f (x, y, \pi (p))$ if it exists and otherwise $\tilde{f}(x)$ is undefined. Note that $\sup(\dom(\tilde{f})) \leq \pi(\lambda) < \theta$. Moreover, since $S$ is cofinal in $\alpha$ and $f$ is cofinal in $\OR^{\bar{N}}$ it follows that $\tilde{f}$ is cofinal in $\alpha$. Let $B \subseteq \sup(\pi(\lambda)) < \theta$ be such that $B = \{ \xi < \pi(\lambda) \colon N \models \exists y \varphi_f (\xi, y, \pi(p))\}$. Note that $B$ is $\Sigma_n$-definable. In the case $B \in N$ it follows from $\Sigma_n$-admissibility and the $\Sigma_n$-elementarity of $\pi$ that $f \in \bar{N}$, a contradiction! In the case $B \notin N$, we have $\rho^N_n \leq \sup(\pi(\lambda)) < \theta$, a contradiction!

    Now, let us consider the case that $\dom(f) = \bar{\theta}$. Note that $\pi$ is continuous at $\bar{\theta}$, since $S$ is cofinal in $\theta$. Let $\tilde{f}$ be defined as before. Since $\lambda$ is the least failure of $\Sigma_n$-Collection in $\bar{N}$, it follows that $\tilde{f}(\gamma)$ is defined for all $\gamma \in \theta$. But then we have by $\Sigma_n$-admissibility of $N$, that $\tilde{f} \in N$, and it follows that $f \in \bar{N}$. Contradiction!
\end{proof}

\begin{lem} \label{lem: S infty is generating}
    Let $N \in \F$. Then
    \[
        N = \Hull{N}{n}(\delta^N \cup S_\infty).
    \] 
\end{lem}

\prf
    Note that by Lemma \ref{lem: standard parameter of admissible mice}, $\mad = \Hull{\mad}{n+1} (\omega)$. Let $N \in \F$ and $i \colon \mad \to N$ be the iteration map according to $\Sigma$. Since $\rho_n^{\mad} = \theta^{\mad}$ and the iteration tree $\T$ is $n$-maximal, $i$ is an $n$-embedding. Moreover, $N$ is $n$-sound and $N = \Hull{N}{n+1} (\delta^N)$.

    Let $k < \omega$ and let $T_k := \tho^N_n (\alpha_k \cup \{ \theta \}) \in N$. Note that the function $f$ that sends $\varphi(\vec{x}) \in T_k$ to $\lv_\varphi^N(\vec{x})$ is $r\Sigma_n$ over $N$. Thus, by $\Sigma_n$-Collection there is some $\eta < \alpha$ such that for every $\varphi(\vec{x}) \in T_k$ there is a subtheory of $\tho^N_{n-1}(\eta \cup \{ \theta \})$, which witnesses $\varphi(\vec{x}) \in T_k$. However, since $S_\infty$ is cofinal in $\alpha$ we may assume without loss of generality that $\eta \in S_\infty$. But then, as $\tho^N_{n-1}(\eta) \in \Hull{N}{n}(S_\infty)$, $T_k \in \Hull{N}{n}(S_\infty)$.

    Now note that since for every $r\Sigma_{n+1}$ formula $\varphi$ and $\vec{x} \in \fin{\delta^N}$, the $r\Sigma_{n+1}$ formula $\exists z (m\tau_\varphi (\vec{x}) = z)$ has a witness which is coded by a subtheory of $\tho_{n}^N (\gamma)$ for $\gamma < \alpha$ large enough. Thus, since $\langle \alpha_k \rangle_{k<\omega}$ is cofinal in $\theta$, we may assume that $\gamma = \alpha_k$ for some $k < \omega$. But $T_k \in \Hull{N}{n}(\delta^N \cup \{S_\infty\})$, so that $m\tau^N_\varphi (\vec{x}) \in \Hull{N}{n}(S_\infty \cup \delta^N)$.
\eprf

Note that since $\vec{\gamma} \subseteq  S^{L(\R^+)}_{n-1}$, we have the following.

\begin{cor} \label{cor: S is generating}
    Let $N \in \F$. Then
    \[
        N = \bigcup_{s \in \fin{S_\infty}} H^N_s.
    \]
\end{cor}

\begin{dfn}
    Let $N \in \F$ and set $S_\infty^* = i^{\F}_{N\infty} [S_\infty]$.
\end{dfn}

Note that for any $N,P \in \F$ we have $i^{\F}_{N\infty} \restriction S_\infty = i^{\F}_{P\infty} \restriction S_\infty$. Thus, $S_\infty^*$ is independent of $N$.

The next corollary is an immediate consequence of Lemma \ref{lem: S infty is generating}, since the iteration maps are $n$-embeddings.

\begin{cor} \label{lem: generators of the direct limit}
    $M_\infty^{\F} = \Hull{M_\infty^{\F}}{n} (\delta_\infty \cup S_\infty^*)$.
\end{cor}


\subsection{Properties of the direct limit} \label{Properties of the direct limit}

We are now going to establish that $M_\infty^\F$ and $M_\infty^{\tilde{\D}}$ are equal.

\begin{dfn} \label{dfn: sigma}
    Let $\sigma \colon M_\infty^{\tilde{\D}} \to M_\infty^\F$ be defined as follows: Given $(N,s) \in \tilde{\D}$ and $x \in H^{N}_s$ let $P \in \F$ be such that $N \dashrightarrow P$ and $P$ is $s$-stable and set \[
        \sigma (i^{\tilde{\D}}_{(N,s)\infty} (x)) = i^{\F}_{P\infty} (i^{\tilde{\D}}_{(N,s),(P,s)} (x)).
    \]
\end{dfn}

The proof of the following Lemma is a variant of the proof of Claim 2 in \cite{SchlutzenbergTheta2016}. 

\begin{lem} \label{lem: direct limits are the same}
    $M_\infty^{\tilde{\D}} = M_\infty^\F$ and $\sigma = \id$.
    \end{lem}
    
    \prf
        It suffices to see that $\sigma$ is surjective. Let $y \in M_\infty^\F$. Let $P \in \F$ and $\bar{y} \in P$ be such that $i^\F_{P \infty} (\bar{y}) = y$. By Corollary \ref{cor: S is generating}, there is $s \in \fin{S_\infty}$ such that $\bar{y} \in H^P_s$. Note that by Lemma \ref{fix points}, $P$ is $s$-stable and hence $s$-iterable. Thus,
        \[
            \sigma(i^{\tilde{\D}}_{(P,s),\infty}(\bar{y})) = i^\F_{P\infty} (\bar{y}) = y
        \] 
        and $y \in \ran(\sigma)$.
    \eprf

\begin{dfn}
    Let $M_\infty = M_\infty^{\tilde{\D}} = M_\infty^{\F}$ and let $\delta_\infty$ be the unique Woodin cardinal of $M_\infty$, $\kappa_\infty$ be the unique inaccessible cardinal of $M_\infty$ greater than $\delta_\infty$, and $\theta_\infty = (\kappa_\infty)^{+M_\infty}$.
\end{dfn}

\begin{lem}
Let $\eta_\infty$ be the least measurable cardinal of $M_\infty$. Then the following hold:
\begin{enumerate}
    \item \label{least measurable}$\eta_\infty = \kappa$, and
    \item \label{least Woodin}$\delta_\infty = \theta = \kappa^{+\lxg}$.
\end{enumerate}

\end{lem}

\begin{proof}


    Showing clause \ref{least measurable} is a standard argument, so we omit the proof. That $\delta_\infty \geq \theta$ follows as in the $L[x,G]$-case, so we omit the argument.
    Let us show that $\delta_\infty \leq \theta$. 
    Let $\gamma < \delta_\infty$ and $(N,s) \in \tilde{\D}$ be such that there is $\bar{\gamma} \in H^N_s$ such that $i^{\tilde{\D}}_{(N,s)\infty} (\bar{\gamma}) = \gamma$. Note that $\bar{\gamma} < \gamma^N_s$.
    Let 
    \[
        A:= \{ (Q,\beta) \colon (N,s) \leq (Q,s) \text{ and } \beta < i^{\tilde{\D}}_{(N,s),(Q,s)} (\bar{\gamma})   \}
    \]
    Let $f \colon A \to \OR$ be given by 
    \[
        f((Q,\beta)) = i^{\tilde{\D}}_{(Q,s)\infty} (\beta).
    \]
    We have that $\gamma \subseteq \ran (f)$. Note that the map which sends pairs $(P,\beta),(Q,\beta) \in A$ such that $(P,s) \leq (Q,s)$ to $i^{\tilde{\D}}_{(P,s),(Q,s)}$ is definable over $\lx \vert \max(s)$. Thus, over $\lx \vert \max(s)$, we may define the direct limit of these maps and then take its transitive collapse in $\lx$, so that $A,f \in \lx$. Since $A$ may be coded by a subset of $\kappa$, we have that $\gamma < \theta$. We have shown that $\delta_\infty \leq \theta$.
\end{proof}

\begin{dfn}
    Let
    \[
        \tilde{\D} \restriction S_\infty = \{ (N,s) \in \tilde{\D} \mid s \in \fin{S_\infty} \}.
    \] 
\end{dfn}

\begin{lem}
    $\tilde{\D} \restriction S_\infty$ covers $\tilde{\D}$ in that for all $z \in M_\infty$ there is $(N,s) \in \tilde{\D} \restriction S_\infty$ such that $z \in \ran(i^{\tilde{\D}}_{(N,s)\infty})$.
\end{lem}

\prf
    Let $z \in M_\infty$. Let $(N,s) \in \tilde{\D}$ be such that $z \in \ran(i^{\tilde{\D}}_{(N,s)\infty})$ and let $\bar{z} \in H^N_s$ be such that $i^{\tilde{\D}}_{(N,s)\infty} (\bar{z})=z$. We may assume without loss of generality that $N \in \F$ is a $s$-stable. By Corollary \ref{cor: S is generating}, there is $t \in \fin{S_\infty}$ such that $\bar{z} \in H^N_t$. But then, since $N$ is $t$-stable, $(N,s\cup t) \in \tilde{\D}$. This means that $i^{\tilde{\D}}_{(N,t)\infty} (\bar{z}) = z = i^{\tilde{\D}}_{(N,s)\infty} (\bar{z})$.
\eprf

\subsection{The definability of the direct limit} \label{definability of direct limit}

So far we have established that $\tilde{\D}$ is a $\Sigma_1$-definable class of $\lx$ in the parameters $\R^{\lx}$ and $\kappa$. We now show that $M_\infty$ is a $\Sigma_1$-definable class of $\lx$ in the parameters $\R^{\lx}$ and $\kappa$.

\begin{dfn}
    Let
    \[
        \tilde{\D} \restriction \theta = \{ (N,s) \in \tilde{\D} :  s \in \finite(\theta) \}.
    \] 
    Let $\bar{M}_\infty$ be the direct limit of
    \[
        \langle H^{N}_s ; i^{\tilde{\D}}_{(N,s),(P,t)}: (N ,s) \leq (P,t) \in \tilde{\D} \restriction \theta \rangle,
    \]
    and for $(N,s) \in \tilde{\D} \restriction \theta$, let $i^{\tilde{\D} \restriction \theta}_{(N,s)\infty}$ be the direct limit map. 
\end{dfn}

\begin{lem}
    $(\tilde{\D} \restriction \theta, \leq \restriction \theta)$ is definable over $\lx \vert \theta$
\end{lem}

\prf
    It suffices to see that being $s$-iterable is definable over $\lx \vert \theta$ for $s \in \finite(\theta)$. However, this is true by Lemma \ref{lem: definability of s-iterability}.
\eprf

\begin{dfn}
    Let $\bar{\sigma} \colon \bar{M}_\infty \to M_\infty \vert \theta$ be such that for $y \in \bar{M}_\infty$, if $(N,s) \in \tilde{\D} \restriction \theta$ and $\bar{y} \in H^{N}_s$ are such that $y = i^{\tilde{\D} \restriction \theta}_{(N,s)\infty} (\bar{y})$, then $\bar{\sigma}(y) = i^{\tilde{\D}}_{(N,s)\infty} (\bar{y})$.
\end{dfn}

The following lemma and its proof are a variant of Lemma 4.41 (b) in \cite{sargsyan2021varsovianmodelsii}.

\begin{lem} \label{definability up to kappa+}
    $\bar{\sigma} = \id$.
\end{lem}

\prf
    For $s \in \fin{S_\infty} \setminus \{ \emptyset \}$ and $(N,s) \in \tilde{\D} \restriction S_\infty$ such that $N \in \F$, let $K^N$ be the transitive collapse of 
    \[
        H^N := \Hull{N \vert \max(s)}{\omega} (\kappa \cup s^-),
    \]
    and let $\pi^N \colon K^N \to H^N$ be the inverse of the transitive collapse map. Set $\bar{s}^N = t \cup \{ \OR^{K^N} \}$, where $\pi^N (t) = s^-$.

    We aim to see that $(N , \bar{s}^N) \in \tilde{\D} \restriction \theta$. To this end note that it suffices to see that $N$ is $\bar{s}^N$-stable. Let $P \in \F$ be such that $N \dashrightarrow P$ and let $i := i^{\F}_{N,P} \colon N \to P$. Since $N$ is $s$-stable, $i(\kappa) = \kappa$ and $i$ is a $n$-embedding, we have $i(\bar{s}^N) = \bar{s}^P$. We aim to see that $\OR \cap H^N = \OR \cap H^P$, since then it follows that $\bar{s}^N = \bar{s}^P$ and so $N$ is $\bar{s}^N$-stable.

    Let
    \[
        H := \Hull{\lx \vert \max(s)}{\omega} ( \kappa \cup s^- \cup \{ x \}),
    \]
    where we consider the hull in the language $\mathcal{L}_{\dot{\in}}$.
    We claim that for all $Q \in \F$, $H \cap \OR = H^Q \cap \OR$.
    Let us first show that $H^Q \cap \OR \subseteq H \cap \OR$. Note that $Q^- \in \lx \vert \max(s)$ and $Q \vert \max(s) = L_{\max(s)} (Q^-)$. Thus, $Q \vert \max(s)$ is a definable class of $\lx \vert \max(s)$. But then it is easy to see that $H^Q \cap \OR \subseteq H \cap \OR$.

    In order to show the converse inclusion, let $\xi \in H \cap \OR$. Then by the argument from the proof of Lemma \ref{definability of LR} and Lemma \ref{lem: level by level forcing}, there are a $\Sigma_\omega$ formula $\varphi$, $p \in \mathbb{B}^Q$, $\vec{z} \in \fin{\kappa}$, and a $\mathbb{B}^Q$-name $\dot{x}$ for $x$ such that 
    \[
        Q \models p \Vdash_{\mathbb{B}} \check{\xi} = \tau_\varphi^{\lx \vert \max(s)} (\vec{z},s^-,\dot{x}).
    \]
    However, since we may consider $\mathbb{B}^Q$ as a subset of $\kappa$, this defines $\xi$ over $Q \vert \max(s)$. Thus, $\xi \in H^Q$.

    Let $y \in M_\infty \vert \theta_\infty$. Let $(N,s) \in \tilde{\D}$ be such that $N \in \F$, $N$ is $s$-stable and there is $\bar{y} \in H^{N^+}_s$ such that $i^{\tilde{\D}}_{(N,s) \infty} (\bar{y}) = y$. Let $y^\prime = i^{\tilde{\D} \restriction \theta}_{(N,\bar{s})\infty} (\bar{y})$. We claim $\bar{\sigma}(y^\prime) = y$. By definition, $\bar{\sigma}(y^\prime) = i^{\tilde{\D}}_{(N,\bar{s})\infty} (\bar{y})$. Note that $N$ is $s \cup \bar{s}$-stable and, therefore, $(N,s\cup\bar{s}) \in \tilde{\D}$. However, $i^{\tilde{\D}}_{(N,s),(N,s\cup\bar{s})} (\bar{y}) = i^{\tilde{\D}}_{(N,\bar{s}),(N,s\cup\bar{s})} (\bar{y})$, and the claim follows. We have shown that $y \in \ran(\bar{\sigma})$. Since $\bar{\sigma}$ is an embedding, it follows that $\bar{\sigma} = \id$.
\eprf

\begin{lem} \label{lem: ordinal height of the direct limit}
    $\OR^{M_\infty} = \alpha$.
\end{lem}

\prf
    Note that by definition, $\OR^{M_\infty} \geq \alpha$. Suppose for the sake of contradiction that $\OR^{M_\infty} > \alpha$. By Lemma \ref{definability up to kappa+}, the uncollapsed version of the direct limit up to its largest cardinal is definable over $\lx \vert \theta$. Since $\lx$ is $\Sigma_1$-admissible and the full scheme of foundation holds in $\lx$, it follows that $\theta_\infty \in \lx$. Thus, $\alpha > \theta_\infty$. However, this means that $M_\infty \vert \alpha \models \Sigma_n\text{-}\kp$, since $M_\infty = L_{\OR^{M_\infty}} (M_\infty \vert \delta_\infty )= L_{\OR^{M_\infty}} (M_\infty \vert \theta )$ and $M_\infty \vert \delta_\infty \in \lx$ which holds again by Lemma \ref{definability up to kappa+}. Thus, $M_\infty \vert \alpha \models \tho^\prime$, a contradiction, since $M_\infty \models \tho$!
\eprf

\begin{lem} \label{lem: definability of the direct limit}
    $M_\infty$ is a $\Sigma_1$-definable class of  $\lx$ from the parameters $\R^{\lx}$ and $\{ \theta \}$.
    In addition, $\tilde{\D}$ and $M_\infty$ are $\Sigma_2$-definable classes of $\lx$ without parameters.
\end{lem}

\prf
    The first part follows from Lemma \ref{definability up to kappa+} and Lemma \ref{lem: ordinal height of the direct limit}.
    The second part then follows, since $\R^{\lx}$ is $\Sigma_2$-definable over $\lx$ without parameters.
\eprf

\section{$M_\infty$'s version of the direct limit} \label{M infinitys version}

\begin{dfn}
    Let $* \colon \alpha \to \alpha$ be defined as follows: For $\beta < \alpha$, let $(N,s) \in \tilde{\D}$ be such that $\beta \in s^-$ and set
    \[
        \beta^* = i^{\tilde{\D}}_{(N,s)\infty} (\beta).
    \]
\end{dfn}

    Note that the definition of $\beta^*$ does not depend on $(N,s)$.

\setcounter{claimcounter}{0}

By Lemma \ref{definability of the covering system}, $\tilde{\D}$ is a $\Sigma_1$-definable class of $\lx$ from the parameters $\theta$ and $\R^{\lx}$. This allows us to define a version of the internal covering system $\tilde{\D}$ in $M_\infty$. 

\begin{dfn}
    Let $\varphi$ be the $\Sigma_1$-formula given by Lemma \ref{definability of the covering system} and $N \in \F$ and let $\mathbb{P} = i^{\F}_{N\infty} (\mathbb{B}^N)$. Let $h$ be $(M_\infty,\mathbb{P})$-generic. Let $\tilde{\D}^\infty$ be the class of all $a \in M_\infty[h]$ such that $M_\infty [h] \models \varphi(a,\theta^\infty,\R^{M_\infty[h]})$.

    Let $M_\infty^{\tilde{\D}^\infty}$ be defined analogously via $\psi$, where $\psi$ is as in the first part of Lemma \ref{lem: definability of the direct limit}, and let $i^{\tilde{\D}^\infty}_{(N,s),(P,t)}$ denote the corresponding maps and $i_{(N,s)\infty}^{\tilde{\D}^\infty}$ denote the direct limit maps.
\end{dfn}

\begin{rem}
    Note that notions of pre-$\mad$-like $x$-weasel, $\mad$-like, and $s$-iterability do not apply to the structures in $\tilde{\D}^\infty$, simply because for $N \in \tilde{\D}^\infty$, $\theta^N > \theta$. However, these notions may be straightforwardly adapted for the structures in $\tilde{\D}^\infty$, so that we will also talk about these for the elements in $\tilde{\D}^\infty$. We leave the details of this adaption to the reader.
\end{rem}

\begin{lem} \label{lem: s-iterability in the direct limit}
    For $s \in \finite(\alpha)$, $M_\infty [h] \models ``M_\infty \text{ is }s^*\text{-iterable}"$.
\end{lem}

\prf
    Let $N \in \F$ be $s$-stable. Note that
    \[
        N \models \exists p \in \mathbb{B}^N (p \Vdash_{\mathbb{B}^N} ``N \text{ is } s\text{-iterable})".
    \]
    But this means that over $M_\infty$ it is forced that $M_\infty$ is $s^*$-iterable.
\eprf

In particular, $(M_\infty,s^*) \in \tilde{\D}^\infty$ for all $s \in \finite(\alpha)$. Thus, the following definition makes sense.

\begin{dfn}
    $i^{\tilde{\D}^\infty}_{M_\infty\infty} = \bigcup \{ i^{\tilde{\D}^\infty}_{(M_\infty,s^*)\infty} : s \in \finite(\alpha) \}$.
\end{dfn}

From Lemma \ref{lem: s-iterability in the direct limit} it follows that the following lemma is well-stated.

\begin{lem}
    $M_\infty = \bigcup \{ H^{M_\infty}_{s^*} : s \in \fin{S_\infty} \setminus \{\emptyset\} \} = \bigcup \{ H^{M_\infty}_{s} : s \in \fin{S_\infty^*} \setminus \{\emptyset\}  \}$ and $i^{\tilde{\D}^\infty}_{M_\infty\infty} \colon M_\infty \to M_\infty^{\tilde{\D}^\infty}$.
\end{lem}

\prf
    This follows from Corollary \ref{lem: generators of the direct limit} and the fact that $\delta_\infty = \sup(\delta_\infty \cap \Hull{M_\infty}{n} (S_\infty^*))$.
\eprf

\begin{dfn}
    Let $\F^*$ be the set of all non-dropping iterates $N$ of $M_\infty$ via $n$-maximal trees in $M_\infty \vert \kappa_\infty$ such that $N^- \in M_\infty \vert \kappa_\infty$ and let $M^{\F^*}_\infty$ be the direct limit of $\F^*$.
\end{dfn}

\begin{lem} \label{direct limits agree 2}
    $M^{\F^*}_\infty= M_\infty^{\tilde{\D}^\infty}$. In particular, $M_\infty^{\tilde{\D}^\infty}$ is a normal iterate of $M_\infty$.
\end{lem}

\prf
    Let $N,P \in \F^*$ be such that $N \dashrightarrow P$ and let $i \colon N \to P$ be the iteration map. We claim $i \restriction S_\infty^* = \id$. Let $Q \in \F$. Note that since $Q$ embeds into $M_\infty$ via $i^{\F}_{Q\infty}$, $Q$ embeds into $N$ and $P$ via $n$-embeddings. However, $S_\infty$ is $r\Sigma_{n+1}$-definable over $Q$. Note that then similar arguments as in the proof of Lemma \ref{thm: branch theorem} and Lemma \ref{fix points} show that $i \restriction S_\infty^* = \id$.
    By Corollary \ref{lem: generators of the direct limit}, $M_\infty = \Hull{M_\infty}{n} (\delta_\infty \cup S_\infty^*)$. We thus have that for any $N \in \F^*$, $N = \Hull{N}{n} (\delta^N \cup S_\infty^*)$, where $\delta^N$ is the Woodin cardinal of $N$. Moreover, $\delta^N = \sup(\delta^N \cap \Hull{N}{n}(S_\infty^*))$ as the proof of Lemma \ref{S is cofinal} shows. The claim now follows as in the proof of Lemma \ref{lem: direct limits are the same}.
\eprf

\begin{dfn}
    Let $M_\infty^\infty = M^{\F^*}_\infty= M_\infty^{\tilde{\D}^\infty}$ and let $k \colon M_\infty \to M_\infty^\infty$ be the iteration map given by $\Sigma^{\mad}$.
\end{dfn}

\begin{lem}
    $k = i^{\tilde{\D}^\infty}_{M_\infty\infty}$.
\end{lem}

\prf
    By the proof of Lemma \ref{direct limits agree 2}, for all $s \in \fin{S_\infty^*}$, $M_\infty$ is $s$-stable. However, then similar to Lemma \ref{lem: agreement of maps}, for all $N \in \F^*$, $i^{\F^*}_{M_\infty N} \restriction H^{M_\infty}_s = i^{\tilde{\D}^\infty}_{(M_\infty,s)(N,s)}$ for all $s \in \fin{S_\infty^*}$. The claim then follows from the fact that $\{ (N,s) : N \in \F^* \land s \in \fin{S_\infty^*} \}$ is dense in $\tilde{\D}^\infty$.
\eprf

For the proof of Lemma \ref{lem: iteration maps and star agree} we will need to consider the internal direct limit system as computed in $\lxg$, where $G$ is $(\lx,\col(\omega, {<\kappa}))$-generic. So let us fix such $G$ and let $\tilde{\D}^{\lxg}$ be $\lxg$'s version of $\tilde{\D}$ and let $M_\infty^{\lxg}$ the direct limit of this system. We claim that $M_\infty = M_\infty^{\lxg}$. This is follows from the fact that $\tilde{\D}$ is dense in $\tilde{\D}^{\lxg}$. This is in turn shown by a ``Boolean-valued comparison". More precisely, for $(N,s) \in \tilde{\D}^{\lxg}$, we compare $(P,s) \in \tilde{\D}$ with $(N,s)$ via the process described in the proof of Lemma 3.47 of \cite{Steel_Woodin_2016}. Together with the argument from the proof of Lemma \ref{lem: comparison lemma} it then follows that $\tilde{\D}$ is dense $\tilde{\D}^{\lxg}$. By a straightforward adoption of the arguments so far, we also have that $\tilde{\D}^{\lxg}$ is $\Sigma_1$-definable over $\lxg$ from the parameters $\theta$ and $\R^{\lxg}$, and $M_\infty = M_\infty^{\lxg}$ is $\Sigma_1$-definable from the same parameters.

\begin{lem} \label{lem: iteration maps and star agree}
    $k \restriction \alpha = *$.
\end{lem}

\prf
    Let $\beta < \alpha$. By Corollary \ref{lem: generators of the direct limit}, there is $s \in \fin{S_\infty}$ such that $\beta \in H^{M_\infty}_{s^*}$. Let $N \in \F$ be such that $\beta \in \ran(i^{\tilde{\D}}_{(N,s)\infty})$ and $N$ is $\{\beta\}$-stable and hence $s \cup \{\beta\}$-stable. Let $\bar{\beta} \in N$ be such that $i^{\tilde{\D}}_{(N,s)\infty} (\bar{\beta}) = \beta$. Note that $N$ models
    \[
        \col(\omega, {<\kappa}) \Vdash i^{\tilde{\D}}_{(V[\dot{g}],s)\infty} (\bar{\beta}) = \beta
    \]
    and this statement is $\Sigma_2$ over $N$ by the remark before the statement of the lemma. Since $i^{\F}_{N\infty}$ is $\Sigma_1$-elementary, it follows that 
    \[
        M_\infty \models \col(\omega, {<\kappa_\infty}) \Vdash  i^{\tilde{\D}^\infty}_{(V[\dot{g}],s^*)\infty} (\bar{\beta}) = \beta^*.
    \]
    Since $i^{\tilde{\D}^\infty}_{(M[h],s^*) \infty} (\beta) = k (\beta)$, this finishes the proof.
\eprf

\begin{lem} \label{lem: definability of the star map}
    $*$ is $\Sigma_1$-definable from $* \restriction \theta$ over $M_\infty[*\restriction \theta] := L_\alpha[\E^{M_\infty},*\restriction \theta]$.
\end{lem}

\prf
    Let $E$ be the $(\theta,\delta^\infty_\infty)$-extender derived from $k \colon M_\infty \to M_\infty^\infty$, where $\delta_\infty^\infty$ denotes the unique Woodin cardinal of $M_\infty^\infty$.
    Note that the iteration from $M_\infty$ to $M_\infty^\infty$ is based on $M_\infty \vert \delta_\infty$, so that $E$ is $\Sigma_1$-definable from $k \restriction \delta_\infty = * \restriction \delta_\infty$ and $M_\infty \vert \delta^\infty_\infty$ and thus from $* \restriction \theta$ over $M_\infty[* \restriction \theta]$. Moreover, $k$ is the same as the ultrapower embedding $\sigma \colon M_\infty \to \fu {M_\infty} E 0 = \fu {M_\infty} E n =  M_\infty^\infty$. In order to prove the claim, it suffices to see that $\sigma$ is $\Sigma_1$ definable over $M_\infty[* \restriction \theta]$ from the parameter $E$. Let $\sigma^\prime \colon M_\infty \vert \theta_\infty \to \Ult{M_\infty \vert \theta_\infty}{E}$ be the ultrapower map of the ultrapower of $M_\infty \vert \theta_\infty$ via $E$. Note that $\sigma^\prime$ is $\Sigma_1$ definable over $M_\infty[* \restriction \theta]$ from the parameter $E$. 
    Let 
    \[
        A = \{ \gamma < \alpha : \rho_\omega^{M_\infty \vert \gamma} = \theta_\infty \}.
    \]
    Since $\theta_\infty$ is the largest cardinal of $M_\infty$, $\sup(A) = \alpha$. For $\gamma \in A$, let $T_\gamma = \tho_\omega^{M_\infty \vert \gamma}(\theta_\infty)$. Note that $\sigma$ is continuous at $\theta_\infty$. Thus, $\sigma(T_\gamma) = \bigcup_{\xi < \theta_\infty}\sigma^\prime(T_\gamma \cap M_\infty \vert \xi)$. However, $M_\infty[* \restriction \theta]$ can via $\sigma(T_\gamma)$ compute  $\sigma \restriction (M_\infty \vert \gamma)$. Thus, $\sigma$ is $\Sigma_1$-definable over $M_\infty [* \restriction \theta]$ from $* \restriction \theta$.
\eprf

\begin{dfn}
    Let $M_\infty[*] = L_\alpha [\E^{M_\infty}, *]$.
\end{dfn}

\begin{dfn}
    Let $\Sigma_0$ be the restriction of $\Sigma^{\mad}$ to non-dropping stacks on $M_\infty$ such that the last model is pre-$\mad$-like $x$-weasel and its suitable part is in $M_\infty \vert \kappa_\infty$. 
\end{dfn}
\setcounter{claimcounter}{0}
\begin{lem} \label{strategy mouse}
    $M_\infty [*] = M_\infty [\Sigma_0]$
\end{lem}

\prf
    Let us first prove that $M_\infty [*] \subseteq M_\infty [\Sigma_0]$. Note that $\F^* \in M_\infty [\Sigma_0]$. Thus, $i_{M_\infty \infty}^{\F^*} \restriction \theta \in M_\infty [\Sigma_0]$, where $i_{M_\infty \infty}^{\F^*}$ is the direct limit map. So $M_\infty [*] \subseteq M_\infty [\Sigma_0]$.
    
    In order to show that $M_\infty [\Sigma_0] \subseteq M_\infty [*]$. Let $N$ be a non-dropping iterate of $M_\infty$ via an $n$-maximal tree $\T^\prime$ which is according to $\Sigma^{\mad}$ such that $N$ is a pre-$\mad$-like $x$-weasel and $N^- \in M_\infty \vert \kappa_\infty$. Let $b$ and $\T$ be such that $\T^\prime = \T^{\frown}b$. It suffices to see that $b \in M_\infty[*]$. We may assume that $\delta(\T) = \delta^N$, since otherwise it is clear that $b \in M_\infty[*]$. 
    
    We claim that $\{ b\}$ is $\Sigma_1$-definable over $M_\infty [*]$. Let $B$ be the set of cofinal branches $c$ of $\T$ such that $N \vert \delta^N \subseteq \wfc(\M^\T_c)$. Clearly, $b \in B$. If $B = \{ b \}$, we are done. So suppose for the sake of contradiction that there is $c \in B$ such that $c \neq b$ of $\T$ such that $\delta^N \subseteq \wfc(\M^\T_c)$ and $c \neq b$. Note that there is a unique tree $\U$ on $N$ such that there is a cofinal wellfounded branch $b^\prime$ which is according to $\Sigma_0$ such that if $i^\U_{b^\prime} \circ i^\T_b = i^{\F^*}_{M_\infty \infty}$ and so $(i^\U_{b^\prime} \circ i^\T_b) \restriction \delta_\infty = * \restriction \delta_\infty$. If there is no cofinal branch $c^\prime$ of $\U$ such that $(i^\U_{c^\prime} \circ i^\T_c) \restriction \delta_\infty = * \restriction \delta_\infty$, we are done. So, suppose that there is such $c^\prime$. Note that $c^\prime \neq b^\prime$. This means that $\ran(i^\U_{b^\prime}) \cap \ran(i^\U_{c^\prime}) \supseteq * \restriction \delta_\infty$. However, $* \restriction \delta_\infty$ is cofinal in $\delta^\infty_\infty$, which is a contradiction to Lemma 2.6 of \cite{Steel_Woodin_2016}. Thus, $\{b\}$ is $\Sigma_1$-definable over $M_\infty[*]$.
    By the Spector-Gandy Theorem it follows that $b \in M_\infty[*]$.  
\eprf

\section{$\delta_\infty$ is Woodin in $M_\infty[\Sigma_0]$}

In this section, we will show that $\delta_\infty$ remains a Woodin cardinal in $M_\infty[\Sigma_0]$. The proof is an adaption of \cite{schlutzenberg2021local} to our context.

\begin{dfn}
    Let $j := i_{\mad \infty} \colon \mad \to M_\infty$ be the iteration map given by $\Sigma^{\mad}$ and let $\bar{S}_\infty$ be defined over $\mad$ as $S_\infty$ is defined over $N$.
\end{dfn}

\begin{rem}
    Note that for any $N \in \F$, $S_\infty \subseteq \ran(i_{\mad N})$, since $S_\infty$ is an $r\Sigma_{n+1}$-definable class of $N$, where $i_{\mad N}$ denotes the iteration map given by $\Sigma^{\mad}$, and $i_{\mad N}[\bar{S}_\infty] = S_\infty$. Moreover, $j[\bar{S}_\infty] = S^*_\infty$.
\end{rem}

\begin{lem} \label{lem: hulls are the same}
    $\Hull{M_\infty[*]}{n}(\ran(j)) = \Hull{M_\infty[*]}{n}(S_\infty^*)$, where we consider $M_\infty[*]$ with the predicates $\E^{M_\infty}$ and $*$.
\end{lem}

\prf
    Notice that since $\Hull{M_\infty}{n}(X) \subseteq \Hull{M_\infty[*]}{n}(X)$ for all $X$, it suffices to show that $\mad = \Hull{\mad}{n} (\bar{S}_\infty)$. Let $N \in \F$. Note that $\bar{S}_\infty$ is cofinal in $\OR^{\mad}$ and $\theta^{\mad}$. It follows from the proof of Lemma \ref{lem: S infty is generating} that $\mad = \Hull{\mad}{n} (\bar{S}_\infty)$.
\eprf

\setcounter{claimcounter}{0}

\begin{lem} \label{lem: hull and range are the same}
    Let $H := \Hull{M_\infty[*]}{n}(\ran(j))$. Then $H \cap \alpha = \ran(j) \cap \alpha$.
\end{lem}

\prf
    Clearly, $\ran(j) \subset H$. In order to see the other inclusion, we first prove the following claim. 

    \begin{claim}
        $\ran(j)$ is closed under $*$ and $*^{-1}$. 
    \end{claim}
    \prf
        For $s \in \fin{S_\infty}$ let $k_s = k \restriction H_{s^*}^{M_\infty} = i^{\tilde{\D}^\infty}_{(M^\infty,s^*) \infty}$. Note that $k_s$ is definable from $s^*$ and $* \restriction \max(s^*)$ over $M_\infty[*]$, so that $k_s \in \ran(j)$.
        Let $\beta < \alpha$ and let $s \in \fin{S_\infty}$ be such that $\beta \in H^{M_\infty}_{s^*}$. We have $\beta^* = k_s (\beta)$ and $k_s \in \ran(j)$. Thus, $\beta \in \ran(j)$ if and only if $\beta^* \in \ran(j)$.
    \eprf

    Let $\beta \in H \cap \alpha$. We aim to see that $\beta \in \ran(j)$. By the claim, it suffices to see that $\beta^* \in \ran(j)$. By Lemma \ref{lem: hulls are the same}, we can fix $s\in \fin{S_\infty}$ and an $r\Sigma_n$ formula $\varphi$ such that $\beta$ is the unique $\beta^\prime < \alpha$ such that $M_\infty[*] \models \varphi(s^*,\beta^\prime)$. Let $\eta \in S_\infty^*$ be such that $\eta > \max\{\beta,\theta, \max(s^*)\}$ and $\eta$ is sufficiently large so that $M_\infty[*] \vert \eta$ can compute the values of the elements of $s \cup \kappa$ under the $*$-map. Let $s^+ = s^* \cup \{ \eta^* \}$.

    Let $N \in \F$ be such that $N$ is $\{ \beta \} \cup s^+$-stable. Note that since $M_\infty [*]$ is $\Sigma_1$-definable over $\lxg$ from the parameters $\theta$ and $\R^{\lxg}$, we have that $\beta$ is the unique ordinal less than $\alpha$ such that
    \[
        N \models \col(\omega, {< \kappa}) \Vdash ``M_\infty[*] \models \varphi(\beta,s^*)",
    \]
    and thus $\beta$ is $\Sigma_1$-definable over $N$ from the parameter $(s, \theta)$\footnote{Note that we use here that $* \restriction \theta$ as computed in $\lx$ is the same as computed in $\lxg$. This follows by quite similar arguments as in the remark before Lemma \ref{lem: iteration maps and star agree}.} (note that $\theta$ determines $\kappa$ and $\dot{\R}$). Note that then $\beta$ is also definable over $N \vert \eta$ from the parameter $s$ and a formula $\psi$, so that $\beta \in H^N_{s \cup \{ \eta \}}$.
    Since $N$ is $\{ \beta \}$-stable, $\beta^* = i^{\tilde{\D}}_{(N,s^+)\infty} (\beta)$. But this means that 
    \[
        H^{M_\infty}_{s^+} \models ``\beta^* \text{ is the unique } \gamma \text{ such that }\psi(\gamma,s^*)".
    \]
    Since $\ran(j) \prec_{r\Sigma_{n+1}} M_\infty$ and $s^+ \in \ran(j)$ it follows that $\beta^* \in \ran(j)$.
\eprf

\begin{dfn}
    Let $N$ be a non-dropping $\Sigma^{\mad}$-iterate of $\mad$ such that $N$ is a pre-$\mad$-like $x$-weasel and $N^- \in \lx \vert \kappa$. Let $\Lambda_N$ be the restriction of $\Sigma^{\mad}$ to non-dropping stacks on $N$ such that the last model is pre-$\mad$-like $x$-weasel and in $N \vert \mu^N$. 
    Let $\Lambda = \Lambda_{\mad} = \Sigma_0$.
\end{dfn}

\begin{lem} \label{cor: the correct collapse}
    The transitive collapse of $\Hull{M_\infty[\Sigma_0]}{n}(\ran(j))$ is $\mad [\Lambda]$.
    Moreover, $V_{\delta^{\mad}}^{\mad} = V_{\delta^{\mad}}^{\mad[\Lambda]}$ and $\delta_{\mad}$ is regular in both models.
\end{lem}

\prf
    Since $M_\infty [\Sigma_0] \subseteq \lxg$ and $\theta = \delta_\infty$ is regular in $\lxg$, it follows that $\delta_\infty$ is regular in $M_\infty [\Sigma_0]$. But then, since $V_{\delta_\infty}^{M_\infty} = V_{\delta_\infty}^{M_\infty[\Sigma_0]}$ the claim follows easily.
\eprf

Let 
\[
    \pi_0 \colon \mad \to \mad \subseteq \mad[\Lambda]
\]
be the identity map. Let $\Psi$ be the putative iteration strategy for $\mad[\Lambda]$ given by ``inverse copying" via $\pi_0$. This makes sense by Lemma \ref{cor: the correct collapse}.

For a putative iteration tree $\U$ on $\mad [\Lambda]$ via $\Psi$ and $\beta < \lh(\U)$ such that $[0,\beta]_\U$ does not drop, we write $\M^\U_{\beta} = \N^\U_\beta [\Lambda_\beta^\U]$, i.e.~$\N^\U_\beta$ is putatively $\mad$-like and $\Lambda^\U_\beta$ is some putative iteration strategy. Letting $\T$ be the inverse copy on $\mad$, let
\[
    \pi_\beta \colon \M^\T_\beta \to \N^\U_\beta
\]
be the copy map.

\begin{lem} \label{lem: correct iterations with strategies}
    Let $\T, \U, \beta,$ and $\pi_\beta$ be as above. Then $\N^\U_\beta = \M^\T_\beta$, $\pi_\beta = \id$, and $\Lambda_\beta^\U = \Lambda_{\M^\T_\beta}$.
    Thus, $\M_\beta^\U = \M^\T_\beta [\Lambda_{\M^\T_\beta}]$, and $\M_\beta^\U$ is wellfounded.
\end{lem}

\prf
    Let $R = \M^\T_\beta$. Note that $R$ is a $\Sigma^{\mad}$-iterate of $\mad$. Since $M_\infty^\infty$ is a $\Sigma_2$ definable class of $M_\infty$, we may pull this definition back via $j$, so that every iterate $P$ of $\mad$ and $h$ that is $(P,\col(\omega, {< \mu^P}))$-generic has its own version of $M_\infty$ which we denote by $(M_\infty)^{P[h]}$.

    Let us write for the remaining proof $M_\infty = (M_\infty)^{R[h]}$ for some fixed $h$ which is $(R,\col(\omega, {< \mu^R}))$-generic and $M_\infty [\Sigma] = (M_\infty[*])^{R[h]}$.\footnote{Note that by the same argument as in the proof of Lemma \ref{strategy mouse}, $(M_\infty[*])^{R[h]} = (M_\infty)^{R[h]} [\Lambda_{(M_\infty)^{R[h]}}]$.} It is easy to see by the earlier proofs that $M_\infty$ is a $\Sigma^{\mad}$-iterate of $R$ and therefore a $\Sigma^{\mad}$-iterate of $\mad$. We have the following commuting diagram.
    \[
    \begin{tikzcd} 
        \mad \arrow[r, "i_{\mad R}"] \arrow[rdd, "i_{\mad M_\infty}"'] & R \arrow[dd, "i_{R M_\infty}"] \\
        &  \\
        & M_\infty
    \end{tikzcd}
    \]
    Let 
    \begin{align*}
        \begin{split}
            H^\infty_{\ad} &= \Hull{M_\infty [\Sigma]}{n} (\ran(i_{\mad {M_\infty}})), \\
            H^\infty_R &= \Hull{M_\infty [\Sigma]}{n} (\ran(i_{R M_\infty})), \text{ and} \\
            H^{R[\Lambda_R]}_{\ad} &= \Hull{R[\Lambda_R]}{n} (\ran(i_{\mad R})).
        \end{split}
    \end{align*}
    By the same argument as in the proof of Lemma \ref{lem: hull and range are the same}, we have 
    \begin{align*}
        \begin{split}
            H^\infty_{\ad} \cap \OR &= \ran(\pi_{\mad M_\infty}) \cap \OR \text{, and} \\
            H^\infty_R \cap \OR &= \ran(\pi_{R M_\infty}) \cap \OR,
        \end{split}
    \end{align*}
    so that
    \begin{align*}
        \begin{split}
            H^\infty_{\ad} \cap M_\infty &= \ran(i_{\mad M_\infty}), \\
            H^\infty_R \cap M_\infty &= \ran(i_{R M_\infty}),
        \end{split}
    \end{align*}
    and the transitive collapse of $H^\infty_{\ad}$ is $\mad[\Lambda]$, and the transitive collapse of $H^\infty_R$ is $R[\Lambda_R]$. Note that both their strategies lift to $M_\infty [\Sigma]$. By the commutativity of the maps, it follows that the transitive collapse of $H^{R[\Lambda_R]}_{\ad}$ is $\mad [\Lambda]$.

    Let $i^+_{\mad M_\infty} \colon \mad [\Lambda] \to M_\infty [\Sigma]$ be the inverse of the transitive collapse map. Note that $i_{\mad M_\infty} \subseteq i^+_{\mad M_\infty}$. Likewise, define $i^+_{R M_\infty}$ and $i^+_{\mad R}$. Let $E_{\ad \infty}$ be the $(\delta^{\mad},\delta^{M_\infty})$-extender derived from $\pi_{\mad M_\infty}$, $E_{R \infty}$ be the $(\delta^{R},\delta^{M_\infty})$-extender derived from $\pi_{R M_\infty}$, and $E_{\ad R}$ be the $(\delta^{\mad},\delta^{R})$-extender derived from $\pi_{\mad R}$. Then
    \[
        E_{\ad \infty} = E_{R \infty} \circ E_{\ad R}
    \]
    and
    \begin{align*}
        \begin{split}
            R &= \fu{\mad}{E_{\ad R}}n, \\
            M_\infty &= \fu{\mad}{E_{\ad \infty}}n = \fu{R}{E_{R \infty}}n,
        \end{split}
    \end{align*}
    and $i_{\mad R}$, $i_{\mad M_\infty}$, and $i_{R M_\infty}$ are the ultrapower maps. Note that the extender $E_{\ad \infty}$ can be applied to $\mad [\Lambda]$, since $V_{\delta^{\mad}}^{\mad} = V_{\delta^{\mad}}^{\mad [\Lambda]}$. Likewise for $E_{\ad R}$ and $E_{R \infty}$. The factor map $\rho \colon \fu{\mad[\Lambda]}{E_{\ad \infty}}n \to M_\infty [\Sigma]$ must be the identity, since $i_{\mad M_\infty} \subseteq i^+_{\mad M_\infty}$. The same holds for the other factor maps, so that 
        \begin{align*}
            \begin{split}
                M_\infty [\Sigma] &= \fu{\mad [\Lambda]}{E_{\ad \infty}}n \\
                M_\infty [\Sigma] &= \fu{R [\Lambda_R]}{E_{R \infty}}n, \\
                R [\Lambda_R] &= \fu{\mad [\Lambda]}{E_{\ad R}}n,
            \end{split}
        \end{align*}

    and $i^+_{\ad \infty}$, $i^+_{R \infty}$, and $i^+_{\ad R}$ are the ultrapower maps. However, $E_{\ad R}$ is the branch extender of $[0,\beta]_\U$ in $\U$, so that 
    \[
        \N^\U_\beta [\Lambda^\U_\beta] = \M^\U_\beta = \Ult{\mad[\Lambda]}{E_{\ad R}} = R[\Lambda_R]
    \]
    and $i^\U_{0\beta} = i^+_{\mad R}$ is the ultrapower map. Thus, $\N^\U_\beta = R = \M_\beta^\T$ and $\Lambda_\beta^\U = \Lambda_R = \Lambda_{\M^\T_\beta}$.

    Note that since $\pi_\beta \colon R \to \N^\U_\beta$ is the inverse copy map, we have $\pi_\beta \restriction \delta^R = \id$ and $i^\U_{0\beta} = \pi_\beta \circ i^\T_{0\beta}$. However, since
    \[
        i^\T_{0\beta} = i_{\mad R} \subseteq i^+_{\mad R} = i^\U_{0\beta},
    \]
    it follows that $\pi_\beta \restriction \ran(i^\T_{0\beta}) = \id$. But then
    \[
        \delta^{\N^\U_\beta} \cup \ran({i^\T_{0\beta}} \restriction \bar{S}_\infty) \subseteq \ran(\pi_\beta),
    \]
    so that by the argument of Lemma \ref{lem: S infty is generating}, $\ran(\pi_\beta) = R$, and so $\pi_\beta = \id$.
\eprf

\begin{lem}
    $\delta_\infty$ is Woodin in $M_\infty[\Sigma_0]= M_\infty[*]$.
\end{lem}

\prf
    Suppose for the sake of contradiction that $\delta_\infty$ is not Woodin in $M_\infty [\Sigma_0]$. We arrange $M_\infty [\Sigma_0]$ as a fine structural strategy premouse. Let $Q \lhd M_\infty [\Sigma_0]$ be the Q-structure witnessing that $\delta_\infty$ is not Woodin. Since $M_\infty$ is a normal iterate of $\mad$, there is a limit length $n$-maximal tree $\T$ on $({\mad})^-$ such that $\M(\T) = M_\infty \vert \delta_\infty$. Since $\T$ is definable from $({\mad})^-$ and $M_\infty \vert \delta_\infty$, $\T$ is definable over $\lxg \vert \theta$ and therefore $\T \in \lxg$. Let $b = \Sigma(\T)$ be the unique cofinal wellfounded branch of $\T$. Let $\U$ be the copy of $\T$ on $\mad [\Lambda]$ via the copy map $\pi_0 \colon \mad \to \mad [\Lambda]$. By Lemma \ref{lem: correct iterations with strategies}, $\U$ is the iteration tree that leads from $\mad [\Lambda]$ to $M_\infty [\Sigma_0]$. Note that we assume $\mad [\Lambda]$ to be arranged as a fine structural strategy premouse in this context. We have $i_b^\U (\bar{Q}) = Q$, where $\bar{Q} \lhd \mad [\Lambda]$ is the Q-structure witnessing that $\delta^{\mad}$ is not Woodin in $\mad [\Lambda]$.

    Let $h$ be sufficiently $(\lxg, \col(\omega,\theta))$-generic, so that $\lxg [h]$ is a model of $\Sigma_n$-KP.
    Over $\lxg \vert (\theta + \omega \cdot \omega) [h]$ we might define a tree searching for a pair $(R,c)$ such that
    \begin{enumerate}
        \item $R$ is a strategy premouse extending $\mad \vert \delta^{\mad}$,
        \item $\delta^{\mad}$ is inaccessible in $\mathcal{J} (R)$,
        \item $c$ is a non-dropping $\T$-cofinal branch, and
        \item considering $\T$ as a tree on $\mathcal{J} (R)$, then $i_c^\T (R) = Q$.
    \end{enumerate}
    Note that the pair $(\bar{Q},b)$ witnesses that there is a branch through this tree.

    We claim that $(\bar{Q},b)$ is the unique such witness: For, suppose that $(R,c)$ and $(R^\prime,c^\prime)$ are given by branches of this tree. Then we may consider $\T$ as a tree $\U$ on $\mathcal{J}(R)$ and $\T$ as a tree $\U^\prime$ on $\mathcal{J}(R^\prime)$. Since $\mathcal{J}(R)$ and $\mathcal{J}(R^\prime)$ agree below $\delta^{\mad}$, $\U,\U^\prime$ are based below $\delta^{\mad}$, and $\delta^{\mad}$ is inaccessible in $\mathcal{J}(R)$ and $\mathcal{J}(R^\prime)$, we have enough agreement between the models of $\U$ and $\U^\prime$ in order to run the proof of the Zipper Lemma, even though the trees are not based on the same model. This means that if $(R,c) \neq (R^\prime,c^\prime)$, then $\mad \vert \delta^{\mad} \models \exists \delta (``\delta \text{ is Woodin}")$, which is a contradiction!

    This shows that $(\bar{Q},b)$ is a $\Delta_1$ definable real over $\lxg [h]$. But then by the Spector-Gandy theorem, $(\bar{Q},b) \in \lxg [h]$. Since $\col(\omega, \theta)$ is a homogeneous forcing, it follows that $(\bar{Q},b) \in \lxg$. But $i_b^\T$ is continuous at $\delta^{\mad}$, so that the cofinality of $\delta_\infty = \theta$ is countable in $\lxg$, a contradiction!
\eprf

\begin{cor}
    $\delta_\infty$ is Woodin in $M_\infty[\Lambda]$.
\end{cor}

\section{$M_\infty[*]$ is a ground of $\lx$} \label{end of the finite cases} \label{section: ground}

We show that $M_\infty[*]$ is a ground of $\lx$. The argument we give is closely related to the argument for $L[x]$ in \cite{schlutzenberg2021local}, which is due to Schindler.

\begin{dfn}
    Let $\mathcal{L}$ be the infinitary Boolean language, given by starting with a collection $\{ v_n \}_{n < \omega} \in M_\infty [*] \vert \delta_\infty$ of propositional variables, and closing under negation and arbitrarily set-sized disjunctions in $M_\infty [*] \vert \delta_\infty$, so that $\mathcal{L}$ is a definable class of $M_\infty [x] \vert \delta_\infty$ and $\mathcal{L} \in L_1 (M_\infty [*] \vert \delta_\infty)$. Let $\mathbb{C}$ be the subalgebra of $\mathbb{B} = \mathbb{B}_\omega^{M_\infty}$ such that 
    \[
        \mathbb{C} = \{ \Vert k(\varphi) \Vert_{\mathbb{B}} \mid \varphi \in \mathcal{L} \},
    \]
    where $\Vert \psi \Vert_{\mathbb{B}}$ denotes the Boolean value of $\psi$ with respect to $\mathbb{B}$, and we interpret $\langle v_n \rangle_{n<\omega}$ as the generic real for $\mathbb{B}$.
\end{dfn}

Since $V_{\delta_\infty}^{M_\infty} = V_{\delta_\infty}^{M_\infty [*]}$ and $\delta_\infty$ is a Woodin cardinal in $M_\infty [*]$, $\mathbb{B}$ is a Boolean algebra with the $\delta_\infty$-c.c.~in $M_\infty [*]$. Thus, $\mathbb{C}$ is well-defined. We have $\mathbb{C} \in \lxg$. Note that for all $\varphi \in \mathcal{L}$, $\Vert k(\varphi) \Vert_{\mathbb{B}}^{M_\infty} = \Vert k(\varphi) \Vert_{\mathbb{B}}^{M_\infty [*]}$, so $\mathbb{C} \subseteq \delta_\infty$.

\begin{lem}
    $x$ is $(M_\infty [*],\mathbb{C})$-generic in the sense that 
    \[
        G_x = \{ \Vert k(\varphi) \Vert_{\mathbb{B}} \mid \varphi \in \mathcal{L} \land x \models \varphi \}
    \]
    is $(M_\infty [*],\mathbb{C})$-generic, and $M_\infty [*] [G_x] = \lxg$.
\end{lem}

\prf
    It is easy to see that $G_x$ is a filter. In order to see genericity, let $\langle \varphi_\alpha \rangle_{\alpha < \lambda} \in M_\infty [*]$ be such that $\langle \Vert k(\varphi_\alpha) \Vert_{\mathbb{B}} \rangle_{\alpha < \lambda}$ is a maximal antichain of $\mathbb{C}$. Since $\delta_\infty$ is Woodin in $M_\infty [*]$, $\lambda < \delta_\infty$. Let $\psi = \bigvee_{\alpha < \lambda} \varphi_{\alpha}$, and note that $\varphi \in \mathcal{L}$, since $\delta_\infty$ is inaccessible, so that $\langle \varphi_\alpha \rangle_{\alpha < \lambda}$ cannot be cofinal in $\delta_\infty$, and $\Vert \psi \Vert_{\mathbb{B}} = \bigvee_{\alpha < \lambda} \Vert \varphi_\alpha \Vert_{\mathbb{B}}$. It suffices to see that $x \models \psi$. Suppose for the sake of contradiction that $x \models \neg \psi$. Let $\alpha_\psi = \rk_{<^{M_\infty}} (\psi)$ be the rank of $\psi$ in the order of constructability of $M_\infty$ and let $N \in \F$ be  $\{ \alpha_\psi \}$-stable. It follows that $N \models \Vert \neg \psi \Vert_{\mathbb{B}^N} \neq 0$. Thus, $M_\infty \models \Vert k( \neg \psi) \Vert_{\mathbb{B}} \neq 0$ and so $\Vert k( \neg \psi) \Vert_{\mathbb{B}} \in \mathbb{C}$ is a nonzero condition. But then it is easy to see that $\Vert k(\neg \psi) \Vert_{\mathbb{B}} \bot \Vert k(\varphi_\alpha) \Vert_{\mathbb{B}}$ for all $\alpha < \lambda$. This contradicts the maximality of $\langle \varphi_\alpha \rangle_{\alpha < \lambda}$!
\eprf

\section{$\Sigma_n$-HOD} \label{hod}

\setcounter{claimcounter}{0}

Finally, we aim to characterize $\Sigma_n$-HOD. Note that just as in the classical analysis of $\hod^{L[x,G]}$, we can only characterize the $\Sigma_n$-HOD of $\lxg$ and not of $\lx$. This is for the same reasons as outlined on page 267 of \cite{Steel_Woodin_2016} in the $L[x]$ case. 
Let us fix $G$ which is $(\lx,\col(\omega,{<\kappa}))$-generic and let us write $\Sigma_n \text{-} \hod_{\{\R\}}$ for $\Sigma_n \text{-} \hod_{\{\R^{\lxg}\}}^{\lxg}$ and $\Sigma_n \text{-} \od_{\{\R\}}$ for $\Sigma_n \text{-} \od_{\{\R^{\lxg}\}}^{\lxg}$.
We will work in this last section only the straightforward adaption of the contents of Section \ref{direct limit systems} through Section \ref{M infinitys version} to the context of $\lxg$, i.e.~$\tilde{\D}$ and its related models denote the direct limit systems and their limits computed in $\lxg$.

\begin{lem} \label{lem: sigma_n hod is almost admissible}
    $\Sigma_n \text{-} \hod_{\{\R\}} \models \Sigma_n \text{-}\kp \setminus \{ \Sigma_n \text{-Collection} \} + AC$
\end{lem}

\prf
    Note that for $b \in a \in \Sigma_n \text{-} \hod_{\{\R\}}$, $\tc(\{ b \}) \subseteq \tc (\{ a \})$. Thus, $\Sigma_n \text{-} \hod_{\{\R\}}$ is transitive, and the Axiom of Extensionality and the Axiom of Foundation hold trivially. Clearly, $\emptyset \in \Sigma_n \text{-} \hod_{\{\R\}}$. Moreover, it is easy to see that the Axiom of Pairing and the Axiom of Union hold in $\Sigma_n \text{-} \hod_{\{\R\}}$.
    
    Let us verify the Axiom of $\Sigma_{n-1}$-Aussonderung. Let $a \in \Sigma_n \text{-} \hod_{\{\R\}}$, $\varphi \in \la$  be $\Sigma_{n-1}$, and $p \in \Sigma_n \text{-} \hod_{\{\R\}}$. We aim to see that 
    \[
        b:= \{ u \in a : \Sigma_n \text{-} \hod_{\{\R\}} \models \varphi(u,p)\} \in \Sigma_n \text{-} \hod_{\{\R\}}.
    \]
    Note that since $\tc(b) \subseteq \tc(a) \subseteq \Sigma_n \text{-} \hod_{\{\R\}}$ it suffices to see that $b \in \Sigma_n \text{-} \od_{\{\R\}}$. However, since $a \in \Sigma_n \text{-} \od_{\{\R\}}$, there is a $\Sigma_n$ formula $\psi \in \mathcal{L}_{\dot{\in}}$ and $q \in \fin{\alpha}$ that define $a$ via $\R^{\lxg}$ and likewise $\psi^\prime$ and $q^\prime \in \fin{\alpha}$ that define $\{p\}$ via $\R^{\lxg}$. Thus, $\exists z (\psi(u,q,\R^{\lxg}) \land \varphi(u,z,\R^{\lxg}) \land \psi^\prime(z,q^\prime,\R^{\lxg}))$ defines $b$.

    It remains to see that the Axiom of Choice holds in $\Sigma_n \text{-} \hod_{\{\R\}}$. Let $S := S^{\lxg}_{n-1} = \{ \beta < \alpha : \lxg \vert \beta \prec_{\Sigma_{n-1}} \lxg \}$. For $\beta < \alpha$, let $\Sigma_n \text{-} \od_{\beta,\{\R\}}$ be the class of all $y \in \lxg \vert \beta$ that are ordinal definable over $\lxg \vert \beta$ via a $\Sigma_n$ formula in the language $\mathcal{L}_{\dot{\in}}$ and the parameter $\R^{\lxg}$.

    \begin{claim}
        $\Sigma_n \text{-}\od_{\{\R\}}  = \bigcup_{\beta < \alpha}  \Sigma_n \text{-} \od_{\beta,\{\R\}}$.
    \end{claim}
    \prf
    It is clear that $\Sigma_n \text{-}\od_{\{\R\}}  \supseteq \bigcup_{\beta < \alpha}  \Sigma_n \text{-} \od_{\beta,\{\R\}}$ holds. In order to show that $\Sigma_n \text{-}\od_{\{\R\}}  \subseteq \bigcup_{\beta < \alpha}  \Sigma_n \text{-} \od_{\beta,\{\R\}}$, let $A \in \Sigma_n \text{-}\od_{\{\R\}} $. Then there is a $\Sigma_n$ formula $\varphi \equiv \exists y \psi$, where $\psi$ is $\Pi_{n-1}$ and $\alpha_1,..., \alpha_m < \alpha$ such that
    \[
        z \in A \iff \lxg \models \exists y \psi (z,y,\alpha_1,...,\alpha_m,\R^{\lxg}).
    \]
    Note that this defines a $\Sigma_n$-definable function $f$ with domain $A$. By $\Sigma_ n$-Collection, $f \in \lxg$. Let $\beta \in S$ be such that $f \in \lxg \vert \beta$. Then $A \in \Sigma_n \text{-}\od_{\beta,\{\R\}}$.
    \eprf

    Note that for all $\beta < \alpha$, $\Sigma_n \text{-} \od_{\beta,\{\R\}} \in \Sigma_n \text{-} \od_{\{\R\}}$, since $\Sigma_n \text{-} \od_{\beta,\{\R\}}$ is $\Sigma_n$-definable over $\lxg$ via the parameters $\beta$ and $\R^{\lxg}$.
    For $A \in \Sigma_n \text{-} \od_{\{\R\}}$, let $\alpha_A$ be the least $\beta < \alpha$ such that $A \in \Sigma_n \text{-} \od_{\beta,\{\R\}}$. Note that $\alpha_A$ is $\Sigma_n$-definable over $\lxg$ from the parameter $A$ and $\R^{\lxg}$, since $\alpha_A$ may be defined as the unique $\gamma < \alpha$ such that $A \in \Sigma_n \text{-} \od_{\gamma,\{\R\}}$ and for all $\beta < \gamma$ either $A \not \subset \lxg \vert \beta$ or for all $\Sigma_n$ formulas $\varphi$ and for all $a \in \fin{\beta}$ there exists $z \in A$ such that $\lxg \vert \beta \not \models \varphi(z,a,\R^{\lxg})$.

    Define the order $\leq^{***}$ as follows. For $a,b \in \fin{\alpha}$, let $a \leq^* b$, if $a = b$ or $\max(a \Delta b) \in b$. Note that $\leq^*$ is a well-order on $\fin{\alpha}$. Moreover, it is easy to see that $\leq^* \cap \fin{\beta}$ is $\Sigma_0$-definable over $\lxg \vert \beta$ without parameters for all $\beta \leq \alpha$.

    For $\Sigma_n$ formulas $\varphi(v_0,v_1,...,v_m)$ and $\psi(v_0,v_1,...,v_k)$, and $a \in [\alpha]^m$ and $b \in [\alpha]^k$, let $(\varphi,a) \leq^{**} (\psi,b)$, if the Gödel number of $\varphi$ is less than the Gödel-number of $\psi$, or else $\varphi = \psi$ and $a \leq^* b$. Note that $\leq^{**} \cap  ~\omega \times \fin{\beta}$ is a well-order and is $\Sigma_ 1$-definable over $\lxg \vert \beta$ without parameters for all $\beta \leq \alpha$.

    Let $\beta < \alpha$. For $A \in \Sigma_n \text{-} \od_{\beta,\{\R\}}$ let $(\varphi_A, a_A) \in \omega \times \fin{\beta}$ be the $\leq^{**}$-least pair $(\varphi,a) \in \omega \times \fin{\beta}$ such that for all $z \in \lxg \vert \beta$
    \[
        z \in A \iff \lxg \vert \beta \models \varphi(z,a,\R^{\lxg}).
    \]
    Let $\leq_\beta$ be the order induced by $\leq^{**}$ on $\Sigma_n \text{-}\od_{\beta,\{\R\}}$, i.e.~for $A,B \in  \Sigma_n \text{-} \od_{\beta,\{\R\}}$, $A \leq_\beta B$ iff $(\varphi_A, a_A) \leq (\varphi_B, a_B)$.
    Note that $\Sigma_n \text{-} \od_{\beta,\{\R\}} \in \Sigma_n \text{-} \od_{\{\R\}}$ and $\leq_\beta \in \Sigma_n \text{-} \od_{\{\R\}}$. Now define $\leq^{***}$ such that if $A, B \in \Sigma_n \text{-} \od_{\{\R\}}$, then $A \leq^{***} B$ if $\alpha_A < \alpha_B$, or else $\alpha_A = \alpha_B$ and $A \leq_{\alpha_A} B$.

    Note that for all $\beta < \alpha$ the restriction of $\leq^{***}$ to $\Sigma_n \text{-} \od_{\beta,\{\R\}}$ is in $\Sigma_n \text{-} \od_{\{\R\}}$. For $\beta < \alpha$, let $\Sigma_n \text{-} \hod_{\beta,\{\R\}}$ be the class of all $y \in \lxg \vert \beta$ such that $\tc(\{y\}) \subset \Sigma_n \text{-}\od_{\beta,\{\R\}}$. Next, we aim to see that for all $\beta < \alpha$ the restriction of $\leq^{***}$ to $\Sigma_n \text{-} \hod_{\beta,\{\R\}}$ is in $\Sigma_n \text{-} \hod_{\{\R\}}$. 
    This follows immediately from the following claim.

    \begin{claim}
        $\Sigma_n \text{-} \hod_{\{\R\}} = \bigcup_{\beta < \alpha} \Sigma_n \text{-} \hod_{\beta,\{\R\}}$
    \end{claim}
    \prf
        By the previous claim, $\Sigma_n \text{-} \hod_{\{\R\}} \supseteq \bigcup_{\beta \in S} \Sigma_n \text{-} \hod_{\beta,\{\R\}}$.
        Let $A \in \Sigma_n \text{-} \hod_{\{\R\}}$. Since $\tc(\{A\}) \subseteq \Sigma_n \text{-}\od_{\{\R\}}$ and for all $B \in \tc(\{A\})$, $\alpha_B$ is $\Sigma_n$-definable over $\lxg$, the function $f$ with domain $\tc(\{A\})$ such that $f(B) = \alpha_B$ is $\Sigma_n$-definable over $\lxg$. By $\Sigma_n$-Collection $f \in \lxg$. Then $A \in \Sigma_n \text{-} \hod_{\beta,\{\R\}}$, where $\beta \in S \setminus \sup(\ran(f))$.
    \eprf

    Thus, the Axiom of Choice holds in $\Sigma_n \text{-} \hod_{\{\R\}}$.
\eprf

\begin{lem}
    $V_{\delta_\infty}^{\Sigma_n \text{-} \hod_{\{\R\}}} = V_{\delta_\infty}^{M_\infty}$.
\end{lem}

\prf
Since $M_\infty$ is $\Sigma_1$-definable over $\lxg$ from ordinal parameters and $\R^{\lxg}$, we have $V_{\delta_\infty}^{\Sigma_n \text{-} \hod_{\{\R\}}} \supseteq V_{\delta_\infty}^{M_\infty}$.
For the other inclusion, let $A \in V_{\delta_\infty}^{\Sigma_n \text{-} \hod_{\{\R\}}}$. Note that we may code $A$ as a set of ordinals using the order $\leq^{***}$ defined in the proof of Lemma \ref{lem: sigma_n hod is almost admissible}.
Let $\beta < \delta_\infty$ be such that $A \subseteq \beta$, and let $\varphi$ and $\vec{\gamma} \in \fin{\alpha}$ define $A$ over $\lxg$ from the parameter $\R^{\lxg}$.

Since $\delta_\infty = \sup (\delta_\infty \cap \Hull{M_\infty}{n} (S_\infty^*))$, there is for every $\beta < \delta_\infty$ some $s \in \fin{S_\infty^*}$ such that $\beta < \gamma_s^{M_\infty}$. Since $k \restriction \gamma_s^{M_\infty} = i^{\tilde{\D}^\infty}_{(M_\infty,s)\infty} \restriction \gamma_s^{M_\infty}$, we have that for every $\beta < \delta_\infty$, $k \restriction \beta \in M_\infty$.
We have 
\begin{align*}
    \begin{split}
    &\xi \in A \\
    \iff & \lxg \models \varphi (\xi,\gamma_1,...,\gamma_n,\R) \\
    \iff & M_\infty \models \col(\omega, {< \kappa})
    \Vdash \varphi(\xi^*, \gamma_1^*,...,\gamma_n^*,\R) \\
    \iff & M_\infty \models \col(\omega, {< \kappa})
    \Vdash \varphi(k(\xi), \gamma_1^*,...,\gamma_n^*,\R)
    \end{split}
\end{align*}

Since $k \restriction \beta \in M_\infty$ and $\gamma_1^*,...,\gamma_n^* < \alpha$, $A \in M_\infty$.
\eprf

\begin{lem}
    $\Sigma_n \text{-} \hod_{\{\R\}} = M_\infty [*]$.
\end{lem}

\prf
    Note that $\Sigma_n \text{-} \hod_{\{\R\}} \supseteq L_\alpha [M_\infty, *\restriction \theta]$, since $L_\alpha [M_\infty, *\restriction \theta]$ is $\Sigma_1$-definable over $\lxg$ from ordinal parameters and $\R^{\lxg}$.

    In order to see that $\Sigma_n \text{-} \hod_{\{\R\}} \subseteq L_\alpha [M_\infty, *\restriction \theta]$, let $A \in \Sigma_n \text{-}\hod_{\{\R\}}$. By Lemma \ref{lem: sigma_n hod is almost admissible}, we may assume that $A \subseteq \alpha$. Moreover, since $A \in \lxg$, $A$ is bounded in $\alpha$. Let $\varphi$ be a $\Sigma_n$ formula in the language $\mathcal{L}_{\dot{\in}}$ and let $\alpha_1,...,\alpha_m < \alpha$ be such that for all $\xi \in \lxg$
    \[
        \xi \in A \iff \lxg \models \varphi (\xi,\alpha_1,...,\alpha_m,\R^{\lxg}).
    \]
    Then
    \begin{align*}
        \begin{split}
        &\xi \in A \\
        \iff &\lxg \models \varphi (\xi,\alpha_1,...,\alpha_m,\R^{\lxg}) \\
        \iff &M_\infty \models \col(\omega, {< \kappa})
        \Vdash \varphi(\xi^*, \alpha_1^*,...,\alpha_m^*,\R^{\lxg})
        \end{split}
    \end{align*}

    Let $\beta < \alpha$ be such that $A$ is definable over $\lxg \vert \beta$ from ordinal parameters and $\R^{\lxg}$. Note that the proof of Lemma \ref{lem: definability of the star map} shows that $* \restriction \gamma \in M_\infty[* \restriction \theta]$. It follows that $A \in M_\infty[*]$.
\eprf

\begin{cor} \label{cor: characterization of sigma_n hod}
    $M_\infty [*] = \Sigma_n\text{-HOD}_{\{\R\}} = L_\alpha (A)$ for some $A \in \mathcal{P} (\alpha) \cap \lxg$.
\end{cor}

\begin{lem}
    $M_\infty [*] \models \Sigma_n\text{-}\kp + \mathrm{AC}$ and $V_{\theta}^{M_\infty [*]} = V_{\theta}^{M_\infty}$.
\end{lem}

\prf
    By Lemma \ref{lem: sigma_n hod is almost admissible}, it suffices to see that $\Sigma_n$-Collection holds in $\Sigma_n$-$\hod_{\{\R\}}$. But this follows immediately from Corollary \ref{cor: characterization of sigma_n hod} and the fact that $\lxg \models \Sigma_n\text{-}\kp$.
\eprf

\section*{Acknowledgments}

The first author would like to thank the organizers of the workshop ``Determinacy, Inner Models and Forcing Axioms'' for giving him the opportunity to present parts of this paper.

\section*{Funding}
The first author was funded by the Deutsche Forschungsgemeinschaft (DFG, German Research Foundation) - project number 445387776. In editing this paper the first author was funded by the Austrian Science Fund (FWF) [10.55776/Y1498]. The second author was funded by the Austrian Science Fund (FWF) [10.55776/Y1498]. For open access purposes, the authors have applied a CC BY public copy-right license to any author accepted manuscript version arising from this submission.

\printbibliography

\end{document}